\documentclass[12pt]{article}
\usepackage{a4wide}
\usepackage{amsmath, amsthm, amsfonts, amssymb, bbm, bm}
\usepackage{graphics}
\usepackage{epstopdf}
\usepackage{xypic}
\usepackage[all]{xy}
\usepackage[english]{babel}
\usepackage[font=small,format=plain,labelfont=bf,up]{caption}
\usepackage{enumitem}
\usepackage{upgreek}
\usepackage{accents}
\usepackage{scrextend}
\usepackage{tcolorbox}
\usepackage{silence}
\usepackage{hyperref}       
\hypersetup{       
   pdftitle={Topological field theory on r-spin surfaces and the Arf invariant - I. Runkel and L. Szegedy},
   colorlinks=true,        
   pdfstartview=FitH,      
   allcolors=[rgb]{0,0.5,0},        
	 bookmarksdepth=section, 
   bookmarksopen=false,     
   bookmarksnumbered=true 
}

\makeatletter
\newcommand{\sbullet}{
  \hbox{\fontfamily{lmr}\fontsize{.4\dimexpr(\f@size pt)}{0}\selectfont\textbullet}}
\DeclareRobustCommand{\mathbullet}{\accentset{\sbullet}}
\makeatother
\usepackage{mathtools}

\usepackage{soul}

\usepackage{extpfeil}

\usepackage{diagbox}

\usepackage{tikz-cd}

\usepackage{cellspace}
\DeclarePairedDelimiterX\setc[2]{\{}{\}}{\,#1 \;\delimsize\vert\; #2\,}

\newtheoremstyle{important-thm}
     {3pt}
     {3pt}
     {\slshape}
     {}
     {\bfseries}
     {.}
     {.5em}
     {}

\theoremstyle{plain}

\theoremstyle{important-thm}

\newtheorem{theorem}{Theorem}
\newtheorem{lemma}[theorem]{Lemma}
\newtheorem{proposition}[theorem]{Proposition}
\newtheorem{corollary}[theorem]{Corollary}
\theoremstyle{definition}
\newtheorem{remark}[theorem]{Remark}
\newtheorem{example}[theorem]{Example}
\newtheorem{definition}[theorem]{Definition}
\newtheorem{notation}[theorem]{Notations}

 \numberwithin{equation}{section}
 \numberwithin{theorem}{section}

\DeclareMathOperator{\End}{End}

\DeclareMathOperator{\Rep}{Rep}

\DeclareMathOperator{\id}{id}

\DeclareMathOperator{\Arf}{Arf}

\def\be#1\ee{\begin{equation}#1\end{equation}}
\def\ba#1\ea{\begin{align}#1\end{align}}

\newcommand\Vect{{\mathcal{V}\hspace{-.5pt}ect}}
\newcommand\SVect{{\mathcal{SV}\hspace{-.5pt}ect}}
\newcommand{\Bord}[1]{{\mathcal{B}\hspace{-.5pt}ord_2^{\hspace{1pt} #1}}}

\newcommand{\Mod}[1]{\ (\text{mod}\ #1)}

\newcommand\eps           {\varepsilon}

\newcommand\Cb            {\mathbb{C}}
\newcommand\Ib            {\mathbb{I}}

\newcommand\Rb            {\mathbb{R}}
\newcommand\Zb            {\mathbb{Z}}

\newcommand\Ac            {\mathcal{A}}

\newcommand\Cc            {\mathcal{C}}
\renewcommand\Dc            {\mathcal{D}}
\newcommand\Ec            {\mathcal{E}}
\newcommand\Fc            {\mathcal{F}}

\newcommand\Kc            {\mathcal{K}}
\renewcommand\Lc            {\mathcal{L}}
\newcommand\Mc            {\mathcal{M}}
\renewcommand\Rc            {\mathcal{R}}
\newcommand\Sc            {\mathcal{S}}
\newcommand\Sb            {\mathbb{S}_1}
\renewcommand\Yc            {\mathcal{Y}}
\newcommand\Zc            {\mathcal{Z}}
\newcommand\funZ            {\mathcal{Z}}

\newcommand{\Cl}{C\hspace*{-1pt}\ell}

\newcommand{\void}[1]{}

\newcommand\doi[2]        {\href{http://dx.doi.org/#1}{#2}}

\begin{document}

\thispagestyle{empty}
\def\thefootnote{\fnsymbol{footnote}}
\begin{flushright}
ZMP-HH/18-7
\\
Hamburger Beitr\"age zur Mathematik 726
\end{flushright}
\vskip 3em
\begin{center}\LARGE
Topological field theory on $r$-spin surfaces\\ 
and the Arf invariant
\end{center}

\vskip 2em
\begin{center}
{\large 
Ingo Runkel $^{a}$~~and~~L\'or\'ant Szegedy $^{b}$~\footnote{Emails: {\tt ingo.runkel@uni-hamburg.de}~,~{\tt lorant.szegedy@ist.ac.at}}}
\\[1em]
${}^{a}$ Fachbereich Mathematik, Universit\"at Hamburg\\
Bundesstra\ss e 55, 20146 Hamburg, Germany
\\[1em]
${}^{b}$ Institute of Science and Technology Austria\\
Am Campus 1, 3400 Klosterneuburg, Austria
\end{center}

\vskip 2em

\begin{abstract}
We give a combinatorial model for $r$-spin surfaces with parametrised boundary based on 
  \cite{Novak:2015phd}.
The $r$-spin structure is encoded in terms of $\Zb_r$-valued indices  assigned to the edges of a polygonal decomposition. 
This combinatorial model is designed for our state sum construction of two-dimensional 
topological field theories on $r$-spin surfaces.
We show that an example of such a topological field theory 
computes the Arf-invariant of an $r$-spin surface as introduced in 
  \cite{Randal:2014rs,Geiges:2012rs}. 
This implies in particular that the $r$-spin 
Arf-invariant is constant on orbits of the mapping class group, providing an alternative proof of that fact.
\end{abstract}

\setcounter{footnote}{0}
\def\thefootnote{\arabic{footnote}}

\newpage

{\small

\tableofcontents

}

\newpage

\section{Introduction}

\allowdisplaybreaks

\subsubsection*{\texorpdfstring{$r$}{r}-spin topological field theories}
The rotation group in two dimensions is fundamentally different from the higher dimensional rotation groups. Namely, $SO(n)$ for $n\ge3$ has universal cover $Spin(n)$ which has a finite fibre (namely $\Zb_2$), while the universal cover of $SO(2)$ is $\Rb$ which has an infinite fibre given by $\Zb$. Accordingly, in two dimensions one can speak of $r$-spin structures, for $r \in \Zb_{\ge 0}$, where one considers the connected cover of $SO(2)$ with fibre $\Zb_r$. The special case $r=0$ is the universal cover. We review $r$-spin surfaces in detail in Section~\ref{sec:rspinbord}. Here we only mention that a $1$-spin surface is just an oriented surface, a $2$-spin surface is what is usually referred to as a surface with spin structure, and giving a $0$-spin structure on a surface is equivalent to giving a framing.
We stress that the case $r=0$ is included in all of the following discussion.

We consider $r$-spin surfaces whose boundary components are parametrised by  annuli with $r$-spin structure. The $r$-spin structures on these annuli are in bijection with $\Zb_r$.
The $r$-spin surfaces with parametrised boundary form a symmetric monoidal category $\Bord{r}$, 
whose objects are ``circles with $r$-spin structures'', 
which we describe as finite lists of elements of $\Zb_r$,
and which dictate
the restriction of the $r$-spin structure of a bordism to the in- and outgoing boundary components.

One defines a two-dimensional $r$-spin topological field theory (TFT) to be a symmetric monoidal functor 
\begin{align} \label{eq:intro-rTFT-Z}
\funZ : \Bord{r} \to \mathcal{S} \ ,
\end{align}
for a symmetric monoidal target category $\mathcal{S}$, which we will assume to be additive and idempotent-complete (and have countable direct sums in case $r=0$).

\subsubsection*{Combinatorial model of \texorpdfstring{$r$}{r}-spin surfaces}
In \cite{Novak:2015phd} a combinatorial description of $r$-spin surfaces 
is given based on the choice of a triangulation. For our applications,  
triangulations are cumbersome due to the large number of triangles required even for simple surfaces. 
We give a more convenient combinatorial model based on
decompositions into polygons called PLCW-decompositions \cite{Kirillov:2012pl} (see Section~\ref{sec:PLCW-dec}). 
For example, this allows one to describe a genus $g$-surface with $b$ boundary components 
with $g+b \ge 1$ in terms of a single $(4g+3b)$-gon 
with appropriately identified edges.

The combinatorial representation of an $r$-spin structure on a surface $\Sigma$ is in terms of a marked PLCW-decomposition, that is:
\begin{itemize}
\item
a PLCW decomposition of $\Sigma$ such that each boundary component contains a single edge and a single vertex,
\item
a choice of a marked edge for each face (before identification of the edges),
\item
an orientation of each edge,
\item 
an edge index $s_e \in \Zb_r$ for each edge $e$,
\end{itemize}
and where the edge indices need to satisfy a consistency condition around each vertex, see Section~\ref{sec:combsub}.
To obtain an $r$-spin structure from the above data, one endows each face with its unique 
(up to isomorphism) 
$r$-spin structure and then uses the edge indices to define transition functions between the faces. Finally, one extends the $r$-spin structure to the vertices, which is possible due to the above consistency condition.
Different sets of combinatorial data can describe isomorphic
$r$-spin structures on a given surface, and we give an equivalence relation which precisely encodes that redundancy (Theorem~\ref{thm:rscomb}).

\begin{remark}
The mapping class group of a surface acts on (isomorphism classes of) $r$-spin structures on that surface. Counts of the orbits of this action
  can be found in several places in the literature \cite{Jarvis:1998,Natanzon:2004harf,Randal:2014rs,Geiges:2012rs,Kawazumi:2017mcg,Salter:2017mon}. They differ in the detailed setup and in the restrictions placed on $r,g,b$. 
In \cite{Szegedy:2018phd} we give a different treatment covering all cases of $r,g,b \ge 0$ using our combinatorial model and a state-sum $r$-spin TFT.  
  \label{rem:mcg-orbit-count}
\end{remark}

\subsubsection*{State-sum construction}

We use the above combinatorial model to give a state-sum construction of $r$-spin TFTs.
The input data is a Frobenius algebra $A \in \mathcal{S}$ whose Nakayama automorphism $N$ satisfies $N^r=\id_A$, and whose window element $\mu \circ \Delta \circ \eta : \Ib \to A$ is invertible (here $\mu$, $\Delta$, $\eta$ are the product, coproduct, and unit of $A$, respectively). 
We have (Theorem~\ref{thm:tft}):
\begin{theorem}
  Let $A\in\Sc$ be a Frobenius algebra with $N^r=\id$ and with invertible window element
  in a symmetric monoidal category $\Sc$.
  The state-sum construction defines a symmetric monoidal functor
  \begin{align}
    \funZ_A:\Bord{r}\to\Sc \ .
    \label{eq:thm:state-sum-tft}
  \end{align}
  \label{thm:state-sum-tft}
\end{theorem}

We prove this theorem by reducing it to the case of state sums for triangulations, which was treated in \cite{Novak:2015phd}.
State-sum constructions in the case of 2-spin were considered previously in \cite{Barrett:2013sp,Novak:2014sp,Gaiotto:2016spin}.

Write $\funZ_A$ for the functor \eqref{eq:intro-rTFT-Z} obtained in this way. We show that
\begin{align}
	Z^r(A) := \bigoplus_{\lambda \in \Zb_r} Z_\lambda \ ,
\end{align}
where $Z_\lambda$ is the value of the functor $\funZ_A$ on the $r$-spin circle $\lambda$, gets equipped by $\funZ_A$ with a unital associative $\Zb_r$-graded 
algebra structure which can be understood as a $\Zb_r$-graded version of the centre of an algebra
(Proposition~\ref{prop:state-space-graded-center}).
For $r=2$, this algebraic structure on state spaces has also been found in \cite{Moore:2006db}.

In \cite{Dyckerhoff:2015csg} Frobenius algebras with $N^r=\id$ appear under the name of $\Lambda_r$-Frobenius algebras in relation to $r$-spin surfaces. In \cite{Stern:2016stft} $\Lambda_r$-Frobenius algebras have been used to describe $r$-spin TFTs defined on ``open bordisms'', 
meaning that the objects in the bordism category are disjoint unions of intervals.
Our $r$-spin TFTs are defined on ``closed bordisms'', meaning that objects are disjoint unions of circles. These $r$-spin TFTs have been classified in \cite{SternSzegedy} in terms of closed $\Lambda_r$-Frobenius algebras.

\subsubsection*{TFT computing the Arf-invariant}
As an example, let $\mathcal{S}$ be the category of super vector spaces
over some field $k$ not of characteristic $2$ and $A$ the Clifford algebra $\Cl(1) = k \oplus k \theta$ in one odd generator $\theta$. Assume that $r$ is even. One finds that $Z_\lambda = k \theta^\lambda$ for $\lambda \in \Zb_r$ and that the following holds (Section~\ref{sec:arftft} and Theorem~\ref{thm:arftft}):

\begin{theorem}\label{intro:1}
Let $\Sigma$ be an $r$-spin surface of genus $g$
with $b$ ingoing boundary components of $r$-spin structures $\lambda_1,\dots,\lambda_b \in \Zb_r$ and no outgoing boundary components. Then
\begin{align}
	\funZ_{\Cl(1)}(\Sigma)(\theta^{\lambda_1} \otimes \cdots \otimes \theta^{\lambda_b}) = 2^{1-g} \, (-1)^{\Arf(\Sigma)} \ ,
\end{align}
where $\Arf(\Sigma) \in \Zb_2$ is the Arf-invariant of the $r$-spin structure of $\Sigma$ as defined in 
\cite{Randal:2014rs,Geiges:2012rs}.
\end{theorem}

By construction, $\funZ_{\Cl(1)}(\Sigma)$ is invariant under the action of the mapping class group of $\Sigma$. Thus the above theorem also proves that the 
$r$-spin 
Arf-invariant is constant on mapping class group orbits, a fact already shown in \cite{Randal:2014rs,Geiges:2012rs} by different means.
For usual spin structures, so $r=2$, the fact that a spin-TFT can compute the Arf-invariant 
(incidentally, for the same algebra) was already noticed in 
\cite{Moore:2006db,Gunningham:2016sph,Barrett:2013sp,Gaiotto:2016spin}.
{}From this point of view Theorem~\ref{intro:1} is
not surprising as an $r$-spin structure for even $r$ also defines a 2-spin structure, and this correspondence is compatible with the Arf-invariant.
A slightly different notion of an Arf-invariant for $r$-spin structures is given in 
\cite[Def.\,5.1]{Natanzon:2004harf}.

\subsubsection*{Structure of this paper}
This paper is organised as follows. In Section~\ref{sec:comb} 
we describe the combinatorial model for $r$-spin structures and state its main properties. 
In Section~\ref{sec:tft} we use this model to give a state-sum construction of $r$-spin TFTs, 
and we compute the value of these TFTs on several bordisms as an example. 
In Section~\ref{sec:arf} we show that for $r$ even, 
the $r$-spin state-sum TFT for the two-dimensional Clifford algebra computes the $r$-spin Arf-invariant.
Finally, in Appendix~\ref{app:novak} we relate the description of $r$-spin structures in terms of PLCW-decompositions that we use here to the triangulation-based model of \cite{Novak:2015phd}. We furthermore give the proofs of those properties of the combinatorial model and of $r$-spin state-sum TFTs which require the triangulation-based model and have been omitted in the main text.

\subsubsection*{Acknowledgments}

We would like to thank
	Nils Carqueville,
	Tobias Dyckerhoff,
	Jan Hesse,
	Ehud Meir,
	Sebastian Novak,
	Louis-Hadrien Robert,
	Nick Salter,
	Walker Stern and
	Lukas Woike
for helpful discussions and comments. LS was supported by the 
DFG Research Training Group 1670 ``Mathematics Inspired by String Theory and Quantum Field Theory''.

\newpage

\section{Combinatorial description of r-spin surfaces}\label{sec:comb}

In this section we present the combinatorial model for of $r$-spin structures and state its properties. We start by reviewing the definition of an $r$-spin structure (Section~\ref{sec:rspinbord}) and of the decomposition of surfaces we will use (Section~\ref{sec:PLCW-dec}). The main results in this section are the bijection of the combinatorial data modulo an appropriate equivalence relation and isomorphism classes of $r$-spin structures (Theorem~\ref{thm:rscomb} in Section~\ref{sec:combsub}) and the counting of these isomorphism classes for compact connected surfaces (Proposition~\ref{prop:sigmagb} in Section~\ref{sec:sigmagb}).

\subsection{$r$-spin surfaces}\label{sec:rspinbord}

Here we recall the definition of $r$-spin structures and of related notions, following \cite{Novak:2015phd}.
Denote by $GL_2^+(\Rb)$ the set of real $2{\times}2$ matrices of positive determinant, and
let $p_{GL}^r:\widetilde{GL}_2^r\to GL_2^+(\Rb)$ be the $r$-fold connected cover
for $r\in\Zb_{>0}$ and the universal cover for $r=0$.
Note that in both cases the fibres are isomorphic to $\Zb_r=\Zb/r\Zb$.
By a \textsl{surface} we mean an oriented two-dimensional smooth manifold. For a surface $\Sigma$ we denote by $F_{GL^+}\Sigma\to\Sigma$ the oriented frame bundle over $\Sigma$
(``oriented'' means that orientation on the tangent space induced by the frame agrees with that of $\Sigma$).

\begin{definition}\label{def:rspin-basics}
\begin{enumerate}
	\item An \textsl{$r$-spin structure} on a surface	
$\Sigma$
is a pair $(\eta,p)$, where $\eta:P_{\widetilde{GL}}\Sigma\to\Sigma$ is a 
principal $\widetilde{GL}_2^r$-bundle and $p:P_{\widetilde{GL}}\Sigma\to F_{GL^+}\Sigma$
is a bundle map intertwining the $\widetilde{GL}_2^r$- and $GL_2^+$-actions
on $P_{\widetilde{GL}}\Sigma$ and $F_{GL^+}\Sigma$ respectively.
\item
	An \textsl{$r$-spin surface} is a surface together with an $r$-spin structure.
\item\label{def:rspin-basics.3}
A \textsl{morphism of $r$-spin surfaces} $\tilde{f}:\Sigma\to\Sigma'$ is a
bundle map between the $r$-spin surfaces,
such that the underlying map of surfaces $f$ is a local diffeomorphism, and
such that the diagram 
\begin{equation}
  \begin{tikzcd}[column sep=large, row sep=large]
	  P_{\widetilde{GL}}\Sigma \rar{\tilde{f}} \dar[swap]{p} & P_{\widetilde{GL}}\Sigma'  \dar{p'} \\
	  F_{GL}\Sigma \rar{df_*} \dar  & F_{GL}\Sigma' \dar \\
	\Sigma \rar{f}&  \Sigma'
  \end{tikzcd}\label{eq:def:morphism-of-r-spin-surfaces}
\end{equation}
commutes, where $df_*$ denotes the induced map from the derivative of $f$.
\item\label{def:rspin-basics.4}
A \textsl{morphism of $r$-spin structures over $\Sigma$}
is a morphism of $r$-spin surfaces
whose  underlying map of surfaces is the
identity on $\Sigma$. 
We write
\begin{equation}
  \Rc^r(\Sigma)
\end{equation}
for the set of isomorphism classes of $r$-spin structures on $\Sigma$.
\end{enumerate}
\end{definition}

Note that $p:P_{\widetilde{GL}}\Sigma\to F_{GL^+}\Sigma$ is
a $\Zb_r$-principal bundle ($r\in\Zb_{\ge0}$).
Also, morphisms of $r$-spin structures are always isomorphisms as they are maps of principal bundles.
A \textsl{diffeomorphism of $r$-spin surfaces} is
a morphism of $r$-spin surfaces with a diffeomorphism as
underlying map of surfaces. 
Let us denote by 
	\begin{align}
		\Dc^r(\Sigma)
		\label{eq:diffeoclass}
	\end{align}
	the diffeomorphism classes of
	$r$-spin surfaces with underlying surface $\Sigma$.
Note that by construction we have a surjection
	\begin{align}
		\Rc^r(\Sigma)\twoheadrightarrow\Dc^r(\Sigma)\ ,
		\label{eq:surjection-of-classes}
	\end{align}
	given by passing to orbits under the action of the mapping class group of $\Sigma$ acting on $\Rc^r(\Sigma)$.
	As we shall see, this surjection is almost never injective.	

\medskip

Even though we do not need it in the rest of the paper, let us mention that a 0-spin structure is the same as a framing. 
A \textsl{framing of $\Sigma$} is a homotopy class of trivialisations of the oriented frame bundle over $\Sigma$. Let $T(\Sigma)$ denote the set of framings of $\Sigma$. We have:

\begin{proposition} \label{prop:framing}
There is a bijection $T(\Sigma) \xrightarrow{\ \sim\ } \Rc^0(\Sigma)$.
\end{proposition}

\begin{proof}
Take a framing and pick a representative trivialisation, i.e.\ an isomorphism of $GL_2^+$ principal bundles
$\varphi : F_{GL}\Sigma \xrightarrow{\ \sim\ }  GL_2^+\times\Sigma$.
Define
\begin{align}
	p_{\varphi} \,:=\,& \left[
	\widetilde{GL}_2^0 \times \Sigma
	\xrightarrow{p_{GL}^0\times\id_{\Sigma}}
	GL_2^+ \times \Sigma
	\xrightarrow{\varphi^{-1}}
	F_{GL}\Sigma
	\right] \ ,
	\nonumber \\
	\pi_{\varphi}\,:=\,& \left[
	\widetilde{GL}_2^0\times\Sigma
	\xrightarrow{p_{\varphi}}F_{GL}\Sigma\to\Sigma \right] \ .
\end{align}
Then $\rho_{\varphi}:=(\pi_{\varphi},p_{\varphi})$ is a 0-spin structure.
Changing $\varphi$ by a homotopy gives an isomorphic 0-spin structure
	\cite[Ch.\,4, Thm.\,9.9]{Husemoller:fb}.
This defines a map $F: T(\Sigma) \to \Rc^0(\Sigma)$.

Next we define a map in the opposite direction.
	Since $\widetilde{GL}_2^0$ is contractible, for any
	0-spin structure $\zeta=(\pi:P_{\widetilde{GL}}\Sigma\to\Sigma,p)$, 
	$\pi$ is a trivialisable $\widetilde{GL}_2^0$ principal bundle 
		\cite[Thm.\,12.2]{Steenrod:1951top}. 
	Let $\tilde{\phi}_{\zeta}:P_{\widetilde{GL}}\Sigma\to \widetilde{GL}_2^0\times\Sigma$
	denote such a trivialisation. Then there exists a unique morphism of
	principal $GL_2^+$ bundles $\phi_{\zeta}:F_{GL}\Sigma\to GL_2^+\times\Sigma$ such that
	\begin{equation}\label{eq:0-spin_frame_aux1}
	  \begin{tikzcd}[column sep=large, row sep=large]
		  P_{\widetilde{GL}}\Sigma \rar{\tilde{\phi}_{\zeta}} \dar[swap]{p} & 
		  \widetilde{GL}_2^0\times\Sigma  \dar{p_{GL}^0\times\id_{\Sigma}} \\
		  F_{GL}\Sigma \rar{\phi_{\zeta}} & 
	GL_2^+\times\Sigma
	  \end{tikzcd}
	\end{equation}
	commutes.
	Again by contractability, any two choices of trivialisations $\tilde\phi_\zeta$ are homotopic
	and so the corresponding $\phi_\zeta$ are homotopic, too. 
	By the same argument, different choices of representatives of isomorphism
	classes of $0$-spin structures give homotopic $\phi_{\zeta}$'s.
	This defines a map $G:   \Rc^0(\Sigma) \xrightarrow{\ \sim\ } T(\Sigma)$.

The two maps $F$ and $G$ are inverse to each other. Indeed, for $[\zeta] \in \mathcal{R}^0(\Sigma)$, the $0$-spin structure one obtains after constructing $F(G([\zeta]))$ is isomorphic to $\zeta$ via $\tilde\phi_\zeta$ as in \eqref{eq:0-spin_frame_aux1}, so that indeed $F(G([\zeta])) = [\zeta]$. Conversely, starting from a homotopy class of trivialisations $[\varphi] \in T(\Sigma)$, 
in computing $G(F([\varphi]))$ we see that in \eqref{eq:0-spin_frame_aux1} we can take $\tilde\phi_\zeta = \id$ and $\phi_\zeta = \varphi$, so that $G(F([\varphi])) = [\varphi]$.
\end{proof}

After this aside on framings, let us return to $r$-spin surfaces and give a basic example which will later serve to parametrise the boundary components of $r$-spin bordisms.

\begin{example}\label{ex:cx}
For $\kappa\in\Zb$ let $\Cb^{\kappa}$ denote the $r$-spin structure on $\Cb^{\times}$ given by 
the trivial principal $\widetilde{GL}_2^r$-bundle $\widetilde{GL}_2^r\times\Cb^{\times}$
and the map 
\begin{align}
	p^{\kappa}:\widetilde{GL}_2^r\times\Cb^{\times}&\to GL_2^+\times\Cb^{\times}\nonumber\\
	(g,z)&\mapsto(z^{\kappa}.p_{GL}^r(g),z)\ ,	
	\label{eq:cxproj}
\end{align}
where $z \in \Cb^{\times}$ acts on $M \in GL_2^+$ by
\begin{align}
	z.M=
	\begin{pmatrix}
		\mathrm{Re}z&-\mathrm{Im}z\\
		\mathrm{Im}z& \mathrm{Re}z\\
	\end{pmatrix} 
	M\ .
	\label{eq:caction}
\end{align}
	Since the $\widetilde{GL}_2^r$-action is from the right
	and $p_{GL}^r$ is a group homomorphism, 
	$p^\kappa$ indeed intertwines the 
	$\widetilde{GL}_2^r$- and $GL_2^+$-actions.
\end{example}
 
\begin{lemma}
[{\cite[Sec.\,3.4]{Novak:2015phd}}]
	$\Cb^{\kappa}$ and $\Cb^{\kappa'}$ are isomorphic $r$-spin structures
	if and only if $\kappa\equiv\kappa'\Mod{r}$.
	The map $\Zb_r \to \mathcal{R}^r(\mathbb{C}^\times)$, 
	$\kappa \mapsto [\Cb^{\kappa}]$ is a bijection.
	\label{lem:clambda}
\end{lemma}

In the case that $r>0$, it will be convenient to fix once and for all a set of representatives of $\Zb_r$ in $\Zb$, say $\{0,1,\dots,r-1\}$, and to agree that for $\lambda \in \Zb_r$, $\Cb^\lambda$ stands for $\Cb^\kappa$, with $\kappa \in \Zb$ the chosen representative for $\lambda$.

\begin{notation}
For an $r$-spin surface $\Sigma$, 
by abuse of notation we will often use the same symbol $\Sigma$ to denote its underlying surface. That is, $\Sigma$ stands for the triple $\Sigma,\eta,p$ from Definition~\ref{def:rspin-basics}\,(1).
\end{notation}

A \textsl{collar} is an open neighbourhood of $\Sb$ in $\Cb^{\times}$.
An \textsl{ingoing} (resp.\ \textsl{outgoing}) \textsl{collar} is the intersection of a collar with
the set $\setc*{z\in\Cb^{\times}}{|z|\ge 1}$
(resp.\ $\setc*{z\in\Cb^{\times}}{|z|\le 1}$).
A \textsl{boundary parametrisation} of a surface $\Sigma$ is:
\begin{enumerate}
\item 
A disjoint decomposition $B_\mathrm{in} \sqcup B_\mathrm{out} = \pi_0(\partial\Sigma)$ (the in- and outgoing boundary components). $B_\mathrm{in}$ and/or $B_\mathrm{out}$ are allowed to be empty.
\item
A collection of ingoing collars $U_b$, $b \in B_\mathrm{in}$, and outgoing collars  $V_c$, $c \in B_\mathrm{out}$, together with a pair of orientation preserving embeddings
\begin{align}
	\phi_\mathrm{in}:\bigsqcup_{b\in B_\mathrm{in}}U_b\hookrightarrow\Sigma\hookleftarrow
	\bigsqcup_{c\in B_\mathrm{out}}V_c:\phi_\mathrm{out} \ .
	\label{eq:bdrparam}
\end{align}
We require that for each $b$, the restriction $\phi_\mathrm{in}|_{U_b}$ maps $\mathbb{S}^1$ diffeomorphically to the connected component $b$ of $\partial\Sigma$, and analogously for $\phi_\mathrm{out}|_{V_c}$.
\end{enumerate}
An \textsl{$r$-spin boundary parametrisation} of an $r$-spin surface $\Sigma$ is:
\begin{enumerate}
	\item A boundary parametrisation of the underlying surface $\Sigma$ as above; 
		we use the same notation as in \eqref{eq:bdrparam}.
		\label{part:bdryparam}
	\item A pair of maps 
	fixing the restriction of the $r$-spin structure to the in- and outgoing boundary components	
		\begin{align}
			\lambda:B_{in}&\to\Zb_r        & \text{and}&&\mu:B_{out}&\to\Zb_r.
			\label{eq:collarmaps}\\
			             b&\mapsto\lambda_b&           &&          c&\mapsto\mu_c\nonumber
		\end{align}
	\item A pair of morphisms of $r$-spin surfaces
	which parametrise the in- and outgoing boundary components by  collars with $r$-spin structure,	
		\begin{align}
			\varphi_\mathrm{in}:\bigsqcup_{b\in B_\mathrm{in}}
			U_b^{\lambda_b}
			\hookrightarrow\Sigma\hookleftarrow
			\bigsqcup_{c\in B_\mathrm{out}}
			V_c^{\mu_c}
			:\varphi_\mathrm{out} \ .
			\label{eq:rsbdrparam}
\end{align} 
Here, $U_b^{\lambda_b}$ is the restriction of $\Cb^{\lambda_b}$ to the ingoing collar $U_b$, and analogously $V_c^{\mu_c} := \Cb^{\mu_c}|_{V_c}$. 
The maps of surfaces underlying $\varphi_{\mathrm{in}/\mathrm{out}}$ 
are required to be the maps $\phi_{\mathrm{in}/\mathrm{out}}$ in \eqref{eq:bdrparam} from 
Part~\ref{part:bdryparam}.
\end{enumerate}
Note that by Lemma~\ref{lem:clambda}, the maps $\lambda,\mu$ in part 2 are not extra data, but are uniquely determined by the $r$-spin surface $\Sigma$
and the boundary parametrisation.

For diffeomorphisms between $r$-spin surfaces with parametrised boundary we only require that they  respect  germs of the boundary parametrisation. In more detail, let $\Sigma$ be as in \eqref{eq:rsbdrparam} and let
\begin{align}
	\psi_{in}:\bigsqcup_{d\in B'_{in}}
	P_d^{\rho_d}
	\hookrightarrow\Xi\hookleftarrow
	\bigsqcup_{e\in B'_{out}}
	Q_e^{\sigma_e}
	:\psi_{out} 
	\label{eq:rsbdrparam-2nd}
\end{align} 
be another $r$-spin surface with boundary parametrisation. 
A \textsl{diffeomorphism of $r$-spin surfaces with boundary parametrisation} $\Sigma \to \Xi$ is an $r$-spin diffeomorphism $f : \Sigma \to \Xi$ subject to the following compatibility condition. Let $b \in B_\mathrm{in}^\Sigma$ be an ingoing boundary component of $\Sigma$ and let $f_*(b) \in \pi_0(\partial\Xi)$ be its image under $f$. We require that $f_*(b) \in B_\mathrm{in}^\Xi$ and that $\lambda_b = \rho_{f_*(b)}$. Furthermore, there has to exist an ingoing collar $C$ contained in both $U_b$ and $P_{f_*(b)}$ such that the diagram
\begin{equation}
	\begin{tikzcd}[row sep=small]
		& U_b^{\lambda_b} \ar[hookrightarrow]{r}{\varphi_\mathrm{in}} & \Sigma \ar{dd}{f} \\
		C^{\lambda_b} \ar[hookrightarrow]{ru} \ar[hookrightarrow]{rd} \\
		& P_{f_*(b)}^{\rho_{f_*(b)}} \ar[hookrightarrow]{r}{\psi_\mathrm{in}} & \Xi
	\end{tikzcd}
\end{equation}
of $r$-spin morphisms
commutes. An analogous condition has to hold for each outgoing boundary component $c \in B_\mathrm{out}$.

\medskip

By an \textsl{$r$-spin object} we mean a pair $(X,\rho)$ 
consisting of a finite set $X$ and a map 
$\rho : X \to \Zb_r$, $x \mapsto 
\rho_x$. 
Below we will construct a category whose objects are $r$-spin objects,
and whose morphisms are certain equivalence classes of $r$-spin surfaces, which we turn to now.

\begin{definition}
Let $(X,\rho)$ and $(Y,\sigma)$ be two $r$-spin objects.
An \textsl{$r$-spin bordism from $(X,\rho)$ to $(Y,\sigma)$} is a compact
$r$-spin surface $\Sigma$ with boundary parametrisation  as in \eqref{eq:rsbdrparam} together with bijections 
$\beta_\mathrm{in} : X \xrightarrow{~\sim~} B_\mathrm{in}$ and $\beta_\mathrm{out} : Y \xrightarrow{\ \sim\ } B_\mathrm{out}$ such that
\begin{equation} \label{eq:boundary-label-compatible}
\begin{tikzcd}
	X \ar{rr}{\beta_\mathrm{in}} \ar{dr}[swap]{\rho} && 
	B_\mathrm{in} \ar{dl}{\lambda} \\
	& \Zb_r&
\end{tikzcd}
\quad \text{and} \quad
\begin{tikzcd}
	Y \ar{rr}{\beta_\mathrm{out}} \ar{dr}[swap]{\sigma} && 
	B_\mathrm{out} \ar{dl}{\mu} \\
	& \Zb_r&
\end{tikzcd}
\end{equation}
commute.
We will often abbreviate an $r$-spin bordism $\Sigma$ from $(X,\rho)$ to $(Y,\sigma)$ as $\Sigma : \rho \to \sigma$. 
\end{definition}

Given $r$-spin bordisms $\Sigma : \rho \to \sigma$ and $\Xi : \sigma \to \tau$, the \textsl{glued $r$-spin bordism} $\Xi \circ \Sigma : \rho \to \tau$ is defined as follows. Denote by $Y$ the source of $\sigma$, i.e.\ $\sigma : Y \to \Zb_r$. 
For every $y \in Y$, the boundary component $\beta^\Sigma_\mathrm{out}(y)$ of $\Sigma$ is glued to the boundary component  $\beta^\Xi_\mathrm{in}(y)$ of $\Xi$ using the $r$-spin boundary parametrisations $\varphi^\Sigma_\mathrm{out}$ and $\varphi^\Xi_\mathrm{in}$. 
The diagrams in~\eqref{eq:boundary-label-compatible} ensure that the $r$-spin structures on the corresponding  collars 
are restrictions of the same $r$-spin structure on $\Cb^\times$.

Two $r$-spin bordisms between the same $r$-spin objects,
$\Sigma,\Sigma' : (X,\rho) \to (Y,\sigma)$ are called \textsl{equivalent}
if there is a diffeomorphism $f: \Sigma \to \Sigma'$ of $r$-spin surfaces with boundary parametrisation such that
with $f_* : \pi_0(\partial\Sigma) \to \pi_0(\partial \Sigma')$,
\begin{equation}
	\begin{tikzcd}[row sep=small]
		& B_\mathrm{in} \ar{dd}{f_*} \\
		X \ar{ru}{\beta_{\mathrm{in}}} \ar{rd}[swap]{\beta'_{\mathrm{in}}} \\
		  & B'_\mathrm{in}
	\end{tikzcd}
	\qquad \text{and} \qquad
	\begin{tikzcd}[row sep=small]
		B_\mathrm{out} \ar{dd}[swap]{f_*} \\
		& Y \ar{lu}[swap]{\beta_{\mathrm{out}}} \ar{ld}{\beta'_{\mathrm{out}}} \\
		  B'_\mathrm{out}
	\end{tikzcd}
\end{equation}
commute. 
Let $[\Xi]:\sigma\to\tau$ and $[\Sigma]:\rho\to\sigma$ be equivalence classes of $r$-spin bordisms.
The composition $[\Xi]\circ[\Sigma]:=[\Xi\circ\Sigma]:\rho\to\tau$ is well defined, that is independent of the choice of representatives $\Xi$, $\Sigma$ of the classes to be glued.
In the following we will by abuse of notation write the same symbol $\Sigma$ for an $r$-spin bordism $\Sigma$ and its equivalence class $[\Sigma]$.

\begin{definition}
	The \textsl{category of $r$-spin bordisms} $\Bord{r}$ has $r$-spin objects
as objects and equivalence classes of $r$-spin bordisms as morphisms.
\end{definition}

$\Bord{r}$ is a symmetric monoidal category with tensor product 
on objects and morphisms given by disjoint union.
The identities and the symmetric structure are given by $r$-spin cylinders with appropriately parametrised boundary.

\subsection{PLCW decompositions}\label{sec:PLCW-dec}

\begin{figure}[b!]
	\centering
	\def\svgwidth{16cm}
	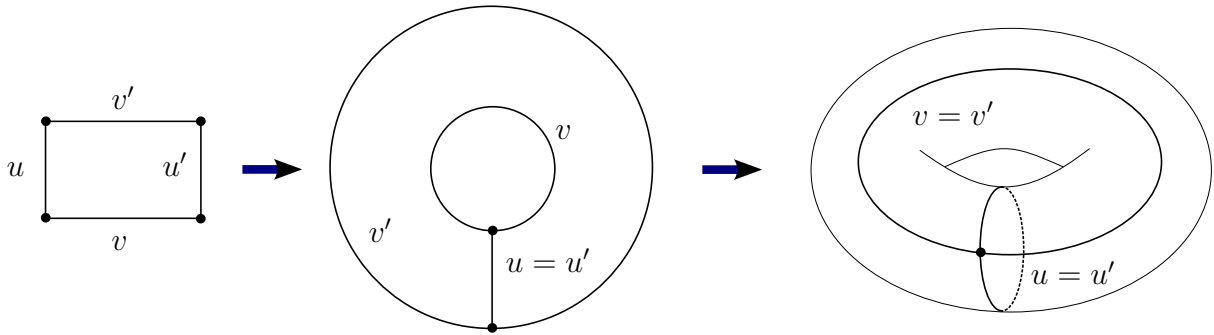
	\caption{Glueing a torus from a rectangle. 
		Each step is a regular cell map and each 
		generalised cell decomposition is a PLCW decomposition.}
	\label{fig:torus}
\end{figure}

In Section~\ref{sec:combsub} we will use a cell decomposition to combinatorially encode 
$r$-spin structures on surfaces, and
in Section~\ref{sec:state-sum-constr} we will use this description to build an $r$-spin TFT.
For explicit calculations it is helpful to keep the number of faces and edges to a minimum. 
The notion of a PLCW decomposition from \cite{Kirillov:2012pl}, and which we review in this section, is well suited for such calculations.
For example, there is a PLCW decomposition of a torus consisting of
1 face, 2 edges and 1 vertex, see Figure~\ref{fig:torus}.
For comparison, using simplicial sets would require at least
2 faces, 3 edges and 1 vertex; 
using simplicial complexes 
(i.e.\ triangulations, as in \cite{Novak:2015phd}) 
would require at least
14 faces, 21 edges and 7 vertices (see e.g.\ \cite{Lutz:2005fw}).

Now we turn to the definitions following \cite{Kirillov:2012pl}.
Let $C\subset\Rb^N$ be a compact set,
let $\mathring{C}$ denote its interior and
let $\mathbullet{C}:=C\setminus\mathring{C}$ denote its boundary.
Let $B^N=\left[ -1,1 \right]^N\in\Rb^N$ 
denote the closed $N$-ball,
or rather a piece-wise linear (PL for short) version thereof. 
Then $\mathbullet{B}^N=S^{N-1}$ is the 
(PL-version of the)
$(N-1)$-sphere.
A PL
map $\varphi:C\to\Rb^M$ is called a \textsl{regular map} 
if $\varphi|_{\mathrm{Int}(C)}$ is injective.
A compact subset $C\subset\Rb^N$ is a \textsl{generalised $n$-cell} 
(or simply \textsl{cell}), if $\mathring{C}=\varphi(\mathring{B}^n)$ and
$\mathbullet{C}=\varphi(\mathbullet{B}^n)$ for a regular map $\varphi:B^n\to C$,
which we call a \textsl{characteristic map of $C$}.
A \textsl{generalised cell decomposition} is a finite collection of cells
such that the interiors of cells do not intersect and the
boundary of any cell is a union of cells.
Examples are shown in Figure~\ref{fig:torus} and in Figure~\ref{fig:gcc}.
We denote the \textsl{$n$-skeleton} of $K$ by $K^n$,
which is the union of the \textsl{set of $k$-cells} $K_k$ with $k\le n$, and 
we define the \textsl{dimension} $\dim{K}$ of $K$ to be the highest integer $n$ for which the
set of $n$-cells is nonempty.
We denote the set of boundaries of an $n$-cell $C\in K_n$ by $\partial(C)\subset K_{n-1}$.
A \textsl{regular cell map} $f:L\to K$ between 
generalised cell decompositions $L$ and $K$
is a piecewise linear 
map $f:\bigcup_{C\in L}C\to\bigcup_{D\in K}D$
such that for every $C\in L$ with characteristic map $\varphi$
there is a cell $D=f(C)\in K$ for which $f\circ\varphi$ is a characteristic map.
An example of a regular cell map is shown in Figure~\ref{fig:torus},
a non-example is shown in Figure~\ref{fig:gcc}\,$b)$.

\begin{definition}
	A \textsl{PLCW decomposition} $K$ is a generalised cell decomposition of dimension $n$
	such that if $n>0$
	\begin{itemize}
		\item $K^{n-1}$ is a PLCW decomposition and
		\item for any $n$-cell $A\in K_n$ with characteristic map $\varphi$
			there is a PLCW decomposition $L$ of $S^{n-1}$, such that 
			$\varphi|_{S^{n-1}}:L\to K^{n-1}$ is a regular cell map.
	\end{itemize}
\end{definition}

\begin{figure}[tb]
	\centering
	\def\svgwidth{16cm}
	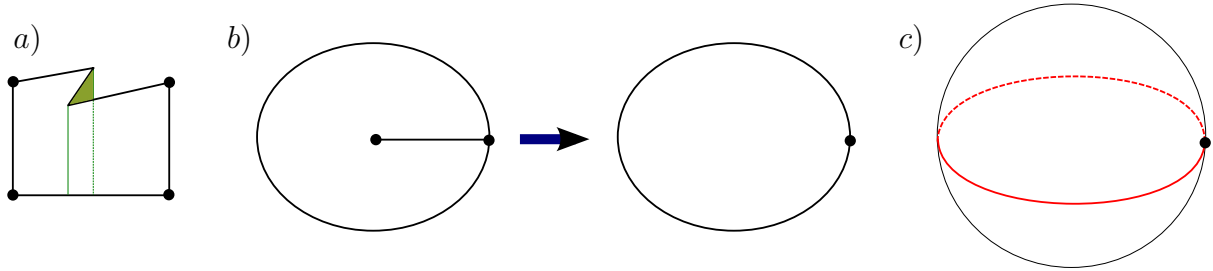
	\caption{$a)$ A generalised cell decomposition which is not a PLCW decomposition.
		There are one 2-cell, four 1-cells and four 0-cells. 
		One can visualise it by folding a paper and glueing it only along the bottom edge.
		$b)$ A triangle with two sides identified and a 1-gon, 
		both PLCW decompositions. 
		The map between them is not a regular cell map
		as the edge in the middle has no image.
		$c)$ A PLCW decomposition of a sphere into two faces, 
		one edge (red line) and one vertex.}
	\label{fig:gcc}
\end{figure}

Examples of PLCW decompositions are shown in Figure~\ref{fig:torus}, Figure~\ref{fig:gcc} $b)$ and $c)$.
A generalised cell decomposition which is not a PLCW decomposition 
is shown in Figure~\ref{fig:gcc} $a)$.
Each PLCW decomposition can be related by a series of local elementary moves (cf.\ Section~\ref{sec:eltmove} below),
and each PLCW decomposition can be refined to a simplicial complex 
\cite[Thm.\,6.3]{Kirillov:2012pl}.
For more details see~\cite[Sec.\,6--8]{Kirillov:2012pl}.

{}From now on we specialise to 2 dimensional PLCW decompositions.
Let $\Sigma$ be a compact surface with
a PLCW decomposition $\Sigma_2$, $\Sigma_1$, $\Sigma_0$.
We call these sets \textsl{faces}, \textsl{edges} and \textsl{vertices} respectively;
one can think of faces as $n$-gons with $n\ge1$.
For $g+b\ge1$, PLCW decompositions also allow for a decomposition of 
any compact connected surface $\Sigma_{g,b}$ of genus $g$ and with $b$ boundary components 
into a single face which is a $(4g+3b)$-gon, see Section~\ref{sec:sigmagb}.

\medskip

To apply PLCW decompositions to smooth manifolds, we can use that 
a PLCW decomposition can be refined to a simplicial complex, 
and that PL cell maps for a simplicial complex can be approximated by smooth maps, 
giving smooth manifolds \cite[Sec.\,10]{Munkres:1966dt}.

\subsection{Combinatorial description of $r$-spin structures}\label{sec:combsub}

In this section we extend the combinatorial description of $r$-spin structures in \cite{Novak:2015phd}, 
which uses a triangulation of the underlying surface, to PLCW decompositions.
We will only consider PLCW decompositions where the boundary components consist of a single vertex and a single edge. 
In the following section we explain how this combinatorial model relates to the description of $r$-spin structures in terms of holonomies of curves.

Let $\Sigma$ be a surface with parametrised boundary, with a PLCW decomposition,
with a marking of one edge of each face and an orientation of each edge.
We do not require that the orientation of the boundary edges
corresponds to the orientation of the boundary components, 
but we orient the faces according to the orientation of the surface.
This induces an ordering of the edges of each face, the starting edge being the marked one, see Figure~\ref{fig:clockwisevertex}.
By an \textsl{edge index assignment} we mean a map $s:\Sigma_1\to\Zb_r$, $e\mapsto s_e$.

\begin{definition}\label{def:marking-PLCW}
	We call an assignment of edge markings, edge orientations
	and edge indices a \textsl{marking} of a PLCW decomposition
	and a PLCW decomposition together with a marking a \textsl{marked PLCW decomposition}.
\end{definition}

\begin{figure}[tb]
	\centering
	\def\svgwidth{6cm}
	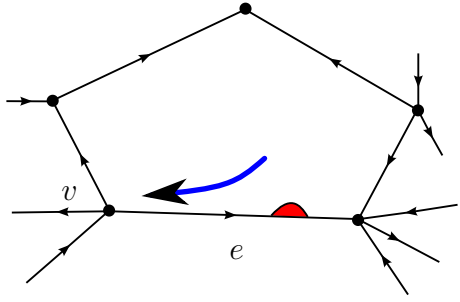
	\caption{Figure of a face with adjacent edges and vertices in a marked PLCW decomposition.
		The orientation of the face is that of the paper plane,
		the orientation of the edges is 
			indicated by an arrow on them.
	The half-dot indicates the marked edge of the face the half-dot lies in.
	The arrow in the middle
	shows the clockwise direction along the marked edge $e$ and 
	$v$ is the vertex sitting on the boundary of $e$ in clockwise direction.
	Note that the clockwise vertex $v$ of the edge $e$ is determined by the orientation of the face and
	not by the orientation of the edge $e$.}
	\label{fig:clockwisevertex}
\end{figure}

For a vertex $v\in\Sigma_0$ let $D_v$ be 
the number of faces whose marked edge has $v$ as its boundary vertex in clockwise direction (with respect to the orientation of the face),
as shown in Figure~\ref{fig:clockwisevertex}.
Let $\partial^{-1}(v)\subset\Sigma_1$ denote the edges
whose boundary contain $v$: 
\begin{align}
	\partial^{-1}(v):=\setc*{e\in\Sigma_1}{v\in\partial(e)}\ .
	\label{eq:edges-of-vertex}
\end{align}
The orientation of an edge gives a starting and an ending vertex, which might be the same.
Let $N_v^{\mathrm{start}}$ (resp. $N_v^{\mathrm{end}}$) be the number of edges starting (resp.\ ending) at the vertex $v$
and let 
\begin{align}
\label{eq:Nv=Nvstart+Nvend}
	N_v=N_v^{\mathrm{start}}+N_v^{\mathrm{end}} \ .
\end{align}
We note that an edge which starts and ends at $v$ contributes 1 to both $N_v^{\mathrm{start}}$ and to $N_v^{\mathrm{end}}$.
For every edge $e\in\partial^{-1}(v)$ let
\begin{align}
	\hat{s}_e=
	\begin{cases}
		-1&\text{ if $e$ starts and ends at $v$,}\\
		s_e&\text{ if $e$ is pointing out of $v$,}\\
		-1-s_e&\text{ if $e$ is pointing into $v$.}
	\end{cases}
	\label{eq:modedgeind}
\end{align}
Recall the maps $\lambda:B_{in}\to\Zb_r$ and 
$\mu:B_{out}\to\Zb_r$ from \eqref{eq:collarmaps}, as well as
our convention that we only consider PLCW decompositions with exactly one vertex and one edge on each boundary component.
For a vertex $u$ on a boundary component let
us write by slight abuse of notation $u$ 
for this boundary component and let
\begin{align}
	R_u:=
	\begin{cases}
		\lambda_u-1&\text{ if $u\in B_{in}$,}\\
		1-\mu_u&\text{ if $u\in B_{out}$.}
	\end{cases}
	\label{eq:bdrysign}
\end{align}
We call a marking \textsl{admissible} with given maps $\lambda$ and $\mu$,
if for every inner vertex $v\in\Sigma_0\cap(\Sigma\setminus\partial\Sigma)$ and for every boundary vertex $u\in\Sigma_0\cap\partial\Sigma$
the following conditions are satisfied:
\begin{align}
	\sum_{e\in \partial^{-1}(v)}\hat{s}_e&\equiv D_v-N_v+1&\Mod{r}\ ,
	\label{eq:vertexcond1}\\
	\sum_{e\in \partial^{-1}(u)}\hat{s}_e&\equiv D_u-N_u+1-R_u&\Mod{r}\ .
	\label{eq:vertexcond2}
\end{align}

For an arbitrary marking of a PLCW decomposition of $\Sigma$ 
one can define an $r$-spin structure with $r$-spin
boundary parametrisation on $\Sigma$ minus its vertices
by taking the trivial $r$-spin structure on faces and fixing
the transition functions using the marking.
The marking of an edge for a face contains the information how a standard polygon is embedded into the surface.
The above $r$-spin structure extends uniquely to the vertices of $\Sigma$,
if and only if the marking is admissible for $\lambda$ and $\mu$.
The $r$-spin boundary parametrisations are given by the
inclusion of $r$-spin collars (as prescribed by $\lambda$ and $\mu$)	
over the collars of the boundary parametrisation of $\Sigma$.
For more details on this construction we refer the reader to 
Appendix~\ref{app:refine}.
\begin{definition}\label{def:rspinstrindex}
	Denote the $r$-spin structure with $r$-spin boundary parametrisation 
	defined above by $\Sigma(s,\lambda,\mu)$. 
\end{definition}

There is some redundancy in the description of an $r$-spin structure via a marking. A one-to-one correspondence between certain equivalence classes of markings and isomorphism classes of $r$-spin structures will be given in Theorem~\ref{thm:rscomb} below. As preparation we first give a list of local modifications of the marking which lead to isomorphic $r$-spin structures.

\begin{figure}[tb]
	\centering
	\def\svgwidth{16cm}
	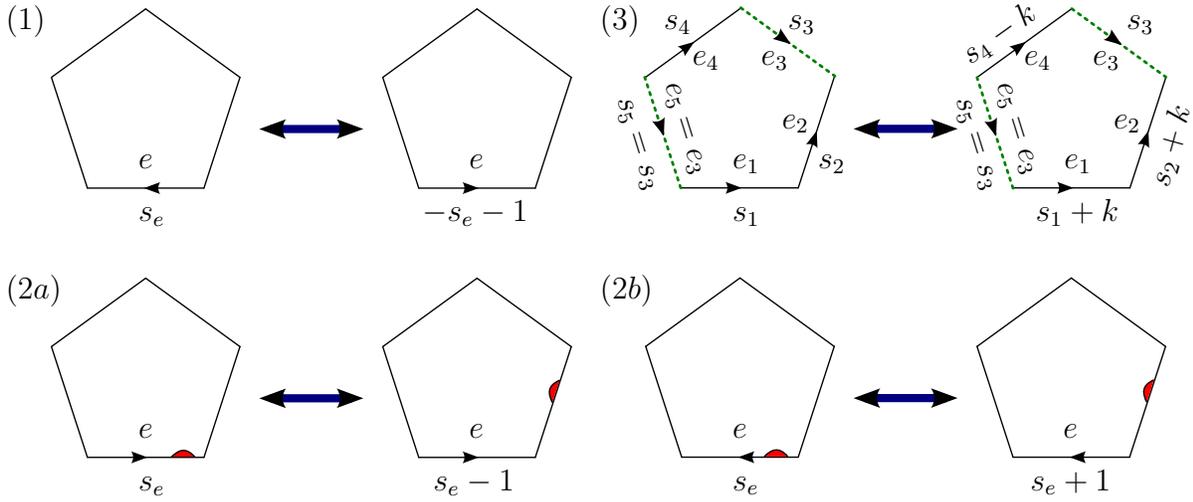
	\caption{Moves of Lemma~\ref{lem:moves} for a face of $\Sigma$. 
	All edge orientations and markings are arbitrary
	unless shown explicitly. $(1)$ Flipping the edge orientation of $e$.
	$(2a)$, $(2b)$ Moving the edge marking for a face. $(3)$ Shifting the edge indices for a face.
	The dotted edges $e_3$ and $e_5$ are identified, hence the edge index remains unchanged.
	The edges $e_1$ and $e_2$ are counterclockwise oriented, hence the $+k$ shift
	of the corresponding edge indices $s_1$ and $s_2$,
	the edge $e_4$ is clockwise oriented, hence the $-k$ shift of $s_4$.}
	\label{fig:fixmoves}
\end{figure}

\begin{lemma} \label{lem:moves}
	The following changes of the marking of the PLCW decomposition of $\Sigma$ (but keeping the PLCW decomposition fixed) 
	give isomorphic
$r$-spin structures:
	\begin{enumerate}
		\item Flip the orientation of an edge $e$ and change its edge index $s_e\mapsto -1-s_e$ (see Figure~\ref{fig:fixmoves}\,(1)).
			\label{lem:moves:1}
		\item Move the marking on an edge $e$ of a polygon to the following edge counterclockwise
			and change the edge index of the previously marked edge $s_e\mapsto s_e-1$,
			if this edge is oriented counterclockwise, $s_e\mapsto s_e+1$ otherwise (see Figure~\ref{fig:fixmoves} ($2a$) and ($2b$)).
			\label{lem:moves:2}
		\item Let $k\in\Zb$. Shift the edge index of each edge of a polygon 
			by $+k$, if the edge is oriented counterclockwise 
			with respect to the orientation of the polygon,
			and by $-k$ otherwise.
		If two edges of a polygon are identified (i.e.\ are given by the same $e \in \Sigma_1$),
		do not change its edge index.
			For an illustration, see Figure~\ref{fig:fixmoves} Part~$3$. We call this a \textsl{deck transformation}.
			\label{lem:moves:3}
	\end{enumerate}
These operations on the marking commute with each other
in the sense that the final edge indices do not depend on the order in which a given set of operations \ref{lem:moves:1}--\ref{lem:moves:3} is applied.
\end{lemma}

Note that the operation in~\ref{lem:moves:3} is the same as moving around the marking of a face completely by applying operation~\ref{lem:moves:2}.
This lemma  is proved in Appendix~\ref{app:proofs}.

Let $\Sigma$ be a surface with a fixed PLCW decomposition. 
Write $(m,o,s)$ for a given marking of $\Sigma$, where $m$ denotes the edge markings of the faces, $o$ the edge orientations and $s$ the edge indices (cf.\ Definition~\ref{def:marking-PLCW}).
Let $\Mc(\Sigma)_{\lambda,\mu}^{PLCW}$ denote 
the set of all admissible markings for the
maps $\lambda$ and $\mu$ on $\Sigma$.
The operations in Lemma~\ref{lem:moves} generate an equivalence relation $\sim_{\mathrm{fix}}$
on $\Mc(\Sigma)_{\lambda,\mu}^{PLCW}$. 
Let us denote equivalence classes by $\left[ m,o,s \right]$.
The following lemma gives a more concrete description of the equivalence classes.

\begin{lemma}\label{lem:fixed-mo-relation}
Let $(m,o,s) \in \Mc(\Sigma)_{\lambda,\mu}^{PLCW}$.
We have:
\begin{enumerate}
\item For every choice $m',o'$ there is some $s'$ such that $\left[ m,o,s \right]\sim_{\mathrm{fix}}\left[ m',o',s' \right]$.
\item For a given choice of edge indices $\tilde s$ we have $\left[ m,o,s \right]\sim_{\mathrm{fix}}\left[ m,o,\tilde s \right]$ if and only if $s$ and $\tilde s$ are related by a sequence of deck transformations (operation~\ref{lem:moves:3}) in Lemma~\ref{lem:moves}. 
\end{enumerate}
\end{lemma}

\begin{proof}
The first statement is immediate from operations~\ref{lem:moves:1} and~\ref{lem:moves:2} in Lemma~\ref{lem:moves}. For the second statement recall that 
operations \ref{lem:moves:1}--\ref{lem:moves:3} commute, and operation~\ref{lem:moves:3} is redundant. Any sequence of operations can thus be written as $M = \prod_e (\text{op.\,\ref{lem:moves:1} for edge $e$}) \prod_f (\text{op.\,\ref{lem:moves:2} for face $f$})$. Since $m$ and $o$ do not change, operation~\ref{lem:moves:1} for an edge $e$ must occur in pairs, leaving $s_e$ unchanged, and operation~\ref{lem:moves:2} for a face $f$ must occur in multiples of the number of edges of that face, so that the total change is expressible in terms of operation~\ref{lem:moves:3}, $M = \prod_f (\text{op.\,\ref{lem:moves:3} for face $f$})$.
\end{proof}

Let $\Rc^r(\Sigma)_{\lambda,\mu}$ denote
the isomorphism classes of $r$-spin structures with 
$r$-spin boundary parametrisation for the maps $\lambda$ and $\mu$.
The following theorem is proved in Appendix~\ref{app:proofs}.
\begin{theorem} \label{thm:rscomb}
Let $\Sigma$ be a surface with PLCW decomposition.
	The map 
	\begin{align}
		\Mc(\Sigma)_{\lambda,\mu}^{PLCW}/\sim_{\mathrm{fix}}&
		~\longrightarrow~\Rc^r(\Sigma)_{\lambda,\mu}\nonumber\\
		\left[ m,o,s \right]
		&~\longmapsto~\left[\Sigma(s,\lambda,\mu)\right]
		\label{eq:rscomb}
	\end{align}
	is a bijection. 
	On the right hand side it is understood that the edge markings and orientations of $\Sigma$ are given by $m,o$.
\end{theorem}

\begin{remark} \label{rem:rscomb}
When combined with Lemma~\ref{lem:fixed-mo-relation}, this shows
that for a fixed edge marking and orientation 
	the admissible edge index assignments up to deck transformations are
	in bijection with the isomorphism classes of $r$-spin structures with 
	$r$-spin boundary parametrisation for the maps $\lambda$ and $\mu$.
\end{remark}

\subsection{Holonomies in the combinatorial model}
  In this section we explain how to compute holonomies along curves in $\Sigma$
  which then characterise the $r$-spin structure on $\Sigma$.
  We hope that this sheds some light on the geometric meaning of the quantities appearing 
  in the combinatorial model.

Let $\Sigma$ be a compact $r$-spin surface with parametrised boundary with
maps $\lambda:B_{in}\to\Zb_r$ and $\mu:B_{out}\to\Zb_r$.
By a \textsl{curve} in $\Sigma$ we mean a smooth immersion 
$\gamma:\left[ 0,1 \right]\to \Sigma$ (i.e.\ $\gamma$ has nowhere vanishing derivative), 
and which is either closed, or which starts and ends on the 
boundary of $\Sigma$. In the former case we require in addition that 
the tangent vectors at the start and end point agree: 
$\frac{d}{dt}\gamma(0)=\frac{d}{dt}\gamma(1)$. 
In the latter case we require that the start
and end points are the images of $1\in \Sb\subset\Cb$ under the boundary parametrisation maps
and that the tangent vector of the curve is the same as the tangent vector of the boundary curve.
Two curves $\gamma_0$ and $\gamma_1$ with $\gamma_0(0)=\gamma_1(0)$ and $\gamma_0(1)=\gamma_1(1)$ 
are \textsl{homotopic} if there is a homotopy $(s,t)\mapsto \gamma_s(t)$ between them, 
such that for each $s$, $\gamma_s$ is a curve in the above sense.
In particular, since $\frac{d}{dt}\gamma$ must remain nonzero everywhere along the homotopy, one cannot ``pull straight'' a loop in the curve.

Pick a lift $\gamma_F:\left[ 0,1 \right]\to F_{GL}\Sigma$ of $\gamma$ to the oriented frame bundle
by taking the tangent vector of $\gamma$ 
(which is non-zero since $\gamma$ is an immersion) 
and adding another non-zero and 
non-parallel 
vector such that the orientation induced by them 
agrees with the orientation of the surface.
Such a lift of a curve in $\Sigma$ to $F_{GL}\Sigma$ is unique up to homotopy, 
see e.g.\ \cite[p.\,26]{Novak:2015phd}.
Also, if two curves in $\Sigma$ are homotopic, then their lifts to $F_{GL}\Sigma$
	are homotopic as well.

Consider a 
disc $D$ around $1$ in $\Cb^\times$
with $r$-spin structure $D^\kappa$ given by the restriction of $\Cb^\kappa$ 
for $\kappa\in\Zb_r$ 
as in Example~\ref{ex:cx}. As on a contractible surface, all $r$-spin structures are isomorphic (see e.g.\ \cite[Lem.\,3.10]{Novak:2015phd}), there is an isomorphism of $r$-spin structures $D^0 \to D^\kappa$. In fact, there are exactly $r$ such isomorphisms, and we pick the one which acts as the identity on the fibre over $1$ (by Example~\ref{ex:cx}, the fibre and projection over $1 \in \Cb^\times$ agree for all $\Cb^\kappa$).
This construction will be needed to assign a holonomy to curves between  different boundary components.

Recall that $P_{\widetilde{GL}}\Sigma$ is a principal $\Zb_r$ bundle over $F_{GL}\Sigma$.
Pick a lift $\tilde{\gamma}:\left[ 0,1 \right]\to P_{\widetilde{GL}}\Sigma$ 
of $\gamma_F$ to the $r$-spin bundle.
Since the fibers of $P_{\widetilde{GL}}\Sigma\xrightarrow{p}F_{GL}\Sigma$ are discrete, 
this lift is unique after fixing it at one point and homotopic curves in $F_{GL}\Sigma$ lift to 
homotopic curves in $P_{\widetilde{GL}}\Sigma$.
If $\gamma$ is a closed curve let $\zeta(\gamma)\in\Zb_r$ denote the holonomy of $\tilde{\gamma}$ at $\gamma(0)=\gamma(1)$.
If $\gamma$ is not closed, 
use the isomorphism $D^0 \to D^\kappa$ from above to identify the fibers $\Zb_r$ over the start- and end-point of $\gamma_F$, and 
let again $\zeta(\gamma)\in\Zb_r$ denote the resulting holonomy of $\tilde{\gamma}$.

\medskip

	\begin{figure}[tb]
		\centering
		\def\svgwidth{5cm}
		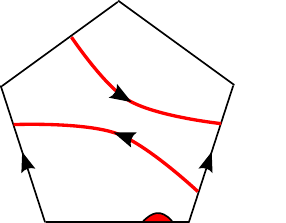
		\caption{Two arcs $p,q\in A(\gamma)$ of a curve $\gamma$ on a face $f$. 
			Here $f_p=f_q=f$, $\hat{s}_{e_p}=-s_{e_p}-1$,
			$\hat{s}_{e_q}=s_{e_q}$, $\hat{\delta}_{f_p}^p=+1$ and 
			$\hat{\delta}_{f_q}^q=0$.
		}
		\label{fig:facehol}
	\end{figure}

We now explain how to compute these holonomies in terms of the combinatorial description of $r$-spin structures.
	Take a decorated PLCW decomposition of $\Sigma$ with edge index assignment $s$ 
	and consider the $r$-spin structure $\Sigma(s,\lambda,\mu)$ 
given by Definition~\ref{def:rspinstrindex}.
We may assume the PLCW decomposition to be fine enough so that its edges 
split $\gamma$ into a set of arcs $A(\gamma)$
	as in Figure~\ref{fig:facehol}.
	Then for every $a\in A(\gamma)$ there is a face $f_a\in\Sigma_2$ containing $a$
	and an edge $e_a$ in the boundary of $f_a$ where the arc $a$ leaves the face $f_a$
	(see Figure~\ref{fig:facehol}).  
Let us assume that $e_a$ is not a boundary edge.
	For $s_{e_a}$ the edge index of the edge $e_a$ 
	let $\hat{s}_{e_a}^a=s_{e_a}$ if the edge $e_a$ and $a$ cross positively, and
	$\hat{s}_{e_a}^a=-s_{e_a}-1$ otherwise (see again Figure~\ref{fig:facehol} for conventions).
	Let $\hat{\delta}_{f_a}^a=+1$ if
	the clockwise vertex of the marked edge 
	of the face $f_a$ is on the right side of $a$
	(before glueing the edges)
	and 
		$\hat{\delta}_{f_a}^a=0$ 
	otherwise. 
	If $\gamma$ is not a closed curve, let $e_{\mathrm{start}}$ 
	(resp.\ $e_{\mathrm{end}}$) denote the boundary
	edge where $\gamma$ starts (resp.\ ends), 
	and let $s_{e_{\mathrm{start}}}$ (resp.\ $s_{e_{\mathrm{end}}}$)
	be its edge index. 
	Recall that at the starting (ending) point of $\gamma$ 
	the tangent vector of $\gamma$ is parallel to the
	boundary edge.
	Set 
	$\hat{s}_{e_{\mathrm{start}}}:=-s_{e_{\mathrm{start}}}-1$ 
	if
	the edge $e_{\mathrm{start}}$ and 
	the tangent vector point in the same direction and
	$\hat{s}_{e_{\mathrm{start}}}:=s_{e_{\mathrm{start}}}$ 
	otherwise.
	Set 
	$\hat{s}_{e_{\mathrm{end}}}:=s_{e_{\mathrm{end}}}$ 
	if
	the edge $e_{\mathrm{end}}$ and 
	the tangent vector point in the same direction and
	$\hat{s}_{e_{\mathrm{end}}}:=-s_{e_{\mathrm{end}}}-1$ 
	otherwise.

The proof of the following proposition relies on the relation to triangulations 
introduced in Appendix~\ref{app:novak} and is given in Appendix~\ref{app:pf:lem:indexhol}.

\begin{proposition}\label{prop:indexhol}	
Let $\gamma$ be a curve in $\Sigma$.
	Then:
	\begin{enumerate}
		\item 
			If $\gamma$ 
			bounds a disc $\mathbb{D}$ embedded in $\Sigma$,
$\zeta(\gamma)=1$ if $\gamma$ is oriented counter-clockwise around the 
		boundary of $\mathbb{D}$ 
			and $\zeta(\gamma)=-1$ otherwise.
			\label{lem:indexhol1}
		\item 
			If $\gamma'$ is a curve homotopic to $\gamma$ then $\zeta(\gamma')=\zeta(\gamma)$;
			\label{lem:indexhol3}
		\item 
We have
			\begin{align}
				\zeta(\gamma)=
				\sum_{a\in A(\gamma)} (\hat{s}_{e_a}^a+\hat{\delta}_{f_a}^a)+
				\begin{cases}
					0&; \text{$\gamma$ is closed}\ ,\\
						\hat{s}_{e_{\mathrm{start}}}+1&; \text{$\gamma$ is not closed}\ .
				\end{cases}
				\label{eq:holonomy}
			\end{align}
			\label{lem:indexhol2}
	\end{enumerate}
\end{proposition}

Note that in Part~\ref{lem:indexhol2}, in case the curve goes from boundary to boundary, the edge index of the boundary edge
where the endpoint of the curve lies is included in the sum over $A(\gamma)$.

\begin{example}
The admissibility conditions for an inner vertex \eqref{eq:vertexcond1} express that the $r$-spin structure near the vertex is the trivial $r$-spin structure, i.e.\, the unique one extending from the surface to the vertex.
Indeed, pick a curve $\gamma$ oriented counter-clockwise around the vertex. 
By part~\ref{lem:indexhol1} of Proposition~\ref{prop:indexhol}, the holonomy is $\zeta(\gamma)=1$. To match \eqref{eq:holonomy} and \eqref{eq:vertexcond1}, it remains to check that $\sum_{a\in A(\gamma)} \hat{\delta}_{f_a}^a
= N_v-D_v$. This in turn follows since an edge $e_a$ that contributes $+1$ to $D_v$ has $\hat{\delta}_{f_a}^a=0$ and vice versa.
\end{example}

\subsection{Elementary moves on marked PLCW decompositions}\label{sec:eltmove}

In Section~\ref{sec:combsub} we defined the $r$-spin structure $\Sigma(s,\lambda,\mu)$ in terms of a marked PLCW decomposition, and we explained how to change the marking while staying within a given isomorphism class of $r$-spin structures. In this section we state how the marking needs to change when modifying the underlying PLCW decomposition by elementary moves in order to produce isomorphic $r$-spin structures.

	\begin{figure}[tb]
		\centering
		\def\svgwidth{16cm}
		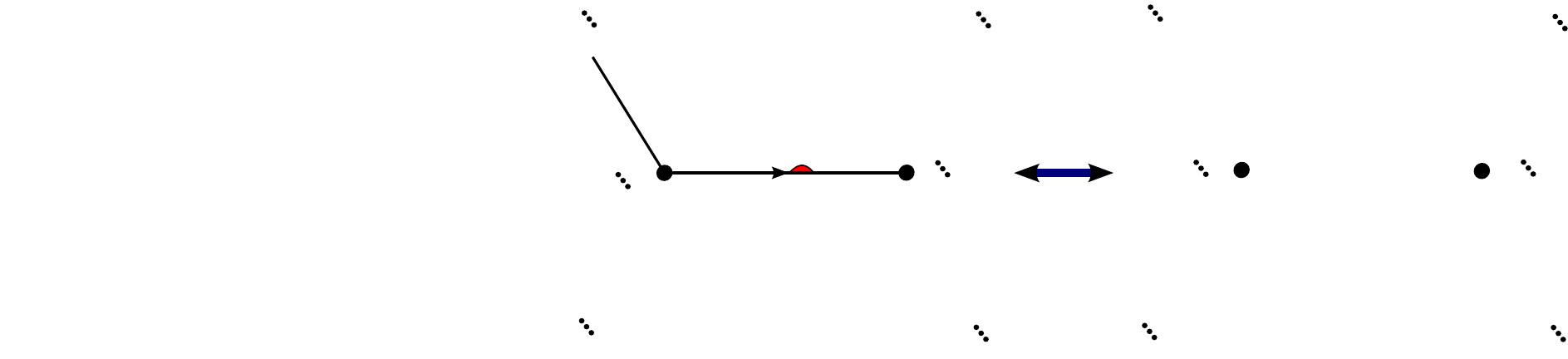
		\caption{Elementary moves of a marked PLCW decomposition.
			Figure~$a)$ shows edges between faces $f$ and $f'$
			(which are allowed to be the same).
			The edges are marked so that
			the vertex $w$ is the clockwise vertex for the face $f$
			(cf.\ Figure~\ref{fig:clockwisevertex}). 
			This convention is not restrictive as one can change
			the orientation of the edges and the markings using 
			Lemma~\ref{lem:moves}.
			In Figure~$b)$, on the left hand side the horizontal edge 
			between the vertices $v$ and $v'$ (which are allowed to be the same)
			is marked for the top polygon, but not for the bottom polygon, 
			and it has edge index 0. For the joint polygon on the right hand side, 
			the marked edge is taken to be that from the bottom polygon on the left.
			Note that this latter convention for the markings is not restrictive,
			as using Lemma~\ref{lem:moves} one can move the markings around.}
		\label{fig:edgemove}
	\end{figure}

\begin{definition}
	An \textsl{elementary move} on a PLCW decomposition of a surface
	is either 
	\begin{itemize}
		\item removing or adding a bivalent vertex as shown in Figure~\ref{fig:edgemove}~$a$), or
		\item removing or adding an edge as shown in Figure~\ref{fig:edgemove}~$b$).
	\end{itemize}
\end{definition}

By \cite[Thm.\,7.4]{Kirillov:2012pl}, any two
PLCW decompositions can be
related by elementary moves. We prove the following proposition in 
Appendix~\ref{app:proofs}.

\begin{proposition} \label{prop:elementarymoves}
	The elementary moves in Figure~\ref{fig:edgemove} 
	induce isomorphisms of $r$-spin structures.
\end{proposition}

	\begin{figure}[tb]
		\centering
		\def\svgwidth{8cm}
		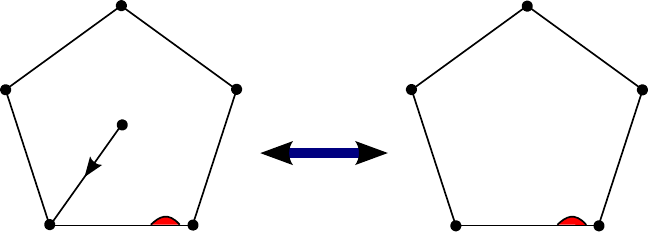
		\caption{The edge index of an univalent vertex is fixed by only the marking and orientation of edges.
		Removing a univalent vertex induces an isomorphism of $r$-spin structures.}
		\label{fig:univalent}
	\end{figure}

The edge index of an edge with a univalent vertex
is fixed by the orientation and the marking of the edge,
in particular it is independent of the rest of the edge indices.
In Lemma~\ref{lem:univalent} we will show that removing univalent vertices
induces an isomorphism of $r$-spin structures.
For an illustration, see Figure~\ref{fig:univalent}.

\subsection{Example: Connected $r$-spin surfaces}\label{sec:sigmagb}

In this section we illustrate how one can use the combinatorial formalism
to count isomorphism classes of $r$-spin structures.
This recovers results obtained in \cite{Randal:2014rs,Geiges:2012rs}
using a different formalism.
\begin{notation}\label{not:edgeindex}
	Whenever it does not cause confusion we will use the same symbols
	for edge labels and for edge indices. For example
	for $e\in\Sigma_1$ we will simply write $e\in\Zb_r$ instead of $s_e\in\Zb_r$.
\end{notation}

\begin{lemma}\label{lem:sphere}
	There exists $r$-spin structures on the sphere if and only if $r=1$ or $r=2$.
	If there exists an $r$-spin structure on the sphere then it is unique up to isomorphism.
\end{lemma}
\begin{proof}
	Let us consider the sphere decomposed into two 1-gons,
	one edge $u$ and one vertex $v$
	as in Figure~\ref{fig:gcc} $c)$, with edge index $u$ 
	(cf. Notations~\ref{not:edgeindex}).
	Let us collect the ingredients for the vertex condition 
	\eqref{eq:vertexcond1}.
	The edge $u$ starts and ends at the vertex, therefore 
		$\hat{u}=-1$ 
	from \eqref{eq:modedgeind}.

	The number $N_v$ of in- and outgoing edges for $v$ is $N_v=1+1=2$, cf.\ \eqref{eq:Nv=Nvstart+Nvend}.
	The number of faces with $v$ in clockwise direction from their marked edge 
	is $D_v=2$,
	since the edge is marked for both faces.
	The vertex condition \eqref{eq:vertexcond1} then reads
	\begin{align*}
		-1\equiv 2-2+1&&\Mod{r},
	\end{align*}
	which holds if and only if $r=1$ or $r=2$.
	The edge index $u$ can be set arbitrarily
	by operation~\ref{lem:moves:3} in Lemma~\ref{lem:moves},
	and together with Remark~\ref{rem:rscomb} we see that
	for any two values of $u$ 
	the $r$-spin structures on the sphere are isomorphic.
\end{proof}

\begin{proposition}\label{prop:sigmagb}
	Let $\Sigma_{g,b}$ be a connected surface
	of genus $g$ with $b$ boundary components 
	and with maps $\lambda$ and $\mu$. 
	There exists an $r$-spin structure on $\Sigma_{g,b}$ if and only if
	\begin{align}
		\chi(\Sigma_{g,b}) \equiv& \sum_{u\in\pi_0(\partial\Sigma)}R_u&\Mod{r},
		\label{eq:prop:sigmagb}
	\end{align}
	where $\chi(\Sigma_{g,b})=2-2g-b$ denotes the Euler characteristic
	and $R_{u}$ was defined in \eqref{eq:bdrysign}. 
If \eqref{eq:prop:sigmagb} holds, 
the number $|\Rc^r(\Sigma_{g,b})_{\lambda,\mu}|$ of isomorphism classes of $r$-spin structures on $\Sigma_{g,b}$ is
given by:
\begin{center}
	\begin{tabular}{c|c|Sc}
		$r$&$b,g$&$|\Rc^r(\Sigma_{g,b})_{\lambda,\mu}|$\\
		\hline
		0&  $g=0$ and $b=1$
		&1\\
		 & else & infinite\\
		\hline
		$>0$&$b=0$& $r^{2g}$\\
		  &$b\ge1$&$r^{2g+b-1}$\\
	\end{tabular}
\end{center}
\end{proposition}

\begin{figure}[tb]
	\centering
	\def\svgwidth{15cm}
	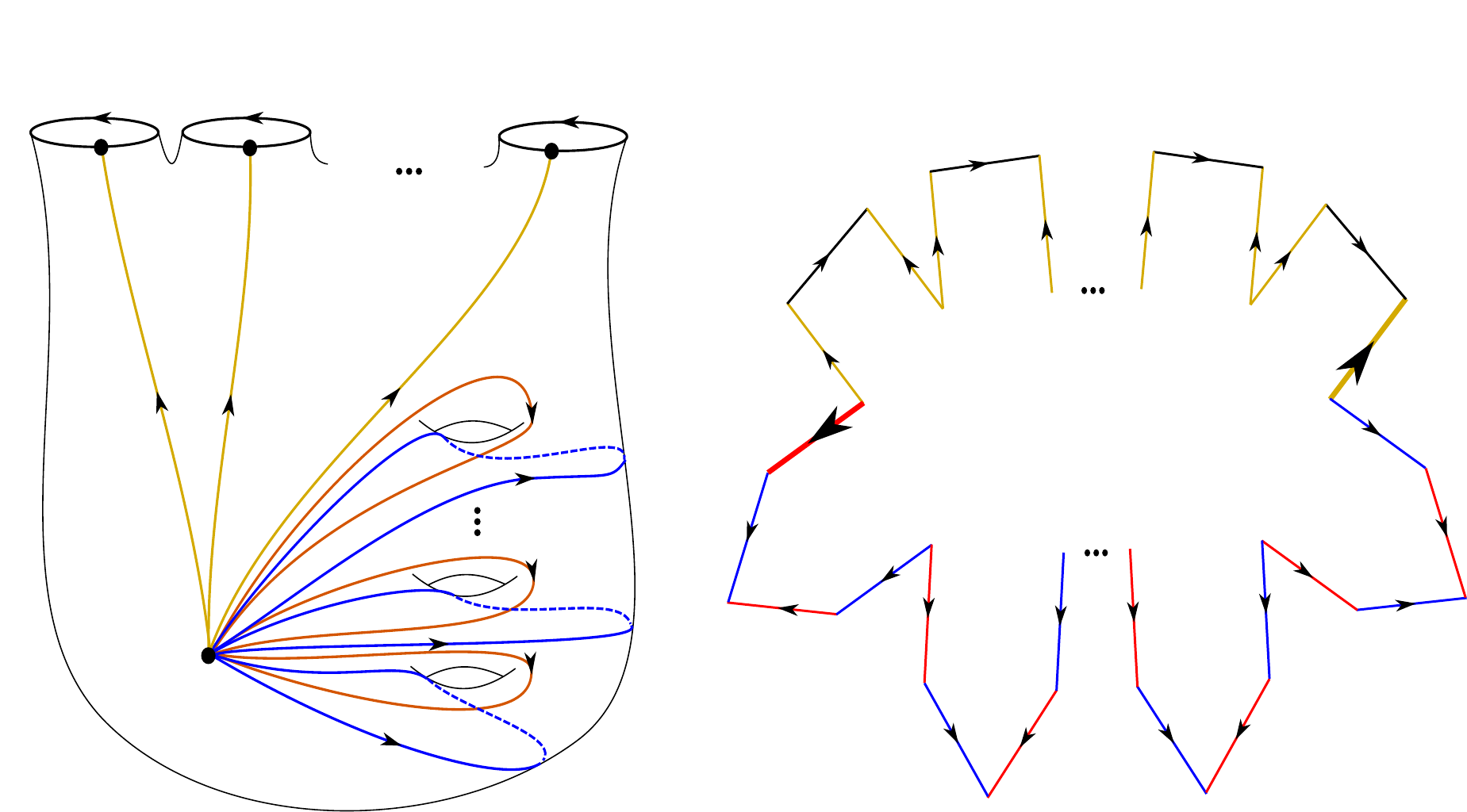
	\caption{ PLCW decomposition of $\Sigma_{g,b}$ for $g+b \ge 1$ 
	using only one face, shown after glueing (Fig.\,a) and before glueing (Fig.\,b) -- 
	the edges labeled with the same symbols are identified.
	In Fig.\,b, the bigger arrows indicate the marked edge, namely edge 
$r_1$ in case $g=0$ (and so $b>0$) and edge $s_1$ in case $g>0$.
      }
	\label{fig:sigmagb_param}
\end{figure}

A similar result has been obtained for the existence of $r$-spin structures 
on closed hyperbolic orbifolds for $r>0$ in \cite[Thm.\,3]{Geiges:2012rs}.
Note that in complex geometry, \eqref{eq:prop:sigmagb} (for $r>0$ and $b=0$) 
is just the condition for the existence of an $r$-th root of the canonical line bundle
(see e.g.~\cite{Witten:1993mm}).

\begin{proof}
The case $g=b=0$ has been discussed in Lemma~\ref{lem:sphere}, so we can assume $g+b\ge1$.
Decompose $\Sigma_{g,b}$ into a $(4g+3b)$-gon
consisting of $2g+2b$ inner edges, $b$ boundary edges, one inner vertex $v_0$ 
and $b$ boundary vertices $v_j$, $j=1,\dots,b$,
as shown in Figure~\ref{fig:sigmagb_param}~$a)$~and~$b)$.
Assign the edge indices $s_i$, $t_i$, $r_j$ and $u_j$, where $i=1,\dots,g$ and $j=1,\dots,b$.
Mark the edge $s_1$ if $g\neq0$ or the edge $r_1$ if $g=0$, 
see Figure~\ref{fig:sigmagb_param}\,$b)$.

We now evaluate the admissibility condition at each vertex.
For the boundary vertex $v_j$ there is the incoming inner
edge $r_j$ 
and the boundary edge $u_j$
which starts and ends at the same vertex $v_j$. Therefore by \eqref{eq:modedgeind}, 
relative to $v_j$ one has $\hat{r}_j=-r_j-1$ and $\hat{u}_j=-1$. 
For either of the two above mentioned markings 
$D_{v_j}=0$ and $N_{v_j}=3$, therefore we have
\begin{align}
	-r_j-1-1&\equiv 0-3+1-R_{u_j}&\Mod{r}&&\text{for $j=1,\dots,b$}.
	\label{eq:vcondbdry}
\end{align}
Thus the $r_j$ are uniquely fixed by the boundary parametrisation $\lambda,\mu$ to be $r_j \equiv R_{u_j} \Mod r$ for all $j$.

For the inner vertex $v_0$ there are $b$
edges leaving the vertex
and $2g$ edges which start and end there. Therefore by \eqref{eq:modedgeind},
relative to $v_0$ one has
$\hat{r}_j=r_j$ and $\hat{s}_i=\hat{t}_i=-1$. $D_{v_0}=1$ and $N_{v_0}=4g+b$,
and so
\begin{align}
	\sum_{j=1}^{b}r_j-2g&\equiv 1-(4g+b)+1&\Mod{r}\ .
	\label{eq:vcondmiddle}
\end{align}
Combining \eqref{eq:vcondbdry} and \eqref{eq:vcondmiddle} 
one obtains \eqref{eq:prop:sigmagb}.

By Remark~\ref{rem:rscomb}, for a fixed marking and orientation, 
edge index assignments up to deck transformations are in bijection with 
$r$-spin structures. From \eqref{eq:vcondbdry} and \eqref{eq:vcondmiddle}
every $(s_i,t_i,u_j)\in\left( \Zb_r \right)^{2g+b}$ gives an admissible edge index assignment.
A deck transformation on the face shifts the $u_j$ parameters simultaneously
and leaves the $s_i$ and $t_i$ parameters fixed. By 
a simple counting we get the number of isomorphism classes of $r$-spin structures.
\end{proof}

\begin{corollary}\label{cor:disk}
There is a unique $r$-spin structure on the disk with boundary condition $\lambda=2$ (ingoing boundary) or $\lambda=0$ (outgoing boundary), and no $r$-spin structure else.
\end{corollary}

	Let $R_j:=R_{u_j}$ from \eqref{eq:bdrysign} and
	let us denote the $r$-spin structure on $\Sigma_{g,b}$ 
	given by the parameters $s_i, t_i, u_j\in\Zb_r$ 
	for $i=1,\dots,g$ and $j=1,\dots,b$ from
	Figure~\ref{fig:sigmagb_param}
	by 
\begin{align}
\label{eq:sigmagb-def}
	\Sigma_{g,b}(s_i, t_i, u_j,R_j) 
\end{align}
(and recall from 
Notation~\ref{not:edgeindex} that the same symbols denote 
edges and the assigned edge indices).

\begin{figure}[tb]
	\centering
	\def\svgwidth{16cm}
	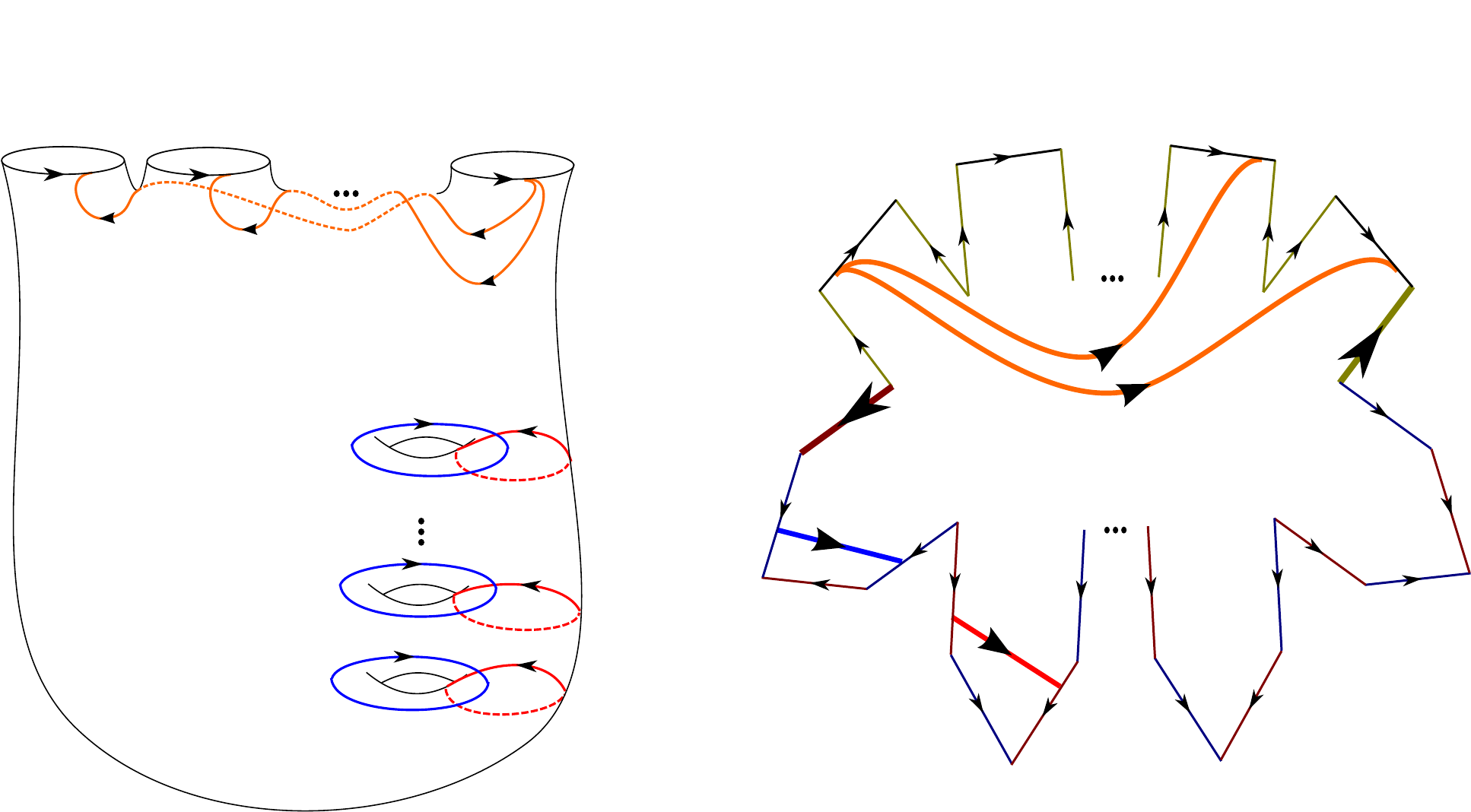
	\caption{$a)$ A set of curves in $\Sigma_{g,b}$: $\setc*{a_i,b_i,c_j}{i=1,\dots,g;j=1,\dots,b-1}$.
	$b)$ These curves in the PLCW decomposition of $\Sigma_{g,b}$ 
	(cf.\ Figure~\ref{fig:sigmagb_param}). 
		The bigger arrows on the edges
		show the marked edge: $r_1$ for $b>0$ and $s_1$ for $b=0$.}
	\label{fig:sigmagb_arcs}
\end{figure}

\begin{proposition}\label{prop:indexhol-sigma-bg}
	The holonomies of the curves in $\Sigma_{g,b}$ shown
	in Figure~\ref{fig:sigmagb_arcs}~$a)$ are
	\begin{align}
		\begin{aligned}
		\zeta(a_i)&=s_i,&
		\zeta(b_i)&=t_i,&
		\zeta(c_j)&=
			u_j-u_b+1,&
		\zeta(\partial_j)&=1-\lambda_j,
	\end{aligned}\label{eq:lem:indexhol}
	\end{align}
	for $i=1,\dots,g$ and $j=1,\dots,b-1$.
\end{proposition}
\begin{proof}
	We only show the calculation of the latter two holonomies.
Consider the single face $f$ shown in Figure~\ref{fig:sigmagb_arcs}~$b)$.
	The tangent vectors of the
	edge $u_b$ and the loop $c_j$ 
	point in the same direction and the loop starts at this edge
	($u_b=e_{\mathrm{start}}$),	
	therefore 
		$\hat{s}_{u_b}^f=\hat{s}_{e_{\mathrm{start}}}=-u_b-1$;
	the tangent vectors of 
	the edge $u_j$ and the loop $c_j$
	point in the same direction and the loop ends at this edge	
	($u_j=e_{\mathrm{end}}$),	
	therefore 
		$\hat{s}_{u_j}^f=\hat{s}_{e_{\mathrm{end}}}=u_j$;
	the clockwise vertex determined by the marked edge of the face $f$ is on the right 
	side of the curve $c_j$ so $\hat{\delta}_{f}^{c_j}=1$. Taking the sum of all these we get
	$\zeta(c_j)=u_j+1-u_b-1+1=u_j-u_b+1$.
	The edge $r_j$ and the loop $\partial_j$ cross negatively and the clockwise vertex is on the right 
	side of the loop, so we get $\zeta(\partial_j)=-r_j-1+1=1-\lambda_j$.
\end{proof}

\section{State-sum construction of \texorpdfstring{$r$}{r}-spin TFTs}\label{sec:tft}

Our first application of the combinatorial description of $r$-spin structures 
is a state-sum construction of $r$-spin TFTs,
see \cite{Barrett:2013sp,Novak:2014sp,Gaiotto:2016spin} for the $2$-spin case and \cite{Novak:2015phd} for general $r$.
We generalise the construction in \cite{Novak:2015phd} from triangulations to PLCW-decompositions, 
which are much more convenient for explicit computations. 
In Sections~\ref{sec:alg-prelim-statesumTFT}~and~\ref{sec:gradedcenter} 
we present some algebraic preliminaries, and in Section~\ref{sec:state-sum-constr} 
we explain how suitable Frobenius algebras produce an $r$-spin TFT via a 
state-sum construction (Theorem~\ref{thm:tft}). In Section~\ref{sec:evaluation} 
we compute the value of the state-sum TFT on connected $r$-spin bordisms.

\subsection{Algebraic notions}\label{sec:alg-prelim-statesumTFT}

\begin{figure}[tb]
	\centering
	\def\svgwidth{7cm}
	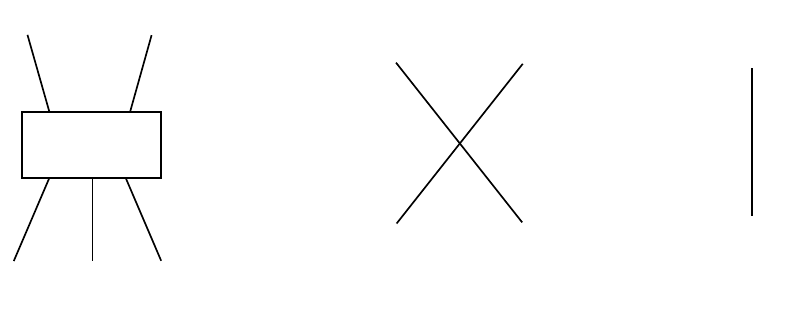
	\caption{String diagram notion of a morphism $f\in\Sc(A\otimes B \otimes C,D\otimes E)$, 
		the symmetric braiding $c_{A,B}$ and the identity $\id_{A}$.}
	\label{fig:f_braiding}
\end{figure}

\begin{figure}[tb]
	\centering
	\def\svgwidth{8cm}
	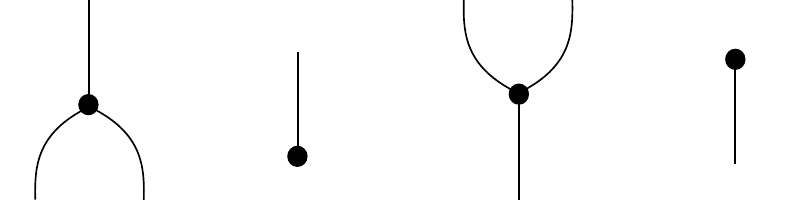
	\caption{String diagrams we will use for the structure morphisms of a Frobenius algebra.}
	\label{fig:frob_morph}
\end{figure}

Let $\Sc$ denote a strict symmetric monoidal category with tensor product $\otimes$,
tensor unit $\Ib$ and braiding $c$.
We use the graphical calculus as shown in Figure~\ref{fig:f_braiding},
and we will omit the labels for objects if they are understood, as e.g.\ in Figure~\ref{fig:frob_morph}.

An object $A\in\Sc$ together with morphisms
$\mu\in\Sc(A\otimes A, A)$ (multiplication), $\eta\in\Sc(\Ib,A)$ (unit), 
$\Delta\in\Sc(A,A\otimes A)$ (comultiplication) and $\eps\in\Sc(A,\Ib)$ (counit), see Figure~\ref{fig:frob_morph}, is a \textsl{Frobenius algebra} if the following relations hold:
\begin{align}
	\begin{aligned}
	\def\svgwidth{6.5cm}
\begingroup%
  \makeatletter%
  \providecommand\color[2][]{%
    \errmessage{(Inkscape) Color is used for the text in Inkscape, but the package 'color.sty' is not loaded}%
    \renewcommand\color[2][]{}%
  }%
  \providecommand\transparent[1]{%
    \errmessage{(Inkscape) Transparency is used (non-zero) for the text in Inkscape, but the package 'transparent.sty' is not loaded}%
    \renewcommand\transparent[1]{}%
  }%
  \providecommand\rotatebox[2]{#2}%
  \ifx\svgwidth\undefined%
    \setlength{\unitlength}{325.07521897bp}%
    \ifx\svgscale\undefined%
      \relax%
    \else%
      \setlength{\unitlength}{\unitlength * \real{\svgscale}}%
    \fi%
  \else%
    \setlength{\unitlength}{\svgwidth}%
  \fi%
  \global\let\svgwidth\undefined%
  \global\let\svgscale\undefined%
  \makeatother%
  \begin{picture}(1,0.32188026)%
    \put(0,0){\includegraphics[width=\unitlength]{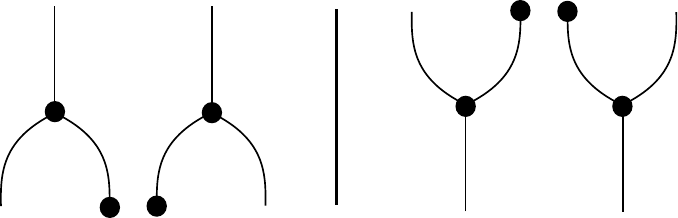}}%
    \put(0.17996867,0.15583766){\color[rgb]{0,0,0}\makebox(0,0)[lb]{\smash{=}}}%
    \put(0.39443207,0.15583766){\color[rgb]{0,0,0}\makebox(0,0)[lb]{\smash{=}}}%
    \put(0.79209679,0.15583766){\color[rgb]{0,0,0}\makebox(0,0)[lb]{\smash{=}}}%
    \put(0.56008363,0.15583766){\color[rgb]{0,0,0}\makebox(0,0)[lb]{\smash{=}}}%
  \end{picture}%
\endgroup%

	\label{eq:unit_counit}
	\end{aligned}\\
	\begin{aligned}
	\def\svgwidth{6.5cm}
\begingroup%
  \makeatletter%
  \providecommand\color[2][]{%
    \errmessage{(Inkscape) Color is used for the text in Inkscape, but the package 'color.sty' is not loaded}%
    \renewcommand\color[2][]{}%
  }%
  \providecommand\transparent[1]{%
    \errmessage{(Inkscape) Transparency is used (non-zero) for the text in Inkscape, but the package 'transparent.sty' is not loaded}%
    \renewcommand\transparent[1]{}%
  }%
  \providecommand\rotatebox[2]{#2}%
  \ifx\svgwidth\undefined%
    \setlength{\unitlength}{340.47649238bp}%
    \ifx\svgscale\undefined%
      \relax%
    \else%
      \setlength{\unitlength}{\unitlength * \real{\svgscale}}%
    \fi%
  \else%
    \setlength{\unitlength}{\svgwidth}%
  \fi%
  \global\let\svgwidth\undefined%
  \global\let\svgscale\undefined%
  \makeatother%
  \begin{picture}(1,0.41495524)%
    \put(0,0){\includegraphics[width=\unitlength]{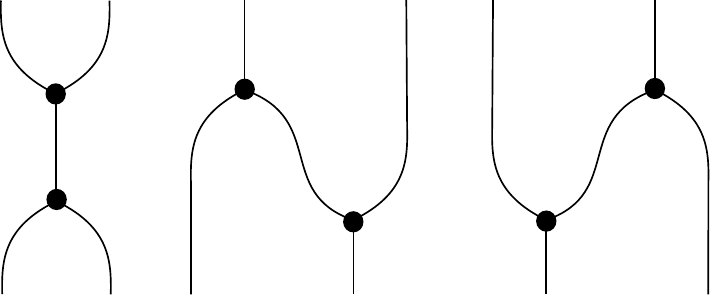}}%
    \put(0.17793549,0.19314058){\color[rgb]{0,0,0}\makebox(0,0)[lb]{\smash{=}}}%
    \put(0.61511215,0.19314058){\color[rgb]{0,0,0}\makebox(0,0)[lb]{\smash{=}}}%
  \end{picture}%
\endgroup%

	\label{eq:frob_rel}
	\end{aligned}
\end{align}
These relations imply that a Frobenius algebra is in particular an associative 	algebra
and coassociative coalgebra,
see \cite[Prop.\,2.3.24]{Kock:2004fa}.
For more details on the definition of algebras, coalgebras and Frobenius algebras in monoidal categories we refer to e.g.\ \cite{Fuchs:2008fa}.

\begin{figure}[tb]
	\centering
	\def\svgwidth{10cm}
	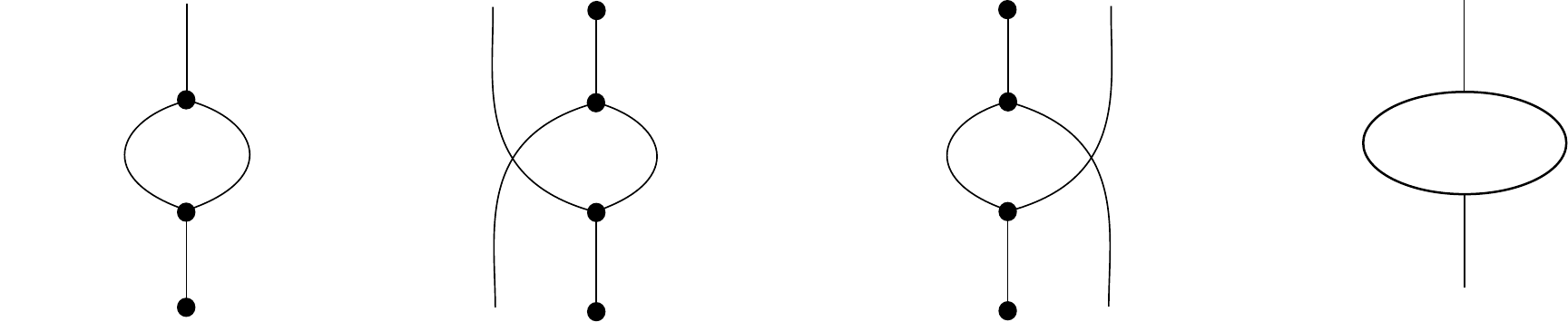
	\caption{The window element $\tau$, the Nakayama automorphism $N$, its inverse $N^{-1}$
	\cite[Sec.\,5.3]{Novak:2015phd}
		and our string diagram abbreviation for the $k$'th power of $N$.}
	\label{fig:window_nakayama}
\end{figure}

For a Frobenius algebra $A$ we define the \textsl{window element} 
$\tau=\mu\circ\Delta\circ\eta$ \cite{Lauda:2007oc} and the 
\textsl{ Nakayama automorphism}
$N=\left( \id_{A}\otimes (\eps\circ\mu) \right)\circ \left(c_{A,A}\otimes\id_{A}\right)
\circ\left( \id_A\otimes(\Delta\circ\eta) \right)$,
see Figure~\ref{fig:window_nakayama}.
Then $\tau$ is central (as follows from an easy calculation)
and $N$ is a morphism of Frobenius algebras 
(see \cite{Fuchs:2008fa} and \cite[Prop.\,4.5]{Novak:2014sp}).
$A$ is called \textsl{symmetric} if $\varepsilon\circ\mu\circ c_{A,A}=\varepsilon\circ\mu$.
	It can be shown from a straightforward calculation that $A$ is symmetric if and only if $N=\id_A$.

\begin{lemma} Let $A\in\Sc$ be a Frobenius algebra with Nakayama automorphism $N$.
	Then for every $n\in\Zb$
		\begin{align}
			\begin{aligned}
				\def\svgwidth{9cm}
\begingroup%
  \makeatletter%
  \providecommand\color[2][]{%
    \errmessage{(Inkscape) Color is used for the text in Inkscape, but the package 'color.sty' is not loaded}%
    \renewcommand\color[2][]{}%
  }%
  \providecommand\transparent[1]{%
    \errmessage{(Inkscape) Transparency is used (non-zero) for the text in Inkscape, but the package 'transparent.sty' is not loaded}%
    \renewcommand\transparent[1]{}%
  }%
  \providecommand\rotatebox[2]{#2}%
  \ifx\svgwidth\undefined%
    \setlength{\unitlength}{468.22349292bp}%
    \ifx\svgscale\undefined%
      \relax%
    \else%
      \setlength{\unitlength}{\unitlength * \real{\svgscale}}%
    \fi%
  \else%
    \setlength{\unitlength}{\svgwidth}%
  \fi%
  \global\let\svgwidth\undefined%
  \global\let\svgscale\undefined%
  \makeatother%
  \begin{picture}(1,0.44989261)%
    \put(0,0){\includegraphics[width=\unitlength]{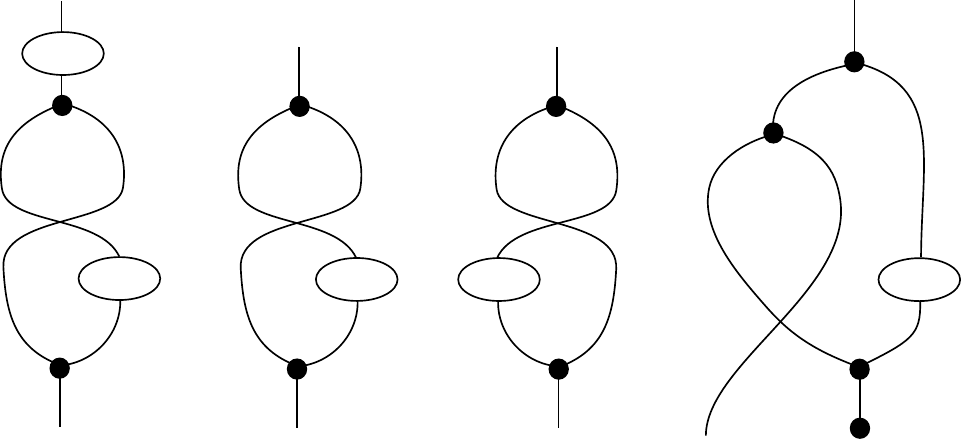}}%
    \put(0.35873861,0.15280151){\color[rgb]{0,0,0}\makebox(0,0)[lb]{\smash{$n$}}}%
    \put(0.41221271,0.2187254){\color[rgb]{0,0,0}\makebox(0,0)[lb]{\smash{$=$}}}%
    \put(0.48325646,0.15324201){\color[rgb]{0,0,0}\makebox(0,0)[lb]{\smash{$-n$}}}%
    \put(0.11207393,0.1542461){\color[rgb]{0,0,0}\makebox(0,0)[lb]{\smash{$n$}}}%
    \put(0.16890364,0.21965186){\color[rgb]{0,0,0}\makebox(0,0)[lb]{\smash{$=$}}}%
    \put(0.05566595,0.38492599){\color[rgb]{0,0,0}\makebox(0,0)[lb]{\smash{$n$}}}%
    \put(0.65109628,0.21849301){\color[rgb]{0,0,0}\makebox(0,0)[lb]{\smash{$=$}}}%
    \put(0.93141129,0.15314565){\color[rgb]{0,0,0}\makebox(0,0)[lb]{\smash{$n$}}}%
  \end{picture}%
\endgroup%

				\label{eq:plem}
			\end{aligned}
		\end{align}
	\label{lem:plem}
\end{lemma}
\begin{proof}
The first and second equations
are proven in \cite[Lem.\,5.12]{Novak:2015phd}.
	The third equation follows from a direct calculation.
\end{proof}

A morphism $\kappa : \Ib \to A$ is called \textsl{invertible} 
if there is a morphism $\kappa' : \Ib \to A$ such that 
$\mu \circ (\kappa \otimes \kappa') = \eta = \mu \circ (\kappa' \otimes \kappa)$. 
In this case we write $\kappa^{-1}$ instead of $\kappa'$ for the unique inverse.

Let $r\in\Zb_{\ge0}$, $(A,\mu,\eta,\Delta,\varepsilon)$ be a Frobenius algebra in $\Sc$ 
with invertible window element $\tau$ 
and with Nakayama automorphism $N$,
such that $N^r=\id_A$ (for $r=0$ this last condition is empty). 
A Frobenius algebra with $N^r=\id_A$
is called a \textsl{$\Lambda_r$-Frobenius algebra} 
in \cite[Prop.\,I.41]{Dyckerhoff:2015csg}. 
Define 
\begin{align}
	P_\lambda:=(\tau^{-1}\cdot(-))\circ\mu\circ c_{A,A}\circ(\id\otimes N^{1-\lambda})\circ\Delta
	\quad \in \End(A) \ .
	\label{eq:pdef}
\end{align}
We collect some properties of $P_\lambda$ in the following lemma.

\begin{lemma} \label{lem:gradedfa1}
$P_{\lambda}$ is an idempotent, and
	for any $\lambda_1,\lambda_2\in\Zb_r$ one has that:
	\begin{align}
	N \circ P_{\lambda_1} &= P_{\lambda_1} \circ N
		\label{eq:gradednaka} 
		\\
		\mu\circ\left( P_{\lambda_1}\otimes P_{\lambda_2} \right)&=
		P_{\lambda_1+\lambda_2}\circ\mu\circ\left( P_{\lambda_1}\otimes P_{\lambda_2} \right) \ ,
		&
		\eta&=P_0\circ\eta \ ,
		\label{eq:gradedmultunit1} \\
		\left( P_{\lambda_1}\otimes P_{\lambda_2} \right)\circ\Delta&=
		\left( P_{\lambda_1}\otimes P_{\lambda_2} \right)\circ\Delta\circ P_{\lambda_1+\lambda_2-2}  \ ,
	&
		\eps&=\eps\circ P_2 \ .
		\label{eq:gradedcomultunit1}
	\end{align}
\end{lemma}

\begin{proof}
That $P_\lambda$ is an idempotent is a direct generalisation of \cite[Lem.\,5.12\,(1)]{Novak:2015phd}.
The additional $(\tau^{-1}\cdot(-))$ removes the ``bubble'' $\mu\circ\Delta$.
The identity \eqref{eq:gradednaka} is immediate from the definition of $P_\lambda$ in \eqref{eq:pdef} and the fact that $N$ is an automorphism of Frobenius algebras.
The first identity in \eqref{eq:gradedmultunit1} is a more general version of \cite[Lem.\,6.8]{Novak:2015phd} and the proof works along the same lines. To 
show the second identity in each of 
\eqref{eq:gradedmultunit1} and \eqref{eq:gradedcomultunit1} just write out the definition~of $P_{\lambda}$, $N$ and $N^{-1}$. 
For the first identity in \eqref{eq:gradedcomultunit1} use \eqref{eq:gradedmultunit1} together with  
	\begin{align}
		\left( (\eps\circ\mu)\otimes\id_{A}\right)\circ
		\left( \id_{A}\otimes P_{\lambda}\otimes\id_{A} \right)\circ
		\left( \id_{A}\otimes(\Delta\circ\eta) \right)=P_{2-\lambda},
		\label{eq:pdual}
	\end{align}
	which follows from a direct calculation.
\end{proof}

\subsection{The $\Zb_r$-graded center}\label{sec:gradedcenter}

Let $A \in \Sc$ be a Frobenius algebra with invertible window element.
Let $\Sc$ be furthermore additive 
(in particular, finite direct sums distribute over tensor products)
and assume that the idempotents $P_\lambda$ split in $\Sc$, i.e.\
\begin{align}
	&P_\lambda=
	\left[ A\xrightarrow{\pi_{\lambda}}Z_{\lambda}\xrightarrow{\iota_{\lambda}}A \right]\ , 
	&& \left[ Z_{\lambda}\xrightarrow{\iota_{\lambda}}A\xrightarrow{\pi_{\lambda}}Z_{\lambda} \right] =\id_{Z_{\lambda}}\ ,
	\label{eq:splitidempot}
\end{align}
for some object $Z_{\lambda}\in\Sc$. For $r=0$ assume furthermore that $\Sc$ has countably infinite direct sums which distribute over the tensor product.
We can now define:

\begin{definition}\label{def:gradedcenter}
	Let $r \in \Zb_{\ge 0}$.  The \textsl{$\Zb_r$-graded center} of 
	a Frobenius algebra $A$ with invertible window element and which satisfies $N^r=\id$
	is the direct sum
\begin{align}
	Z^r(A):=\bigoplus_{\lambda\in\Zb_r}Z_{\lambda}\ .
	\label{eq:def:gradedcenter}
\end{align}
This is a $\Zb_r$-graded object and we call $\lambda\in\Zb_r$ the degree of $Z_{\lambda}$.
\end{definition}

Next we will endow $Z^r(A)$ with an algebra structure induced by $A$.
Write
\begin{align}
	e_{\lambda}:Z_{\lambda}\to Z^r(A)
	\label{eq:embeddings}
\end{align}
for the embeddings of the summands in \eqref{eq:embeddings} and $p_{\lambda}$
for the induced projections which satisfy
\begin{align}
	\left[ Z_{\lambda_1}\xrightarrow{e_{\lambda_1}}Z^r(A)\xrightarrow{p_{\lambda_2}}Z_{\lambda_2} \right]
	=\delta_{\lambda_1,\lambda_2}\id_{Z_{\lambda_1}}\ .
	\label{eq:orthproj}
\end{align}
Lemma~\ref{lem:gradedfa1} suggests to define, for $\lambda_1,\lambda_2\in\Zb_r$,
\begin{align}
	\mu_{\lambda_1,\lambda_2}&:=
	\left[ Z_{\lambda_1}\otimes Z_{\lambda_2} \xrightarrow{\iota_{\lambda_1}\otimes\iota_{\lambda_2}}\right.
	A\otimes A \xrightarrow{\mu} A
	\left.\xrightarrow{\pi_{\lambda_1+\lambda_2}} Z_{\lambda_1+\lambda_2} \right]\ .
	\label{eq:gradedmult3}
\end{align}
By the universal property of direct sums (which in the countably infinite case for $r=0$ still distribute over $\otimes$ by our assumptions) 
there is a unique map 
\begin{align}
	\bar{\mu} : Z^r(A)\otimes Z^r(A) \longrightarrow Z^r(A) 
\end{align}
which satisfies $\bar\mu \circ (e_{\lambda_1} \otimes e_{\lambda_2}) = e_{\lambda_1+\lambda_2} \circ \mu_{\lambda_1,\lambda_2}$.
Let us furthermore define
\begin{align}
	\bar{\eta} :=
	\left[ \Ib\xrightarrow{\eta}A\xrightarrow{\pi_0} Z_0\xrightarrow{e_0} Z^r(A) \right]\ ,
	\label{eq:gradedunit2}
\end{align}
The morphisms $\bar{\mu}$ and $\bar{\eta}$ are degree preserving.
It is straightforward to verify that $Z^r(A)$ together with $\bar\mu$ and $\bar\eta$ 
becomes an associative unital $\Zb_r$-graded algebra in $\Sc$.

One can restrict the Nakayama automorphism of $A$ on the $Z_{\lambda}$'s by
\begin{align}
	N_{Z_{\lambda}}:=
	\left[ Z_{\lambda}\xrightarrow{\iota_{\lambda}}A\xrightarrow{N}A\xrightarrow{\pi_{\lambda}} Z_{\lambda} \right]
	\label{eq:nrestricted}
\end{align}
As in \cite[Lem.\,5.12/3]{Novak:2015phd} on verifies that
\begin{align}
		N_{Z_{\lambda}}^{\mathrm{gcd}(1-\lambda,r)}=\id_{Z_{\lambda}} \ .
		\label{eq:norder}
\end{align}
Recall from the introduction that $\mathrm{gcd}(a,b)$ denotes 
the non-negative generator of the ideal $\langle a,b \rangle \subset \Zb$. 
In particular, for $r=0$ we have 
$\mathrm{gcd}(1-\lambda,r)=|1-\lambda|$. 
The product $\bar\mu$ is in general not commutative, 
but a simple computation shows that its components satisfy:
	\begin{align}
		\begin{aligned}
		\mu_{\lambda_1,\lambda_2}\circ c_{Z_{\lambda_2},Z_{\lambda_1}}
		&=\mu_{\lambda_2,\lambda_1}
		\circ\left( N_{\lambda_2}^{-\lambda_1} \otimes\id_{Z_{\lambda_1}} \right)\\
		&=\mu_{\lambda_2,\lambda_1}
		\circ\left( \id_{Z_{\lambda_2}}\otimes N_{\lambda_1}^{+\lambda_2} \right)\ .
		\end{aligned}
		\label{eq:commutativity}
	\end{align}

Let $\bar N : Z^r(A) \to Z^r(A)$ be the unique morphism such that $\bar N \circ e_\lambda = e_\lambda \circ N_\lambda$ for all $\lambda$. Combining the fact that $N$ is an automorphism of Frobenius algebras with the definition of $\bar\mu$ and $\bar\eta$ and using \eqref{eq:gradednaka} shows that $\bar N$ is an algebra automorphism. By \eqref{eq:norder} we have ${\bar N}^r=\id$.
We collect the above results in the following proposition.

\begin{proposition}
Let $A$ be as in Definition~\ref{def:gradedcenter}.
The $\Zb_r$-graded center $Z^r(A)$ of $A$ is an associative unital algebra via $\bar\mu$, $\bar\eta$ and is equipped with the algebra automorphism $\bar N$ satisfying ${\bar N}^r=\id$. The algebra $Z^r(A)$ satisfies the commutativity conditions, for $\lambda \in \Zb_r$,
\begin{align}
	\begin{aligned}
	\bar\mu \circ c_{Z^r(A),Z^r(A)} \circ (\id \otimes e_\lambda)
	&=
	\bar\mu \circ ({\bar N}^{-\lambda} \otimes e_\lambda)
	\ ,
	\\
	\bar\mu \circ c_{Z^r(A),Z^r(A)} \circ (e_\lambda \otimes \id)
	&=
	\bar\mu \circ (e_\lambda \otimes {\bar N}^\lambda)
	\ .
	\end{aligned}
\end{align}
\end{proposition}

\begin{corollary} \label{cor:centerproj}
	The component $Z_0$ of $Z^r(A)$ is a subalgebra and is the centre of $A$.
\end{corollary}

\subsubsection*{Frobenius algebra structure for $r>0$}

For the rest of this section let us assume that $r>0$.
Since now $\Zb_r$ is finite, we can define the coproduct as the sum
\begin{align}
	\bar{\Delta}&:=\sum_{\lambda_1,\lambda_2\in\Zb_r}
	\left[ Z^r(A)\xrightarrow{p_{\lambda_1+\lambda_2-2}}Z_{\lambda_1+\lambda_2-2}\right.
	\xrightarrow{\Delta_{\lambda_1,\lambda_2}} Z_{\lambda_1}\otimes Z_{\lambda_2}
	\left.\xrightarrow{e_{\lambda_1}\otimes e_{\lambda_2}} Z^r(A)\otimes Z^r(A) \right]\ ,
	\label{eq:gradedcomult2}
\end{align}
with component maps
\begin{align}
	\Delta_{\lambda_1,\lambda_2}&:=
	\left[ Z_{\lambda_1+\lambda_2-2}\xrightarrow{\iota_{\lambda_1+\lambda_2-2}} \right.
		A\xrightarrow{\Delta\circ(\tau\cdot(-))} A\otimes A
	\left.  \xrightarrow{\pi_{\lambda_1}\otimes\pi_{\lambda_2}} Z_{\lambda_2}\otimes Z_{\lambda_1} \right]\ .
	\label{eq:gradedcomult3}
\end{align}
We define the counit
\begin{align}
	\bar{\eps} :=
	\left[ Z^r(A) \xrightarrow{p_{2}}Z_{2}\xrightarrow{\iota_2}A\xrightarrow{(\tau^{-1}\cdot(-))}A\xrightarrow{\varepsilon}\Ib \right]\ .
	\label{eq:gradedcounit2}
\end{align}
The morphisms $\bar{\Delta}$ and $\bar{\varepsilon}$ have degree +2 and -2 respectively.
Note that we inserted a multiplication with $\tau$ and its inverse 
in the definition of $\bar{\eps}$ and $\bar{\Delta}$. 
The reason for this is that we want these maps to match 
the structure maps calculated in Section~\ref{sec:evaluation}
from the state-sum construction.

It is straightforward to see that altogether $Z^r(A)$ becomes a Frobenius algebra, just verify 
\eqref{eq:unit_counit} and 
\eqref{eq:frob_rel}
restricted to individual summands of $Z^r(A)$ by using Lemma~\ref{lem:gradedfa1} to move projectors past structure maps of $A$ and by the properties of $A$ itself.
Altogether we have:

\begin{proposition}
	Let $A$ be as in Definition~\ref{def:gradedcenter}.
For $r>0$, the $\Zb_r$ graded center of $A$ together with
	$\bar{\mu}$, $\bar{\eta}$, $\bar{\Delta}$, $\bar{\varepsilon}$
	is a $\Zb_r$-graded Frobenius algebra.
	The morphisms $\bar{\mu}$ and $\bar{\eta}$ have degree 0, 
	while $\bar{\Delta}$ has degree $2$ and $\bar{\eps}$ has degree $-2$.
\end{proposition}

\begin{remark}\label{rem:fullcenter}
	\begin{enumerate}
		\item 
	The condition $N^r=\id_A$ amounts to $A$ being a representation of
	the group $\Zb_r$. Instead of defining this in a general category,
	let $k$ be a field and let us assume that 
	$A\in\Rep_k(\Zb_r)$, the category of $k$-linear 
	representations of $\Zb_r$.
	Then the algebra $Z^r(A)$ is the full center of $A$ as
	defined in \cite{Davydov:2010fc}, and is in particular
	a commutative algebra in
	 $\Zc(\Rep_k(\Zb_r))$, 
	the monoidal center of $\Rep_k(\Zb_r)$.
	To see this one needs to check that $Z^r(A)$ has the form of 
	the full center as given in
	\cite[Prop.\,9.6]{Davydov:2010fc},
	which has been done in \eqref{eq:commutativity}.

	Note, however, that unless $r=1$ or $r=2$,
	the counit $\bar{\eps}$ and the comultiplication $\bar{\Delta}$ 
	are not degree preserving, 
	i.e.\ $Z^r(A)$ is not a Frobenius algebra in $\Zc(\Rep_k(\Zb_r))$ with these structure maps.

\item
	For $r=0$ one still obtains for every $\lambda_1,\lambda_2\in\Zb$ a non-degeneracy condition, 
	which we do not explain in detail.
	\end{enumerate}
\end{remark}

\subsection{State-sum construction}\label{sec:state-sum-constr}
Let again $r\ge0$ and
$A\in\Sc$ be a Frobenius algebra 
with invertible window element $\tau$ and 
with $N^r=\id_A$.
In this section we define a symmetric monoidal functor $\funZ_A:\Bord{r}\to\Sc$,
that is, a TFT on two-dimensional $r$-spin bordisms. 

Recall the direct sum decomposition $Z^r(A) = \bigoplus_{\lambda \in \Zb_r} Z_\lambda$ of the $\Zb_r$-graded centre from Definition~\ref{def:gradedcenter}.
We define the TFT $\funZ_A$ on objects as follows:
Let $\rho:X\to \Zb_r$ be an $r$-spin object. Then
\begin{align}
	\funZ_A(\rho):=\bigotimes_{x\in X}Z_{\rho_x}\ .
	\label{eq:tft:obj}
\end{align}

\begin{figure}[tb]
	\centering
	\def\svgwidth{6.5cm}
	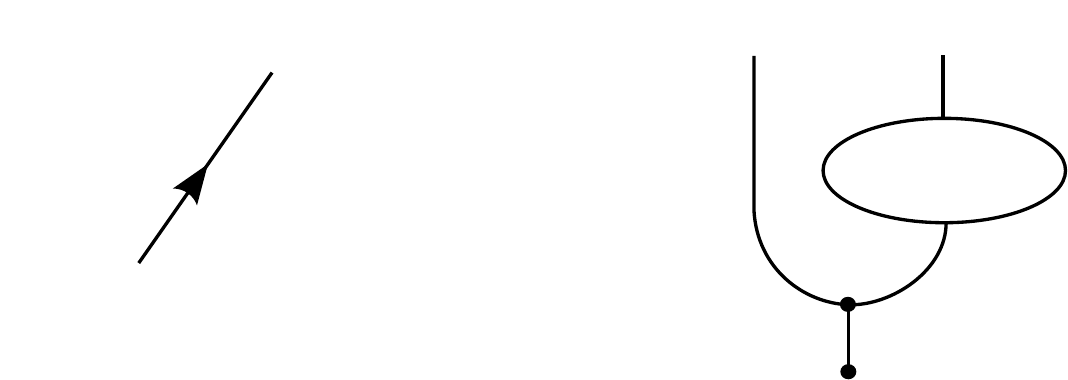
	\caption{$a)$ Left and right sides $(e,l)$ and $(e,r)$ of an inner edge $e$,
		determined by the orientation of $\Sigma$ (paper orientation) and of $e$ (arrow).
		The edge index of $e$ is $s_e$.
		$b)$ Convention for connecting tensor factors belonging to edge sides
		$(e,l)$ and $(e,r)$ of an inner edge $e$
		with the tensor factors belonging to the morphism $g_e$.}
	\label{fig:connect}
\end{figure}

To define $\funZ_A$ on morphisms is more involved and will take up the remainder of this section.
Let $(X,\rho)$ and $(Y,\sigma)$ be two $r$-spin objects.
Let $\Sigma:\rho\to\sigma$ be an $r$-spin bordism
with maps $\lambda:B_{in}\to\Zb_r$, $\mu:B_{out}\to\Zb_r$.
Choose a decorated PLCW decomposition $\Sigma_2$, $\Sigma_1$, $\Sigma_0$
	of the surface $\Sigma$
	with admissible edge index assignment $s$
	such that the $r$-spin structure
	with parametrised boundary 
	$\Sigma(s,\lambda,\mu)$ from Definition~\ref{def:rspinstrindex}
	is isomorphic to the $r$-spin structure of the $r$-spin bordism $\Sigma:\rho\to\sigma$.
Recall that $B_{in}$ and $B_{out}$ denote the 
in- and outgoing boundary components respectively and that
by our conventions they are in bijection with edges on the boundary.

For a face $f\in\Sigma_2$ which is an $n_f$-gon
let us write $(f,k)$, $k=1,\dots,n_f$ for the sides of $f$,
where $(f,1)$ denotes the marked edge of $f$, and the labeling proceeds counter-clockwise 
with respect to the orientation of $f$.
We collect the sides of all faces into a set:
\begin{align}
	S := \setc*{(f,k)}{f \in \Sigma_2 , k=1,\dots,n_f } \ .
\end{align}

We double the set of edges by considering $\Sigma_1 \times \{l,r\}$, where ``$l$'' and ``$r$'' stand for left and right, respectively. Let $E \subset \Sigma_1 \times \{l,r\}$ be the subset of all $(e,l)$ (resp.\ $(e,r)$), which 
have a face attached on the left (resp.\ right) side, cf.\ Figure~\ref{fig:connect}~$a)$. Thus for an inner edge $e\in\Sigma_1$ the set $E$ contains both $(e,l)$ and $(e,r)$,
but for a boundary edge $e'\in\Sigma_1$ the set $E$ contains either
$(e',l)$ or $(e',r)$.
By construction of $S$ and $E$ we obtain a bijection
	\begin{align}
		\Phi:E\xrightarrow{~\sim~}S
		\quad , \quad (e,x) \mapsto (f,k) \ ,
	\end{align}
where $e$ is the $k$'th edge on the boundary of the face $f$ lying on the side $x$ of $e$, counted counter-clockwise from the marked edge of $f$.

		For every vertex $v\in\Sigma_0$ in the interior of $\Sigma$ or
		on an 
		ingoing 
		boundary component of $\Sigma$ choose a side of an edge $(e,x)\in E$ 
		for which $v\in\partial(e)$.
Let 
\begin{align}
V : 	\Sigma_0 \setminus B_{out} \to E
\end{align}
be the resulting function.

To define $\funZ_A(\Sigma)$ we proceed with the following steps.
\begin{enumerate}
	\item Let us introduce the tensor products
		\begin{align}
			\begin{aligned}
			\Ac_{S}&:=\bigotimes_{(f,k) \in S}A^{(f,k)}\ ,&
			\Ac_{E}&:=\bigotimes_{(e,x) \in E}A^{(e,x)}\ ,\\
			\Ac_{in}&:=\bigotimes_{b\in B_{in}}A^{(b,in)}\ ,&
			\Ac_{out}&:=\bigotimes_{c\in B_{out}}A^{(c,out)}\ .
			\end{aligned}
			\label{eq:bigtensor}
		\end{align}
	Every tensor factor is equal to $A$, but
	the various superscripts will help us distinguish tensor factors in the source and target objects of the morphisms we define in the remaining steps.
	\item For an edge $e\in\Sigma_1$
we set
		\begin{align}
			\begin{aligned}
		g_e:=
		\begin{cases}
			A^{(e,in)}	\xrightarrow{N^{-s_e-1}} A^{(e,in)}
			& ; e \in B_{in} \\
			\Ib\xrightarrow{\eta}A\xrightarrow{\Delta}A\otimes A
			\xrightarrow{\id_{A}\otimes N^{s_e+1}}A^{(e,l)} \otimes A^{(e,out)}
			& ; e \in B_{out} \text{, surface is left of $e$ } \\
			\Ib\xrightarrow{\eta}A\xrightarrow{\Delta}A\otimes A
			\xrightarrow{\id_{A}\otimes N^{s_e+1}}A^{(e,out)} \otimes A^{(e,r)}
			& ; e \in B_{out} \text{, surface is right of $e$ } \\
			\Ib\xrightarrow{\eta}A\xrightarrow{\Delta}A\otimes A
			\xrightarrow{\id_{A}\otimes N^{s_e+1}}A^{(e,l)} \otimes A^{(e,r)}
			& ; e \text{ inner edge} \\
		\end{cases}		
			\end{aligned}
			\label{eq:gdef}
		\end{align}
		cf.\ Figure~\ref{fig:connect}.
		Define the linear map
		\begin{align}
			\Cc:=
			\bigotimes_{e\in \Sigma_1} g_{e}:\Ac_{in}\to\Ac_E\otimes\Ac_{out}\ ,
			\label{eq:step4}
		\end{align}
		where it is understood that the tensor factors in 
		$\Ac_E\otimes\Ac_{out}$ are assigned as indicated in \eqref{eq:gdef}.
	\item Note that since all tensor factors in $\mathcal{A}_E$ are algebras, so is $\mathcal{A}_E$ itself.
		For $a : \Ib \to A$ and $(e,x) \in E$ write 
		\begin{align}
			a^{(e,x)} =  \eta \otimes \cdots \otimes a \otimes \cdots \otimes \eta
			: \Ib \to \mathcal{A}_E \ ,
		\end{align}
		where $a$ maps to the tensor factor $A^{(e,x)}$. 
		Define $y : \Ib \to \mathcal{A}_E$ as the following product 
		in the $k$-algebra $\mathcal{S}(\Ib,\mathcal{A}_E)$:
		\begin{align}
			  y = \prod_{v \in \Sigma_0 \setminus B_{out}} (\tau^{-1})^{V(v)} \ .
			\label{eq:step2.1}
		\end{align}
		Finally, we let $\Yc$ be the endomorphism of $\mathcal{A}_E$ obtained by multiplying with $y$,
		\begin{align}
			\Yc :=\left[ \Ac_E \xrightarrow{y \cdot (-)} \Ac_E \right]\ .
			\label{eq:step2}
		\end{align}
	\item Let $\mu^{(1)}:=\id_A$ and let $\mu^{(n)}$ denote the $n$-fold product for $n\ge2$.
		Assign to every face $f\in\Sigma_2$ obtained from an $n_f$-gon 
		the morphism $\eps\circ\mu^{(n_f)} : A_{(f,1)} \otimes \cdots \otimes A_{(f,n_f)} \to \Ib$ and take their tensor product:
		\begin{align}
			\Fc:=\bigotimes_{f\in\Sigma_2}\left(\eps\circ\mu^{(n_f)}\right):\Ac_{S}\to\Ib \ .
			\label{eq:step1}
		\end{align}
	\item 
We will now put the above morphisms together to obtain a morphism 
$\mathcal{L} : \mathcal{A}_{in} \to \mathcal{A}_{out}$. Denote by $\Uppi_\Phi$ the permutation of tensor factors
induced by $\Phi:E\to S$,
		\begin{align}
			\Uppi_\Phi:\Ac_E\to\Ac_S \ .
			\label{eq:step3}
		\end{align}
Using this, we define
		\begin{align}	\Kc&:=\left[\Ac_E\xrightarrow{\Yc}\Ac_E\xrightarrow{\Uppi_\Phi}\Ac_S\xrightarrow{\Fc}\Ib\right]\ ,\\
			\Lc&:=
			\left[ \Ac_{in}\xrightarrow{\Cc} \Ac_E\otimes\Ac_{out} \xrightarrow{\Kc\otimes\id_{\Ac_{out}}} \Ac_{out}  \right]\ .
			\label{eq:step6a}
		\end{align}
	\item Let $\Uppi_{in}$ and $\Uppi_{out}$ 
		denote the permutation of tensor factors induced by 
		the maps $\beta_{in}$ and $\beta_{out}$ respectively:
		\begin{align}
			\Uppi_{in}:&\funZ_A(\rho)=\bigotimes_{x\in X}Z_{\rho_x}\to\bigotimes_{b\in B_{in}}Z_{\lambda_b}\ ,
			\label{eq:betabar-in}\\
			\Uppi_{out}:&\bigotimes_{c\in B_{out}}Z_{\mu_c}\to\bigotimes_{y\in Y}Z_{\sigma_y}=\funZ_A(\sigma)\ .
			\label{eq:betabar-out}
		\end{align}
Using these permutations and the embedding and projection maps 
$\iota_\lambda$, $\pi_\lambda$ from \eqref{eq:splitidempot}
we construct the morphisms linking the action of $\funZ_A$ on objects 
to the tensor products $\mathcal{A}_{in/out}$:
		\begin{align}
			\Ec_{in}:=&\left[\funZ_A(\rho)\xrightarrow{\Uppi_{in}} \bigotimes_{b\in B_{in}}Z_{\lambda_b} \xrightarrow{\bigotimes_{b\in B_{in}}\iota_{\lambda_b}} \Ac_{in}\right]\ ,
			\label{eq:step5in}\\
			\Ec_{out}:=&\left[ \Ac_{out}\xrightarrow{\bigotimes_{c\in B_{out}}\pi_{\mu_c}} \bigotimes_{c\in B_{out}}Z_{\mu_c} \xrightarrow{\Uppi_{out}}\funZ_A(\sigma)\right]\ .
			\label{eq:step5out}
		\end{align}
We have now gathered all ingredients to define the action of $\funZ_A$ on morphisms:		
		\begin{align}				\funZ_A(\Sigma)&:=
			\left[ \funZ_A(\rho)\xrightarrow{\Ec_{in}} \Ac_{in}\xrightarrow{\Lc} \Ac_{out}\xrightarrow{\Ec_{out}} \funZ_A(\sigma)\right]\ .
			\label{eq:step6b}
		\end{align}
\end{enumerate}

\begin{theorem} \label{thm:tft}\leavevmode
Let $A\in\Sc$ be a Frobenius algebra with invertible window element $\tau$ and 
with $N^r=\id_A$.
	\begin{enumerate}
		\item The morphism defined in \eqref{eq:step6b} is 
			independent of the choice of the marked PLCW decomposition and the assignment $V$.
		\item The state-sum construction yields a 
			symmetric monoidal functor $\funZ_A:\Bord{r}\to\Sc$
			whose action on objects and morphisms is given by 
			\eqref{eq:tft:obj} and \eqref{eq:step6b}, respectively. 
	\end{enumerate}
\end{theorem}

The proof of this theorem works by reducing to the corresponding statement 
for triangulations and is given in Appendix~\ref{app:pf:thm:tft}.

\begin{remark}
The above construction yields a TFT on the category of closed $r$-spin bordisms,
where the complete boundary of the $r$-spin bordisms is parametrised,
so the parametrised boundary is a closed manifold. 
One can define a different $r$-spin bordism category,
called the open-closed $r$-spin bordism category, where only a 
one dimensional submanifold of the boundary of $r$-spin surfaces
is parametrised.
The subcategory of the latter generated by the open cup, the open pair of pants and their duals is called the
open $r$-spin bordism category.
In \cite{Stern:2016stft} a TFT on open $r$-spin bordisms was constructed 
using $\Lambda_r$-Frobenius algebras \cite[Prop.\,I.41]{Dyckerhoff:2015csg}
which are Frobenius algebras whose Nakayama automorphism $N$ satisfies $N^r=\id$.
\end{remark}

\subsection{Evaluation of state-sum TFTs on generating $r$-spin bordisms}\label{sec:evaluation}

In this section we apply the state-sum construction from Theorem~\ref{thm:tft} to pairs of pants and discs with $r$-spin structure. On the one hand, these bordisms generate $\Bord{r}$, and on the other hand, we will recover the algebra structure of the $\Zb_r$-graded center $Z^r(A)$ of $A$ in this way.
Finally, we evaluate $\funZ_A$ on a connected bordism of genus $g$ 
with only ingoing boundary components.

\subsubsection*{Pair of pants as multiplication}
Consider the $r$-spin 3-holed sphere
parametrised as in Section~\ref{sec:sigmagb} with 2 ingoing boundary components 
$B_{in}=\left\{ u_1,u_2 \right\}$
and 1 outgoing boundary component 
$B_{out}=\left\{ u_3 \right\}$
between $r$-spin objects
$\rho:\left\{ x_1,x_2 \right\}\to\Zb_r$
and $\sigma:\left\{ y \right\}\to\Zb_r$
with $\beta_{in}(x_i)=u_i$ ($i=1,2$) and $\beta_{out}(y)=u_3$.
Let $\lambda_1:=\lambda_{u_1}$, $\lambda_2:=\lambda_{u_2}$ and $\lambda_3:=\mu_{u_3}$.
Then by \eqref{eq:bdrysign} $R_{u_1}=\lambda_1-1$, $R_{u_2}=\lambda_2-1$ and $R_{u_3}=1-\lambda_3$.
Substituting these and $\chi(\Sigma_{0,3})=2-0-3=-1$ in \eqref{eq:prop:sigmagb} gives
\begin{align}
	\lambda_1+\lambda_2&\equiv \lambda_3&&\Mod{r}\ .\label{eq:vcond:mult}
\end{align}
Denote this $r$-spin bordism by 
\begin{align*}
S_{1,2}(u_1,u_2,u_3,\lambda_1,\lambda_2):= 
\Sigma_{0,3}(u_1,u_2,u_3,\lambda_1-1,\lambda_2-1,1-\lambda_3):\rho\to\sigma\ , 
\end{align*}
(cf.\ \eqref{eq:sigmagb-def}).
The sets $S$, $E$ are (see Figure~\ref{fig:sigmagb_param})
\begin{align}
	S=&\setc*{(f,k)}{k=1,\dots,9}\simeq\left\{ 1,\dots,9 \right\}\ ,
	\label{eq:pants-sides1}\\
	E=&\left\{ (u_1,r),(u_2,r),(r_1,l),(r_1,r),(r_2,l),(r_2,r),(r_3,l),(r_3,r),(u_3,r) \right\}\nonumber\\
	\simeq&\left\{ 1,\dots,9 \right\}\ ,
	\label{eq:pants-sides2}
\end{align}
where in \eqref{eq:pants-sides2} the isomorphism is given by the order of elements of $E$ as listed.
We have one inner vertex $v_0$ and 3 boundary vertices $v_1$, $v_2$ and $v_3$,
with $v_3$ placed on the outgoing boundary component.
We set
\begin{align}
	 V(v_0):&=(r_2,r)\ ,&V(v_1):&=(u_1,r)\ ,&V(v_2):&=(u_2,r)\ .
	\label{eq:pants-sides3}
\end{align}
Following the steps of the state-sum construction we get:
\begin{enumerate}
	\setcounter{enumi}{1}
	\item For the various edge indices
		\begin{align*}
			\Cc=N^{-u_1-1}\otimes N^{-u_2-1}\otimes g_{r_1}\otimes g_{r_2}
			\otimes g_{r_3}\otimes g_{u_3}
		\end{align*}
		from \eqref{eq:step4}.
		Recall from Notation~\ref{not:edgeindex} that the same symbols denote 
		edges and the assigned edge indices.
	\item For the inner vertex and the ingoing vertices we set
		\begin{align*}
			y=(\tau^{-1})^{\otimes2}\otimes \eta^{\otimes3}\otimes		
			(\tau^{-1})\otimes \eta^{\otimes 3}
		\end{align*}
		from \eqref{eq:step2.1} according to the map $V$ in \eqref{eq:pants-sides3}.
	\item For the single 9-gon $\Fc=\eps\circ\mu^{(9)}$ from \eqref{eq:step1}.
	\item The permutation is $\Uppi_{\Phi}=(12543)(89)$ from \eqref{eq:step3}
		where we use the cycle notation for the permutation of tensor factors.
		After a calculation using associativity of the product and
		the last equation of \eqref{eq:plem},
		the morphism $\Lc$ in \eqref{eq:step6a} is
		\begin{align}
			\left[ A\otimes A\xrightarrow{P_{\lambda_1}\circ N^{-u_1-1}\otimes P_{\lambda_2}\circ N^{-u_2-1}} A\otimes A\xrightarrow{\mu}A \xrightarrow{P_{\lambda_1+\lambda_2}\circ N^{u_3+1}}A\right]\ .
			\label{eq:pants:state-sum}
		\end{align}

	\item For the in- and outgoing boundary components we get
		$\Ec_{in}=\iota_{\lambda_1}\otimes\iota_{\lambda_2}$ and
		$\Ec_{out}=\pi_{\lambda_1+\lambda_2}$ from \eqref{eq:step5in} and \eqref{eq:step5out},
		since the permutations induced by $\beta_{in}$ and $\beta_{out}$ 
		from \eqref{eq:betabar-in} and \eqref{eq:betabar-out} are identities.
		Also note that $\rho_{x_1}=\lambda_1$, etc.
		Finally by composing $\Lc$ with $\Ec_{in}$ and $\Ec_{out}$ as in \eqref{eq:step6b} we obtain
		\begin{align}
			\begin{aligned}
			&\funZ_A\left( S_{1,2}(u_1,u_2,u_3,\lambda_1,\lambda_2) \right)
			\\
			=&\left[ Z_{\lambda_1}\otimes Z_{\lambda_2}\xrightarrow{N_{\lambda_1}^{-u_1}\otimes N_{\lambda_2}^{-u_2}} Z_{\lambda_1}\otimes Z_{\lambda_2}\xrightarrow{\mu_{\lambda_1,\lambda_2}}Z_{\lambda_1+\lambda_2} \xrightarrow{N_{\lambda_3}^{u_3}}Z_{\lambda_1+\lambda_2}\right]\ .
			\end{aligned}
			\label{eq:gradedmultcomp}
		\end{align}
\end{enumerate}

Observe that $\funZ_A\left( S_{1,2}(0,0,0,\lambda_1,\lambda_2)\right)=\mu_{\lambda_1,\lambda_2}$
from \eqref{eq:gradedmult3}.

\subsubsection*{Cup as unit}
Consider a disk with outgoing boundary.
By Corollary~\ref{cor:disk}, we get a unique $r$-spin structure for
boundary parametrisation, namely $\mu=0$. Note that the map $\beta_{out}$ is unique.
Using the notation in \eqref{eq:sigmagb-def} we write $S_{1,0}:= \Sigma_{0,1}(u,0):\emptyset\to\rho$. 
with $\rho:\left\{ * \right\}\to\Zb_r$ $\rho_*=0$.
However, since the $r$-spin structure is actually independent of $u$ we may as well set $u=0$.
We have
\begin{align}
	S=&\setc*{(f,k)}{k=1,2,3}\simeq\left\{ 1,2,3 \right\}\ ,
	\label{eq:cup-sides1}\\
	E=&\left\{ (r_1,l),(r_1,r),(u_1,r) \right\} \simeq\left\{ 1,2,3 \right\}\ .
	\label{eq:cup-sides2}
\end{align}
There is an inner vertex $v_0$ and an outgoing boundary vertex $v_1$, and we set
\begin{align}
	V(v_0):=(u_1,r)\ .
	\label{eq:cup-sides3}
\end{align}
By the state-sum construction one has
\begin{enumerate}
	\setcounter{enumi}{1}
	\item For the 2 edges $\Cc=g_{r_1}\otimes g_{u_1}$
		from \eqref{eq:step4}.
	\item For the inner vertex 
		$y=\eta^{\otimes2}\otimes\tau^{-1}$ from \eqref{eq:step2.1}.
	\item For the single 3-gon $\Fc=\eps\circ\mu^{(3)}$ from \eqref{eq:step1}.
	\item The permutation is $\Uppi_{\Phi}=(23)$ from \eqref{eq:step3}.
		Putting the above together according to \eqref{eq:step6a} we get
		\begin{align}
			\Lc=P_0\circ\eta\ .
			\label{eq:gradedunit:state-sum}
		\end{align}
	\item For the (empty) in- and outgoing boundary components we get
		$\Ec_{in}=\id_{\Ib}$ and
		$\Ec_{out}=\pi_{0}$ from \eqref{eq:step5in} and \eqref{eq:step5out}.
		{}From \eqref{eq:step6b} we finally get
		\begin{align}
			\funZ_A(S_{1,0})=\left[\Ib\xrightarrow{\eta}A\xrightarrow{\pi_0}Z_0\right]\ .
			\label{eq:gradedunit}
		\end{align}
\end{enumerate}
Observe that $\left[\Ib\xrightarrow{\funZ_A( S_{1,0})}Z_0\xrightarrow{e_0}\oplus_{\lambda\in\Zb_r}Z_{\lambda}\right]
=\bar{\eta}$ from \eqref{eq:gradedunit2}.

\subsubsection*{Pair of pants as comultiplication}
Consider a 3-holed sphere with the parametrisation as above, 
just with in- and outgoing boundary components exchanged,
i.e.\ $\lambda_1$, $\lambda_2$ 
stand for outgoing boundary components, $\lambda_3$ for the ingoing etc.
Then from \eqref{eq:prop:sigmagb} one has:
\begin{align}
	\lambda_1+\lambda_2-2&\equiv \lambda_3&&\Mod{r}\ .
	\label{eq:copants:bdrycond}
\end{align}
Denote this $r$-spin surface with parametrised boundary by
\begin{align*}
S_{2,1}(u_1,u_2,u_3,\lambda_1,\lambda_2):= 
	\Sigma_{0,3}(u_1,u_2,u_3,1-\lambda_1,1-\lambda_2,\lambda_3-1):\sigma\to\rho\ , 
\end{align*}
(cf.\ \eqref{eq:sigmagb-def}).
The morphism $\Lc$ in \eqref{eq:step6a} assigned to it by the state-sum construction is
\begin{align}
	\left[ A\xrightarrow{P_{\lambda_1+\lambda_2-2}\circ N^{-u_3-1}} A\xrightarrow{\Delta\circ(\tau\cdot(-))}A\otimes A \xrightarrow{P_{\lambda_1}\circ N^{u_1+1}\otimes P_{\lambda_2}\circ N^{u_2+1}}A\otimes A\right]\ .
	\label{eq:copants:state-sum1}
\end{align}
and from \eqref{eq:step6b} one obtains
\begin{align}
	&\funZ_A\left( S_{2,1}(u_1,u_2,u_3,\lambda_1,\lambda_2) \right)= \nonumber\\
	&\left[ Z_{\lambda_1+\lambda_2-2}\xrightarrow{N_{\lambda_1+\lambda_2-2}^{-u_3}} Z_{\lambda_1+\lambda_2-2}\xrightarrow{\Delta_{\lambda_1,\lambda_2}}Z_{\lambda_1} \otimes Z_{\lambda_2}\xrightarrow{N_{\lambda_1}^{u_1}\otimes N_{\lambda_2}^{u_2}}Z_{\lambda_1}\otimes Z_{\lambda_2}\right]\ .
	\label{eq:gradedcomultcomp}
\end{align}
Observe that $\funZ_A\left( S_{2,1}(0,0,0,\lambda_1,\lambda_2)\right)=\Delta_{\lambda_1,\lambda_2}$
from \eqref{eq:gradedcomult3}.
While the above morphism is defined also for $r=0$, 
as was remarked in Section~\ref{sec:gradedcenter}
one can sum these morphisms only in the case when $r\ne0$,
in which case one obtains \eqref{eq:gradedcomult2}.

\subsubsection*{Cap as counit}
Consider an $r$-spin disk with ingoing boundary.
By Corollary~\ref{cor:disk}, the boundary parametrisation 
has $\lambda=2$ and the $r$-spin structure is independent of the edge indices.
Denote this $r$-spin surface with parametrised boundary 
with $S_{0,1}:= \Sigma_{0,1}(0,2):\sigma\to\emptyset$, (cf.\ \eqref{eq:sigmagb-def}),
with $\sigma:\left\{ * \right\}\to\Zb_r$ $\sigma_*=2$.
By the state-sum construction one has
\begin{align}
	\funZ_A(S_{0,1})=
	\left[Z_2\xrightarrow{\iota_2}A\xrightarrow{(\tau^{-1}\cdot(-))}A\xrightarrow{\eps}\Ib\right]\ .
	\label{eq:gradedcounit}
\end{align}
Observe that $\left[\oplus_{\lambda\in\Zb_r}Z_{\lambda}\xrightarrow{p_2}Z_2\xrightarrow{\funZ_A( S_{1,0})}\Ib\right]
=\bar{\eps}$ from \eqref{eq:gradedcounit2}. 

\medskip

We collect the above computations for $\funZ_A$ evaluated on generators in the following proposition:

\begin{proposition}\label{prop:state-space-graded-center}
Let $A\in\Sc$ be a Frobenius algebra with invertible window element $\tau$ and 
with $N^r=\id_A$, and let $\funZ_A$ be the $r$-spin TFT defined in Theorem~\ref{thm:tft}.
The the $\Zb_r$-graded center $Z^r(A)$ is equal to $\bigoplus_{\lambda \in \Zb_r} \funZ_A(\lambda)$ with product and unit (restricted to the corresponding graded components) given by $\funZ_A\left( S_{1,2}(0,0,0,\lambda_1,\lambda_2)\right)$ and $\funZ_A(S_{1,0})$, respectively. 
For $r>0$, we obtain an equality of Frobenius algebras.
\end{proposition}

For $r=2$, the above relation between state spaces and the $\Zb_r$-graded center was already observed in \cite{Moore:2006db}.

\subsubsection*{Connected $r$-spin bordisms}

Finally, let us evaluate $\funZ_A$ on a general connected $r$-spin bordism with only ingoing boundary components, that is, on 
$\Sigma_{g,b}(s_i,t_i,u_j,\lambda_j-1)$ in the notation of \eqref{eq:sigmagb-def}.
Write 
\begin{align}
	\begin{aligned}
		\def\svgwidth{6cm}
\begingroup%
  \makeatletter%
  \providecommand\color[2][]{%
    \errmessage{(Inkscape) Color is used for the text in Inkscape, but the package 'color.sty' is not loaded}%
    \renewcommand\color[2][]{}%
  }%
  \providecommand\transparent[1]{%
    \errmessage{(Inkscape) Transparency is used (non-zero) for the text in Inkscape, but the package 'transparent.sty' is not loaded}%
    \renewcommand\transparent[1]{}%
  }%
  \providecommand\rotatebox[2]{#2}%
  \ifx\svgwidth\undefined%
    \setlength{\unitlength}{175.87275391bp}%
    \ifx\svgscale\undefined%
      \relax%
    \else%
      \setlength{\unitlength}{\unitlength * \real{\svgscale}}%
    \fi%
  \else%
    \setlength{\unitlength}{\svgwidth}%
  \fi%
  \global\let\svgwidth\undefined%
  \global\let\svgscale\undefined%
  \makeatother%
  \begin{picture}(1,0.80001022)%
    \put(0,0){\includegraphics[width=\unitlength]{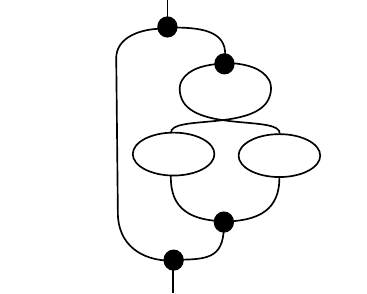}}%
    \put(0.69406592,0.35529227){\color[rgb]{0,0,0}\makebox(0,0)[lb]{\smash{$t+1$}}}%
    \put(0.39627518,0.36321829){\color[rgb]{0,0,0}\makebox(0,0)[lb]{\smash{$s+1$}}}%
    \put(-0.00830679,0.36321829){\color[rgb]{0,0,0}\makebox(0,0)[lb]{\smash{$\varphi(s,t):=$}}}%
  \end{picture}%
\endgroup%

		\label{eq:tftphi}
	\end{aligned}
\end{align}
Using the decomposition of $\Sigma_{g,b}$ from Figure~\ref{fig:sigmagb_param}~$a)$, a straightforward computation along the same lines as above gives the following proposition.

\begin{proposition}\label{prop:tftsgb}
Let $\Sigma_{g,b}(s_i,t_i,u_j,\lambda_j-1)$ denote the $r$-spin surface 
	of Definition~\ref{def:rspinstrindex} with only ingoing boundary components. Then
	\begin{align}
		\funZ_A(\Sigma_{g,b}(s_i,t_i,u_j,\lambda_j-1))=\eps\circ(\tau^{-1}\cdot(-))\circ\prod_{i=1}^g\varphi(s_i,t_i)\circ\mu^{(b)}
		\circ\bigotimes_{j=1}^b (N^{-u_j-1}\circ\iota_{\lambda_j}).
		\label{eq:tftsgb}
	\end{align}
\end{proposition}

\section{\texorpdfstring{$r$}{r}-spin TFT computing the Arf-invariant}\label{sec:arf}

In this section we give an example for the state-sum construction of $r$-spin TFTs, namely for the two-dimensional Clifford algebra in 
super vector spaces, 
and we compute its value on connected $r$-spin bordisms (Section~\ref{sec:arftft}). We then recall the definition of the Arf invariant for $r$-spin surfaces and observe that the TFT obtained from the Clifford algebra computes this invariant (Section~\ref{sec:arfinv}).

\subsection{$r$-spin TFT from a Clifford algebra}\label{sec:arftft}

Let $r\in\Zb_{\ge 0}$ be even and let $k$ be a field not of characteristic $2$.
Let $\Cl \in\SVect$ be the Clifford algebra with one odd generator $\theta$, 
i.e.\ $\Cl = k \oplus k\theta$ with $\theta^2 = 1$. We turn $\Cl$ into a Frobenius algebra via
\begin{align}
	\eps(1)&=2 \ ,&
	\eps(\theta)&=0\ ,&
		\Delta(1) &= \frac{1}{2}(1 \otimes 1 + \theta \otimes \theta) 
		\ ,&
	\Delta(\theta) &= \frac{1}{2}(\theta \otimes 1 + 1 \otimes \theta) \ .
	\label{eq:cl-fa-structure}
\end{align}

\begin{lemma} \label{lem:arffa}
	For the Frobenius algebra $\Cl$ the following hold.
	\begin{enumerate}
		\item $\tau=\mu\circ\Delta\circ\eta=\eta$, hence $\Cl$ has invertible window element.
		\item The Nakayama automorphism is given by $N(\theta^m)=(-1)^{m}\theta^m$.
		\item For $\lambda\in\Zb_r$, $P_{\lambda}(\theta^m)=\frac{1}{2}\left[1+(-1)^{\lambda-m}\right]\theta^m$,
			hence $Z_{\lambda}=k\theta^{\lambda}$.
		\item 
The morphism $\varphi_{s,t}$ from \eqref{eq:tftphi} is given by		
		$\varphi_{s,t}=\frac{1}{2}(-1)^{(s+1)(t+1)}\id_{\Cl}$.
	\end{enumerate}
\end{lemma}

\begin{proof}
	\begin{enumerate}
		\item $\tau(1)=\mu\circ\Delta\circ\eta(1)=
		  \frac{1}{2}\, \mu\left( 1 \otimes 1 + \theta \otimes \theta \right)
		  =\eta(1)$. 
			Its inverse is $\eta$.
		\item 
				$N(1)=1$ in any Frobenius algebra. 
			We calculate $N(\theta)$ in steps:
			\begin{align*}
				\theta\mapsto
				\theta\otimes (1\otimes 1+\theta\otimes\theta)/2
				\mapsto (1\otimes \theta\otimes 1-\theta\otimes \theta\otimes \theta)/2
				\mapsto - \theta\ .
			\end{align*}
		\item We calculate $P_{\lambda}(\theta^m)$ in steps according to (\ref{eq:pdef}):
			\begin{align*}
				\theta^m&
				\mapsto\frac{1}{2}
				(\theta^{m}\otimes 1+
				\theta^{m-1}\otimes \theta)
				\mapsto\frac{1}{2}
				(\theta^{m}\otimes 1+(-1)^{1-\lambda}
				\theta^{m-1}\otimes \theta)\\
				&\mapsto\frac{1}{2}
				(1\otimes \theta^{m}+(-1)^{m-\lambda}
				\theta\otimes \theta^{m-1})
				\mapsto\frac{1}{2}\theta^{m}
				(1 +(-1)^{m-\lambda})\ .
			\end{align*}
			We see that if $\lambda$ and $m$ have the same parity
			this is the identity, otherwise this is zero, i.e.\
			$P_{\lambda}$ is a projection onto $k.\theta^{\lambda}$.
		\item We calculate $\varphi_{s,t}(\theta^m)$ in steps according to
			(\ref{eq:tftphi}):
			\begin{align*}
				\theta^m\mapsto&\frac{1}{2} \sum_{n=0}^1 
				\theta^{m-n}\otimes \theta^{n}
				\mapsto \frac{1}{4} \sum_{n,p=0}^1
				\theta^{m-n}\otimes \theta^{n-p}\otimes \theta^{p}\\
				\mapsto&\frac{1}{4} \sum_{n,p=0}^1 
				(-1)^{(s+1)(n-p)+(t+1)p}
				\theta^{m-n}\otimes \theta^{n-p}\otimes \theta^{p}\\
				\mapsto&\frac{1}{4} \sum_{n,p=0}^1 
				(-1)^{(s+1)(n-p)+(t+1)p+(n-p)p}
				\theta^{m-n}\otimes \theta^{p}\otimes \theta^{n-p}\\
				\mapsto&\frac{1}{4}\theta^{m}
				\sum_{n,p=0}^1 
				(-1)^{(s+1)(n-p)+(t+1)p+(n-p)p} \\
				=&\frac{1}{4}\theta^{m}
				\sum_{n,p=0}^1 
				(-1)^{(s+1+p)(t+1+n-p)-(s+1)(t+1)} 
				=\frac{1}{2}\theta^{m}
				(-1)^{(s+1)(t+1)} \ ,
			\end{align*}
			where at the last step we execute first
			the summation over $n$ for a fixed $p$ and notice
			that we either get 0 or 2.
	\end{enumerate}
\end{proof}

Let $\funZ_{\Cl}$ denote the TFT from Theorem~\ref{thm:tft}
given by the Frobenius algebra $\Cl$ and recall from Section~\ref{sec:sigmagb} the $r$-spin structure with parametrised boundary
$\Sigma_{g,b}(s_i,t_i,u_j,\lambda_j-1)$ with only ingoing boundary components and where $g + b \ge 1$.
By calculating (\ref{eq:tftsgb}) in Proposition~\ref{prop:tftsgb} 
and using \eqref{eq:prop:sigmagb}
we get the following proposition.

\begin{proposition}\label{prop:arftft}
	The value of the TFT $\funZ_{\Cl}$ is
\begin{align}
	\funZ_{\Cl}(\Sigma_{g,b}(s_i,t_i,u_j,\lambda_j-1))(\theta^{\lambda_1}\otimes\dots\otimes \theta^{\lambda_b})
		=2^{1-g}(-1)^{\sum_{n=1}^{g}(s_n+1)(t_n+1)+\sum_{j=1}^{b-1} (u_j-u_b)\lambda_j}.
	\label{eq:tqft:arf2}
\end{align}
\end{proposition}

\begin{corollary}\label{cor:arfsurj}
	Assume that 
	$g\ge1$, or else that $b \ge 1$ 
	and at least one
	of the $\lambda_j$'s is odd (by \eqref{eq:prop:sigmagb} in this case $b \ge 2$ and at least two $\lambda_j$'s are odd).
	Then the following map is surjective:
	\begin{align}
		\Rc^r(\Sigma_{g,b})_{\lambda}&\to 
		\left\{ +1,-1 \right\}\nonumber\\
		\left[\Sigma_{g,b}(s_i,t_i,u_j,\lambda_j-1)\right]
		&\mapsto 2^{g-1}\cdot
		\funZ_{\Cl}(\Sigma_{g,b}(s_i,t_i,u_j,\lambda_j-1))(\theta^{\lambda_1}\otimes\dots\otimes \theta^{\lambda_b})
		\label{eq:arfsurj}
	\end{align}
\end{corollary}

\begin{remark}
\begin{enumerate}
\item
	One can show, using a similar argument as in \cite[Sec.\,6.5]{Novak:2015phd}, 
	that for any choice of Frobenius algebra $A\in\Vect$
	with invertible window element and with $N^r=\id_A$ the TFT $\funZ_A$
	of Section~\ref{sec:state-sum-constr} is independent 
	of the $r$-spin structure.
	The idea is that if there exists a symmetric Frobenius algebra structure on 
	an algebra $A$, then $\funZ_A$ is independent of the $r$-spin structure
	for every other Frobenius algebra structure on $A$ as well.
\item
	Let $r$ be a positive integer and let us consider 
	the category of $\Zb_r$-graded $k$-vector spaces $\Vect_{\Zb_r}$.
By using the correspondence between braided monoidal structures on $\Vect_{\Zb_r}$ and quadratic forms on $\Zb_r$ \cite{Joyal:1993bmc} (see \cite[App.\,A]{Fuchs:2004r3} for a review) one can check that for odd $r$ there is only one symmetric monoidal structure on $\Vect_{\Zb_r}$. For even $r$ there are two: the trivial one inherited from $\Vect$ and the non-trivial one given by the super grading.

\item
One may wonder whether taking $\Vect_{\Zb_r}$ with some choice of 
symmetric monoidal structure would yield more examples of $r$-spin TFTs 
than what one can find with target $\Vect$ or $\SVect$. 
Part 2 shows that this is not so: All symmetric monoidal structures 
on $\Vect_{\Zb_r}$ are inherited from $\Vect$ or $\SVect$ 
(and only from the former for $r$ odd). 
Thus all algebras $A \in \Vect_{\Zb_r}$ as in Theorem~\ref{thm:tft} 
are also algebras in $\Vect$, respectively $\SVect$, with the same properties, 
and produce the same results in the state-sum construction.
\end{enumerate}
\end{remark}

\subsection{The $r$-spin Arf-invariant}\label{sec:arfinv}

\begin{definition}[{\cite[Sec.\,2.4]{Randal:2014rs}} and {\cite[Sec.\,5]{Geiges:2012rs}}]\label{def:arfinv}
Let $r \ge 0$ be even.
	The \textsl{$r$-spin Arf-invariant} of the $r$-spin surface $\Sigma_{g,b}$ is 
	\begin{align}
		&\Arf(\Sigma_{g,b})=\sum_{i=1}^g (\zeta(a_i)+1)\cdot(\zeta(b_i)+1)+
			\sum_{j=1}^{b-1} (\zeta(c_j)+1)\cdot(\zeta(\partial_j)+1)&\Mod{2}\ .
		\label{eq:arfinv}
	\end{align}
\end{definition}
Notice that for $r$ even, 
$r$-spin structures naturally factorise through $2$-spin structures.
Therefore it makes sense to talk about the Arf-invariant of them,
which was introduced for 2-spin structures \cite{Johnson:1980arf}.
A related description of $r$-spin structures and of their Arf-invariant is given in 
\cite[Def.\,5.1]{Natanzon:2004harf} in terms of so called $r$-Arf functions, and in
\cite[Sec.\,3.1]{Salter:2017mon} 
using the language of winding number functions of \cite{Humphries:1989wn}, both akin to the $\zeta(-)$ from above

$\Arf(\Sigma_{g,b})$ is invariant under the action of the mapping class group of $\Sigma_{g,b}$,
which has been proven in 
\cite[Prop.\,2.8]{Randal:2014rs} and
\cite[Lem.\,7]{Geiges:2012rs}.
We provide a different proof of this result in the
corollary to the following theorem.

\begin{theorem}\label{thm:arftft}
	The TFT $\funZ_{\Cl}$ computes the $r$-spin Arf-invariant:
\begin{align}
	\funZ_{\Cl}(\Sigma_{g,b}(s_i,t_i,u_j,\lambda_j-1))(\theta^{\lambda_1}\otimes\dots\otimes \theta^{\lambda_b}) =
		2^{1-g}\cdot
	(-1)^{\Arf(\Sigma_{g,b}(s_i,t_i,u_j,\lambda_j-1))}\ .
	\label{eq:tqft:arf}
\end{align}
\end{theorem}

\begin{proof}
This is immediate from 
Proposition~\ref{prop:arftft}, Proposition~\ref{prop:indexhol-sigma-bg} and Definition~\ref{def:arfinv}.
\end{proof}

Since the morphisms in $\Bord{r}$ are diffeomorphism classes of $r$-spin bordisms (relative boundary), we get:

\begin{corollary}
	The $r$-spin 
	Arf invariant is constant on mapping class group orbits.
\end{corollary}

\appendix

\section{From triangulations to PLCW decompositions}\label{app:novak}

By a {\em triangulation of a surface} we mean a smooth
simplicial complex for the surface
such that each boundary component
consists of 3 edges and 3 vertices.
In \cite{Novak:2015phd} a combinatorial description of $r$-spin
surfaces was given using triangulations. 
The purpose of this appendix is to show 
how to obtain the combinatorial description of $r$-spin surfaces
using PLCW decomposition of Section~\ref{sec:comb} from triangulations.

\subsection{$r$-spin surfaces with triangulations}\label{app:triang}

\begin{figure}[tb]
	\centering
	\def\svgwidth{5cm}
	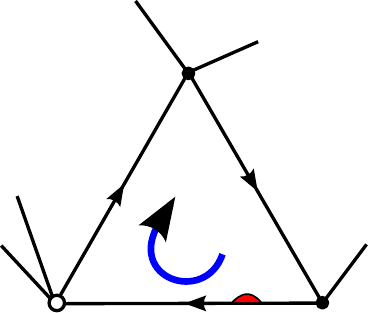
	\caption{The 3 edges and 3 vertices of a boundary component together with 
	the additional marking of one edge.
	The curly arrow shows the orientation of the boundary component and the 
	empty vertex shows the ending vertex of the additionally marked edge.}
	\label{fig:extramarking}
\end{figure}

Let us summarise the results of \cite{Novak:2015phd}.
More precisely let us look at the differences between 
that formalism and the formalism developed in Section~\ref{sec:comb}.

Let $\Sigma$ be a marked triangulation of a surface with parametrised boundary,
i.e.\ every edge has an orientation and an edge index and every face has a marked edge.
Let us assume that all boundary components are ingoing
and recall the notions of Section~\ref{sec:combsub}.
Put an additional marking on one of the edges
of each boundary component $b$.
The induced orientation of the boundary component gives a starting and ending vertex of this
additionally marked edge, see Figure~\ref{fig:extramarking}.
For a boundary vertex $u$ let $\alpha_u:=+1$ if 
it is an ending vertex for the additionally marked edge and
$\alpha_u:= 0$ otherwise.
We furthermore assume that the orientation of boundary edges agrees with the induced orientation of the boundary components.
The marking is called \textsl{admissible} for a given map
$\tilde{\lambda}:\pi_0(\partial\Sigma)\to\Zb_r$ $b\mapsto\tilde{\lambda}_b$, 
if the following hold for every inner vertex $v$ and every boundary vertex $u$ on a boundary component $b$.
\begin{align}
	\sum_{e\in \partial^{-1}(v)}\hat{s}_e&\equiv D_v-N_v+1&\Mod{r}\ ,
	\label{eq:tri:vertexcond1}\\
	\sum_{e\in \partial^{-1}(u)}\hat{s}_e&\equiv D_u-N_u+1+\alpha_u\cdot(1-\tilde{\lambda}_b)&\Mod{r}\ .
	\label{eq:tri:vertexcond2}
\end{align}
Here, $D_{v/u}$, $N_{v/u}$ and $\hat s_e$ are defined as in Section~\ref{sec:combsub}.
According to the construction in \cite[Sec.\,4.8]{Novak:2015phd} we proceed as follows:
\begin{itemize}
	\item Define an $r$-spin structure on $\Sigma$ minus edges and vertices
		by giving the interior of the faces the trivial $r$-spin structure $\Cb^0$.

	\item Define transition functions for every edge
	  	via the edge indices 
		to extend the above to $\Sigma$ minus vertices.
		Note that the transition functions are constant $\Zb_r$ valued functions, since these are transition functions for the $\Zb_r$-bundle $q:P\to F_{\Sigma}$.
	\item There is a unique $r$-spin structure $\Sigma(\underline{s})$ extending to the vertices
		if and only if the edge index assignment is admissible.
		Extend the $r$-spin structure to the vertices.
	\item The $r$-spin boundary parametrisation map is the 
		inclusion of the $r$-spin collars according to the map $\tilde{\lambda}$.
		The inclusions map 
	$1\in\Cb^{\tilde{\lambda}}$ 
		to the boundary vertex 
		determined by the extra marking of the given boundary component.
\end{itemize}

\subsection{Distinguishing in- and outgoing boundary components}\label{app:inout}

The glueing of $r$-spin surfaces with parametrised boundary 
is defined as follows. First for every $\kappa\in\Zb_r$ we
specify an $r$-spin lift 
$I_{\eps}^{\kappa}:\Cb^{\kappa}\to\Cb^{2-\kappa}$ 
($\tilde{s}^{\eps}$ in \cite[Eqn.\,(3.35)]{Novak:2015phd}) of the map $z\mapsto z^{-1}$ 
given by an element $\eps\in\Zb_r$.
Take two boundary components with $r$-spin structure on a neighbourhood of these components
$\Cb^{\kappa}$ and $\Cb^{2-\kappa}$.
We can glue these boundary components 
along their $r$-spin boundary parametrisation composed with $I_{\eps}^{\kappa}$.

To define outgoing boundary components we precompose the above boundary parametrisations 
with $I_{\eps}^{\kappa}$ 
for outgoing boundary components. 
(For convenience we will choose $\eps=0$,
as different choices of $\eps$ can be seen as composition with different $r$-spin cylinders.)
Then one can glue $r$-spin boundary components along in- and outgoing boundary parametrisations
as described in Section~\ref{sec:rspinbord}. We now give more details on the construction.

Let $\Sigma$ be an $r$-spin surface with ingoing $r$-spin boundary parametrisation 
	\begin{align*}
	\tilde{\varphi}:\bigsqcup_{b\in\pi_0(\partial\Sigma)}U_b^{\tilde{\lambda}_b}\to\Sigma
	\end{align*}
for a map $\tilde{\lambda}:\pi_0(\partial\Sigma)\to\Zb_r$ which maps $b\mapsto\tilde{\lambda}_b$.
In order to distinguish in- and outgoing boundary components
we first fix two sets $B_{in},B_{out}\subset\pi_0(\partial\Sigma)$
as in Section~\ref{sec:rspinbord}.
Let $\hat{I}:\Zb_r\to\Zb_r$ 
be the map $x\mapsto 2-x$. We define maps
$\lambda:B_{in}\to\Zb_r$ and $\mu:B_{out}\to\Zb_r$ by
$\lambda:=\tilde{\lambda}|_{B_{in}}$ and $\mu:=\hat{I}\circ\tilde{\lambda}|_{B_{out}}$.
For the in- and outgoing $r$-spin boundary parametrisations we set
\begin{align*}
	\varphi_{in}:=\tilde{\varphi}|_{\bigsqcup_{b\in B_{in}}U_b^{\lambda_b}}\ 
		\text{ and }\ 
		\varphi_{out}:=\tilde{\varphi}|_{\bigsqcup_{c\in B_{out}}U_c^{\lambda_c}} \circ\left(\bigsqcup_{c\in B_{out}}I_0^{\lambda_c}\right)
\end{align*}
respectively.
The admissibility condition \eqref{eq:tri:vertexcond2} needs to be changed since we are parametrising
outgoing boundary components $c\in B_{out}$ with $\Cb^{2-\tilde{\lambda}_c}$ instead of with $\Cb^{\tilde{\lambda}_c}$.
This means that for a vertex $u$ on an outgoing boundary component $c$ the factor $\alpha_u$ needs to be 
$-1$ instead of $+1$, since $1-(2-\tilde{\lambda}_c)=-(1-\tilde{\lambda}_c)$.

\subsection{Refining PLCW decompositions of $r$-spin surfaces}\label{app:refine}

\begin{figure}[tb]
	\centering
	\def\svgwidth{8cm}
	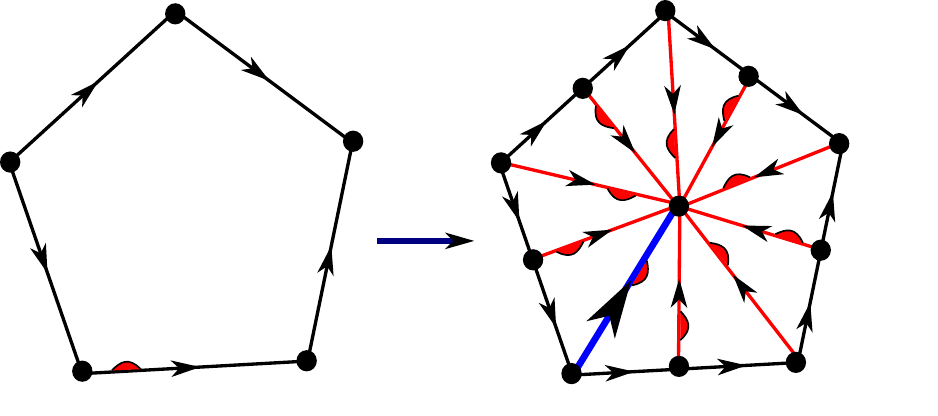
	\caption{New edge indices after a series of radial subdivision.
		The new edge connecting the new vertex in the middle with the vertex 
		which was in the clockwise direction of the marked side of the face 
		(cf. Figure~\ref{fig:clockwisevertex}) has edge index -2, 
		all other new edges inside the face have edge index -1.
		The admissibility conditions \eqref{eq:tri:vertexcond1} and \eqref{eq:tri:vertexcond2} 
		at the vertices remain unchanged at the old vertices 
		and they are satisfied at the new vertices.}
	\label{fig:lem:subdiv}
\end{figure}

By a \textsl{series of radial subdivisions} we mean radially subdividing
the 1-cells and then the 2-cells, see Figure~\ref{fig:lem:subdiv}. 
This means splitting each edge in two by adding a vertex, 
adding a vertex to the interior of each face and adding edges 
between this new vertex and all other vertices of this face.
The following lemmas follow from straightforward calculations.

\begin{lemma} \label{lem:subdiv}
	Let $L$ be a PLCW decomposition obtained by a series of radial 
	subdivisions from a 
	PLCW decomposition $K$
	with admissible marking. 
	Assign to new edges the markings, orientations
	and edge labels as shown in Figure~\ref{fig:lem:subdiv}.
	The vertex conditions \eqref{eq:tri:vertexcond1} and \eqref{eq:tri:vertexcond2}
	are satisfied at the old and
	new vertices.
\end{lemma}

Since we assumed that every boundary component consists of a single vertex and a single edge,
applying two series of radial subdivisions gives four vertices and four edges on each boundary component.
In order to get a triangulation we will modify this refinement as follows.

\begin{figure}[tb]
	\centering
	\def\svgwidth{8cm}
	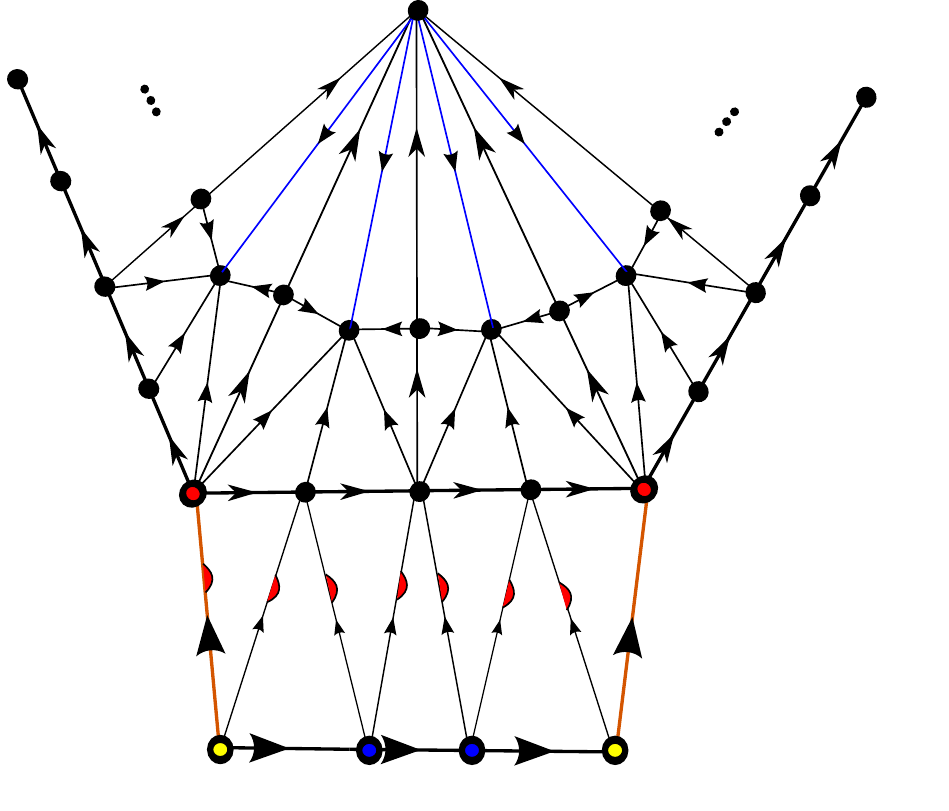
	\caption{Refinement at a boundary component $b$. Edges without labels have edge index -1, 
		the edges between $v_0$ and $v_1$ with edge label $r_b$ are identified.}
	\label{fig:lem:subdivb}
\end{figure}

\begin{lemma} \label{lem:subdivb}
	Let $L$ be a marked PLCW decomposition obtained by applying the steps in
	Lemma~\ref{lem:subdiv} twice on another 
	PLCW decomposition $K$
	with admissible marking. 
	Add 7 triangles at each boundary component and assign the marking
	to the new edges as shown in Figure~\ref{fig:lem:subdivb}
	and put the extra markings on edges on boundary components so that the ending vertex is $v_0$ in Figure~\ref{fig:lem:subdivb}.
	Then the conditions \eqref{eq:tri:vertexcond1} and \eqref{eq:tri:vertexcond2} hold at old and new vertices.
\end{lemma}
\begin{proof}[Sketch of proof]
	Let us assume that at each boundary component 
	there are only two edges connecting to the single vertex: 
	the boundary edge and another one coming from the interior of 
	the surface. In such a situation the refinement is shown in
	Figure~\ref{fig:lem:subdivb}. 
	The conditions \eqref{eq:tri:vertexcond1} and \eqref{eq:tri:vertexcond2}
	can be checked by hand at every vertex.

	If there are boundary components where more edges connect to the boundary vertex from the interior in the original PLCW decomposition,
	checking the conditions \eqref{eq:tri:vertexcond1} and \eqref{eq:tri:vertexcond2} is similar, but we omit the figure here.
\end{proof}

	We now have all the ingredients needed to define 
	an $r$-spin structure with $r$-spin boundary parametrisation
	using the tools developed by \cite{Novak:2015phd}. We proceed as follows.
\begin{itemize}
	\item Take a surface with parametrised boundary
		and a marked PLCW decomposition with some edge indices
		$s$ and maps $\lambda:B_{in}\to\Zb_r$ and $\mu:B_{out}\to\Zb_r$.
	\item Refine this marked PLCW decomposition
		as described in Lemma~\ref{lem:subdivb}.
		This is a triangulation by \cite[Thm.\,6.3]{Kirillov:2012pl}.
\end{itemize}
The new marking obtained this way is admissible in the sense of
\cite{Novak:2015phd} (i.e.\ \eqref{eq:tri:vertexcond1} and \eqref{eq:tri:vertexcond2} hold) 
if and only if the marking of the original PLCW decomposition is admissible
in the sense of Section~\ref{sec:combsub} 
(i.e.\ \eqref{eq:vertexcond1} and \eqref{eq:vertexcond2} hold).

\begin{definition}\label{def:rspintriang}
	Let $\Sigma(s,\lambda,\mu)$ denote the $r$-spin structure on 
	$\Sigma$ obtained by the above steps.
\end{definition}

\subsection{Proofs for Section~\ref{sec:comb}}\label{app:proofs}
\begin{proof}[Proof of Lemma~\ref{lem:moves}]

	\begin{figure}[tb]
		\centering
		\def\svgwidth{16cm}
		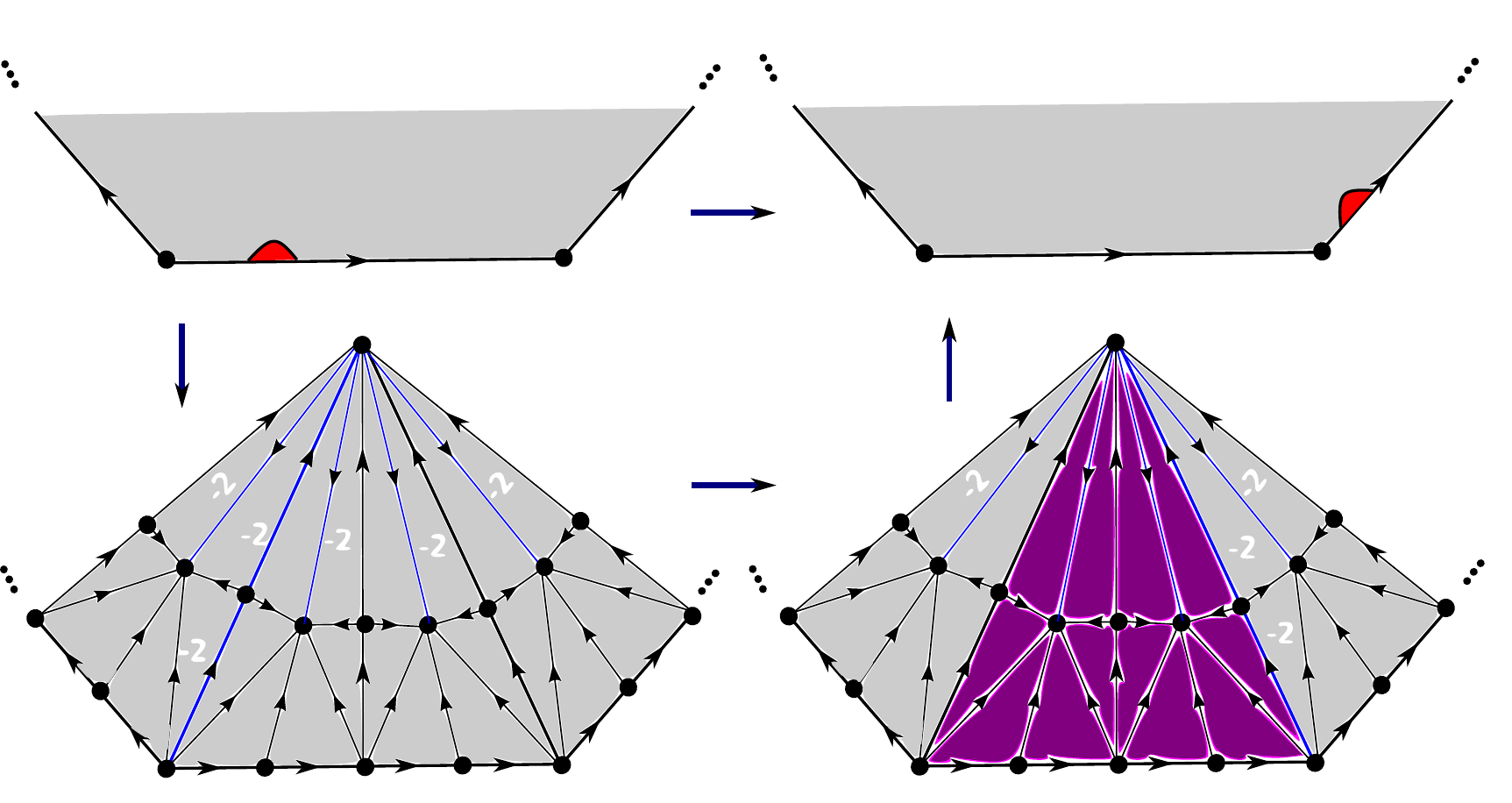
		\caption{Shifting the marking on a face of a PLCW decomposition clockwise. All unlabeled edges have index -1.
			$a)$ Part of a face of a marked PLCW decomposition showing the marked edge.
			$b)$ The corresponding triangulation after two series of radial subdivisions.
			$c)$ Execute a deck transformation on the 12 filled triangles.
			$d)$ The PLCW decomposition with shifted marked edge which produces the triangulation shown in $c)$.}
		\label{fig:move2}
	\end{figure}

	Operation~\ref{lem:moves:1} follows directly from part 2 of \cite[Lem.\,4.11]{Novak:2015phd}.
	
	For Operation~\ref{lem:moves:3} do a deck transformation 
	\cite[Part 1 of Lem.\,4.11]{Novak:2015phd} on all triangles inside the polygon.

	For Operation~\ref{lem:moves:2} first notice that moving the marking of a polygon to the next clockwise edge
	amounts to changing the edge indices as in Figure~\ref{fig:move2}. 
	This is done by a deck transformation on all filled triangles.

	It is a straightforward calculation to show that these operations commute with each other.
\end{proof}

\begin{proof}[Proof of Theorem~\ref{thm:rscomb}]
Let $\Sigma$ be a surface with PLCW decomposition. Let $\Sigma'$ the same surface, but now with a triangulation as obtained by a two-fold series of radial subdivisions as in Section~\ref{app:refine}. For clarity, in this proof we will write $\underline\Sigma$ for the surface without decomposition underlying both $\Sigma$ and $\Sigma'$.

	In \cite[Sec.\,4.8]{Novak:2015phd} $\Mc(\Sigma')_{\tilde{\lambda}}^{\mathrm{triang}}$
	the set of admissible markings for a fixed triangulation of $\underline\Sigma$ with only ingoing boundary components
	and fixed map $\tilde{\lambda}$ has been defined
	along with a similar equivalence relation as $\sim_{fix}$, which we denote by $\sim_{\mathrm{fix}}^{\mathrm{triang}}$.
	\cite[Thm.\,4.18]{Novak:2015phd} gives the isomorphism from 
	the quotient of this set by $\sim_{\mathrm{fix}}^{\mathrm{triang}}$
	to $\Rc^r(\underline\Sigma)_{\tilde{\lambda}}$ the isomorphism classes of $r$-spin structures.
	By a simple reparametrisation 
	as in Section~\ref{app:inout}
	one obtains from this the
	set of admissible markings for in- and outgoing boundary components $\Mc(\Sigma')_{\lambda,\mu}^{\mathrm{triang}}$ and
	the set of isomorphism classes of 
	$r$-spin structures with in- and outgoing boundary components $\Rc^r(\underline\Sigma)_{\lambda,\mu}$. Thus we get a bijection
\be\label{eq:discrete-to-geom-part1}
	\Mc(\Sigma')_{\lambda,\mu}^{\mathrm{triang}} / \sim_{\mathrm{fix}}^{\mathrm{triang}}~
	\xrightarrow{~~f~~} ~\Rc^r(\underline\Sigma)_{\lambda,\mu} \ .
\ee

	Let us denote by $\alpha:\Mc(\Sigma)_{\lambda,\mu}^{\mathrm{PLCW}}\to\Mc(\Sigma')_{\lambda,\mu}^{\mathrm{triang}}$ 
	the map that sends a marked PLCW decomposition to its refinement according to Section~\ref{app:refine}.
	Since the generators of the equivalence relation $\sim_{\mathrm{fix}}$ are built up from
	generators of the equivalence relation $\sim_{\mathrm{fix}}^{\mathrm{triang}}$
	(see the proof of Lemma~\ref{lem:moves} above),
	we get a well defined map 
\be\label{eq:discrete-to-geom-part2}
	\Mc(\Sigma)_{\lambda,\mu}^{\mathrm{PLCW}}/\sim_{\mathrm{fix}}~
	\xrightarrow{~~\bar{\alpha}~~} ~
	\Mc(\Sigma')_{\lambda,\mu}^{\mathrm{triang}}/\sim_{\mathrm{fix}}^{\mathrm{triang}}
	\ .
\ee
By construction the composition of the maps \eqref{eq:discrete-to-geom-part1} and \eqref{eq:discrete-to-geom-part2} is the map \eqref{eq:rscomb} in the statement of the theorem. It therefore remains to show that $\bar\alpha$ is a bijection.
	
\medskip
	\begin{figure}[tb]
		\centering
		\def\svgwidth{6cm}
		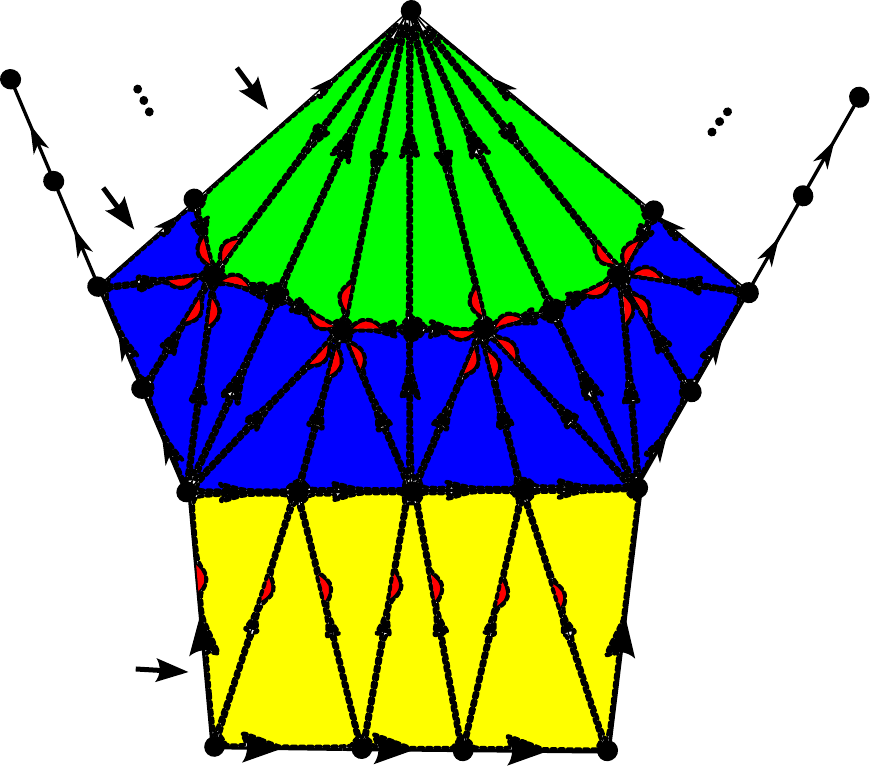
		\caption{
To convert the edge indices of all edges of the triangulation in the interior of some face of the PLCW decomposition to the form shown in Figure~\ref{fig:lem:subdivb} apply the following algorithm to all faces: I) Pick a triangle in area I; proceeding clockwise around the vertex, use deck transformations on each triangle to bring the edge index of each edge radiating from the central vertex to the prescribed value ($-1$ or $-2$); 
note that the final edge in this procedure automatically has the correct index due to the admissibility condition around the central vertex. II) Pick a triangle $t$ in region II which shares an edge with region I but whose neighbour $t'$ in anti-clockwise direction of region II does not. Use a deck transformation on $t$ to set the edge index of the edge on the boundary of region I to the value in Figure~\ref{fig:lem:subdivb}; proceed clockwise around region II setting the edge index between two triangles of region II to the correct value; the edge indices between II and I are determined by the admissibility condition (and so automatically as stated in Figure~\ref{fig:lem:subdivb}); finally, the edge between $t'$ and $t$ has the correct value by the admissibility condition around the vertex between region I and II shared by $t$ and $t'$. III) If the face in question has a boundary component, then in region III one proceeds in the same way as in region II.
			}
		\label{fig:rsubdiv_bdry_pf}
	\end{figure}

\noindent
\textsl{$\bar\alpha$ is surjective:}
	Let $(m',o',s')$ be an admissible marking of $\Sigma'$.
	As a first step, use the relation $\sim_{\mathrm{fix}}^{\mathrm{triang}}$ to change
the edge markings $m'$ and orientations $o'$ to the form prescribed in 
	Section~\ref{app:refine}, resulting in a marking $(m'',o'',s'')$.
	Next follow the algorithm described in Figure~\ref{fig:rsubdiv_bdry_pf} to bring all edge indices of $\Sigma'$ in the interior of faces of $\Sigma$ to the form shown in Figure~\ref{fig:lem:subdivb}. Denote the resulting marking by $(m'',o'',\tilde s)$.
Let $e$ be an interior edge of $\Sigma$ and let $e_1,\dots,e_4$ be the edges of $\Sigma'$ which cover $e$, and $v_{12}$, $v_{23}$, $v_{34}$ the three additional vertices on $e$. The admissibility condition around $v_{12}$, $v_{23}$, $v_{34}$ implies that the edge indices on $e_1,\dots,e_4$ must all be equal. The same argument shows that edge indices on boundary components are all equal. This shows that $(m'',o'',\tilde s)$ lies in the image of $\alpha$.
	
\medskip

\noindent
\textsl{$\bar\alpha$ is injective:}
	Let $(m,o,s),(m',o',s')\in\Mc(\Sigma)_{\lambda,\mu}^{\mathrm{PLCW}}$ such that
	$\bar{\alpha}[(m,o,s)]=\bar{\alpha}[(m',o',s')]$, 
	i.e.\ $\alpha(m,o,s)\sim_{\mathrm{fix}}^{\mathrm{triang}} \alpha(m',o',s')$.
	Notice that Lemma~\ref{lem:fixed-mo-relation} and Remark~\ref{rem:rscomb} apply
	to marked triangulations as well. This means that we can assume 
	that the marked edges and the edge orientations agree ($m=m'$ and $o=o'$) 
	for the PLCW decomposition and the triangulation as well.
	Furthermore, $\alpha(m,o,s)$ and $\alpha(m,o,s')$ are related by 
	a series $D$ of deck transformation on the triangulation: $D( \alpha(m,o,s)) = \alpha(m,o,s')$.
	
	Write $\delta_\Delta(k)$ for a deck transformation by $k$ units on the triangle $\Delta$ of the triangulation of $\Sigma'$. 
	Deck transformations on different triangles commute, so we can write the sequence of deck transformations as $D = \prod_\Delta \delta_\Delta(k_\Delta)$.
	It is not hard to see 
	that the identity $D( \alpha(m,o,s) ) = \alpha(m,o,s')$ requires the $k_\Delta$ for all $\Delta$ belonging to a given face of the PLCW-decomposition of $\Sigma$ to be equal. But this precisely means that $D$ can be written as a product of deck transformations on the PLCW-decomposition of $\Sigma$, i.e.\ $(m,o,s) \sim_{\mathrm{fix}} (m',o',s')$.
\end{proof}

	In the following we are going to give some tools that relate different
	marked triangulations and marked PLCW decompositions which parametrise
	isomorphic $r$-spin structures.
First we recall \cite[Prop.\,4.19 and 4.20]{Novak:2015phd}.

\begin{lemma}
	Let $\Sigma$ and $\Sigma'$ be two $r$-spin surfaces with triangulation 
	and with the same underlying surface related by
	a Pachner 3-1 or 2-2 move as in Figure~\ref{fig:pachner31}~and~\ref{fig:pachner22}. 
	Then these two $r$-spin structures are isomorphic.
	\label{lem:pachner}
	\begin{figure}[tb]
		\centering
		\def\svgwidth{8cm}
		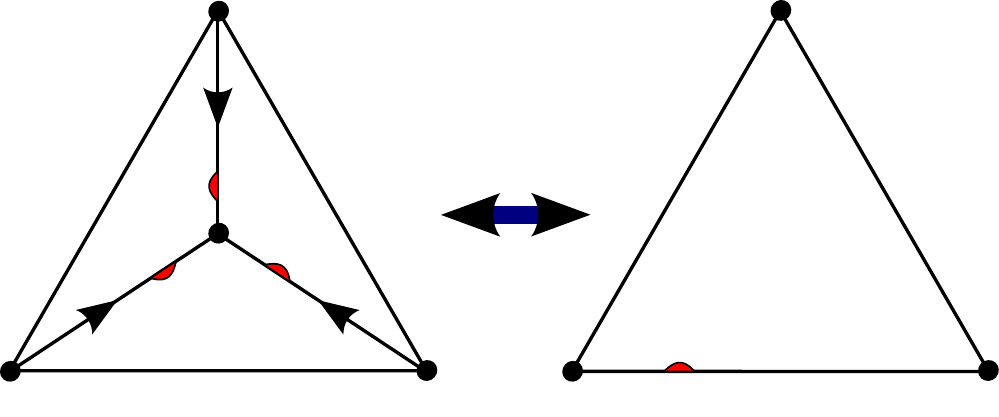
		\caption{Pachner 3-1 move}
		\label{fig:pachner31}
	\end{figure}
	\begin{figure}[tb]
		\centering
		\def\svgwidth{8cm}
		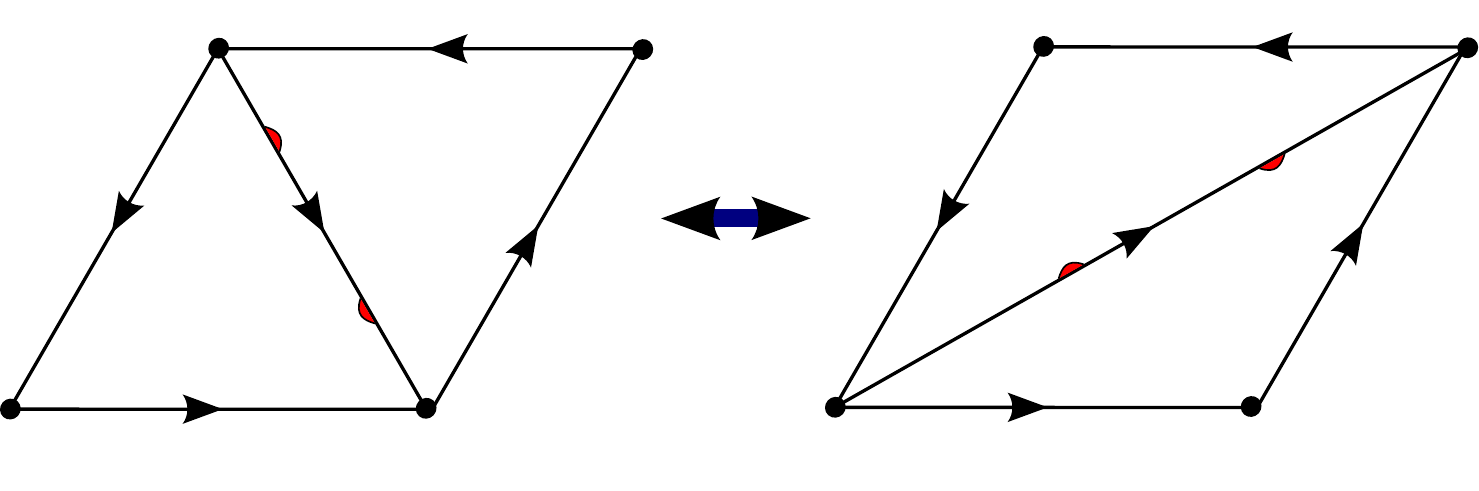
		\caption{Pachner 2-2 move}
		\label{fig:pachner22}
	\end{figure}
\end{lemma}

We define the \textsl{$T_n$-moves} for $n\ge2$ as in Figure~\ref{fig:tnmove},
which takes a $2n$-gon glued together from $2n$ triangles to a $2n$-gon glued together
from $2(n-1)$ triangles.
	\begin{figure}[tb]
		\centering
		\def\svgwidth{12cm}
		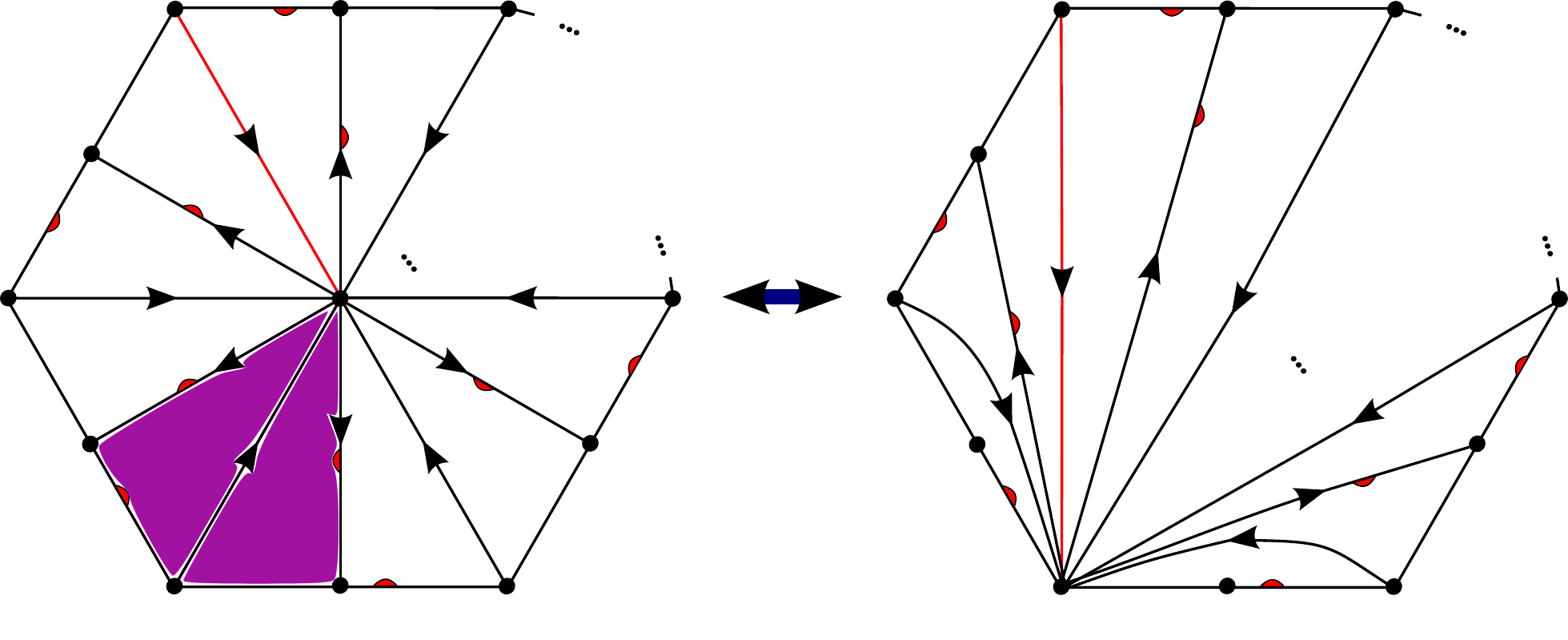
		\caption{$T_n$ move for $n\ge2$. 
			We remove or add the filled triangles.}
		\label{fig:tnmove}
	\end{figure}
\begin{lemma}
	The $T_n$ move induces an isomorphism of $r$-spin structures.
	\label{lem:tnmove}
\end{lemma}
\begin{proof}
	First we show that one can obtain the $T_n$ move on a triangulation
	without any marking by a series of Pachner moves by induction on $n$.
	For $n=2$ do a Pachner 2-2 move and then a Pachner 3-1 move as in Figure~\ref{fig:t2move}.
	Now assume that the statement holds for $n$ and show for $n+1$.
	First we do two Pachner 2-2 moves and then apply a $T_n$ move as in Figure~\ref{fig:tnmove_ind}
	to get exactly the $T_{n+1}$ move.
	\begin{figure}[tb]
		\centering
		\def\svgwidth{14cm}
		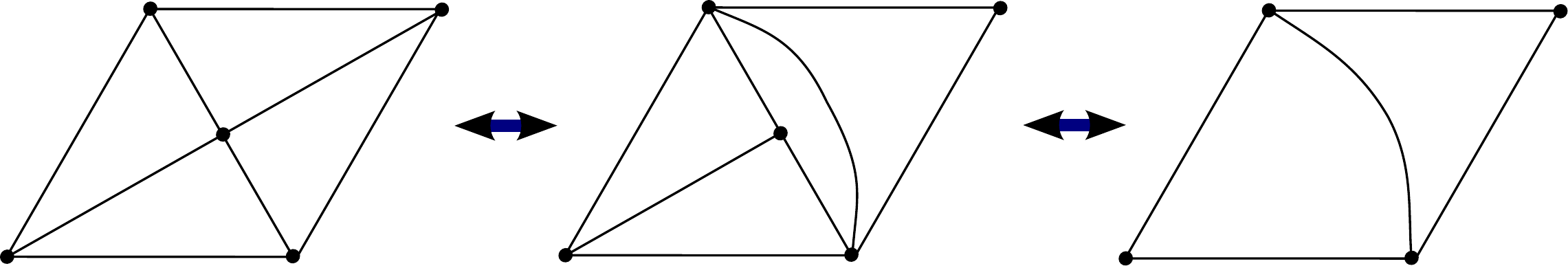
		\caption{$T_2$ move without marking}
		\label{fig:t2move}
	\end{figure}
	\begin{figure}[tb]
		\centering
		\def\svgwidth{16cm}
		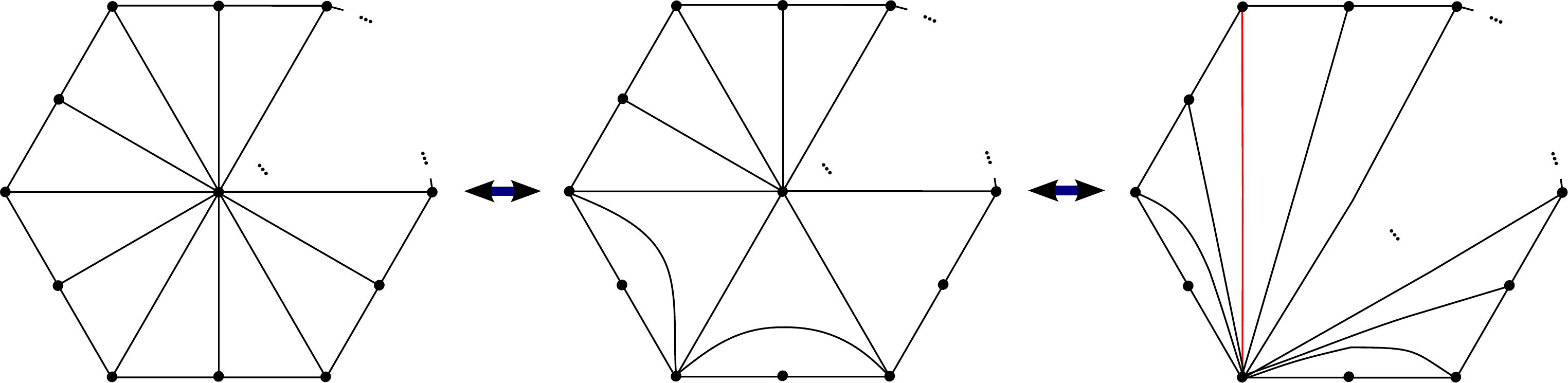
		\caption{Induction step for the $T_{n+1}$ move}
		\label{fig:tnmove_ind}
	\end{figure}

	Since the Pachner moves in Figures~\ref{fig:pachner31} and \ref{fig:pachner22} only change the marking locally,
	it is enough to check how the marking can possibly change near the vertices
	that are touched by these moves. 
	If one calculates \eqref{eq:vertexcond1} for these vertices
	before and after a $T_n$ move one sees that the marking can 
	only change according to Figure~\ref{fig:tnmove}.
\end{proof}

	\begin{figure}[tb]
		\centering
		\def\svgwidth{16cm}
		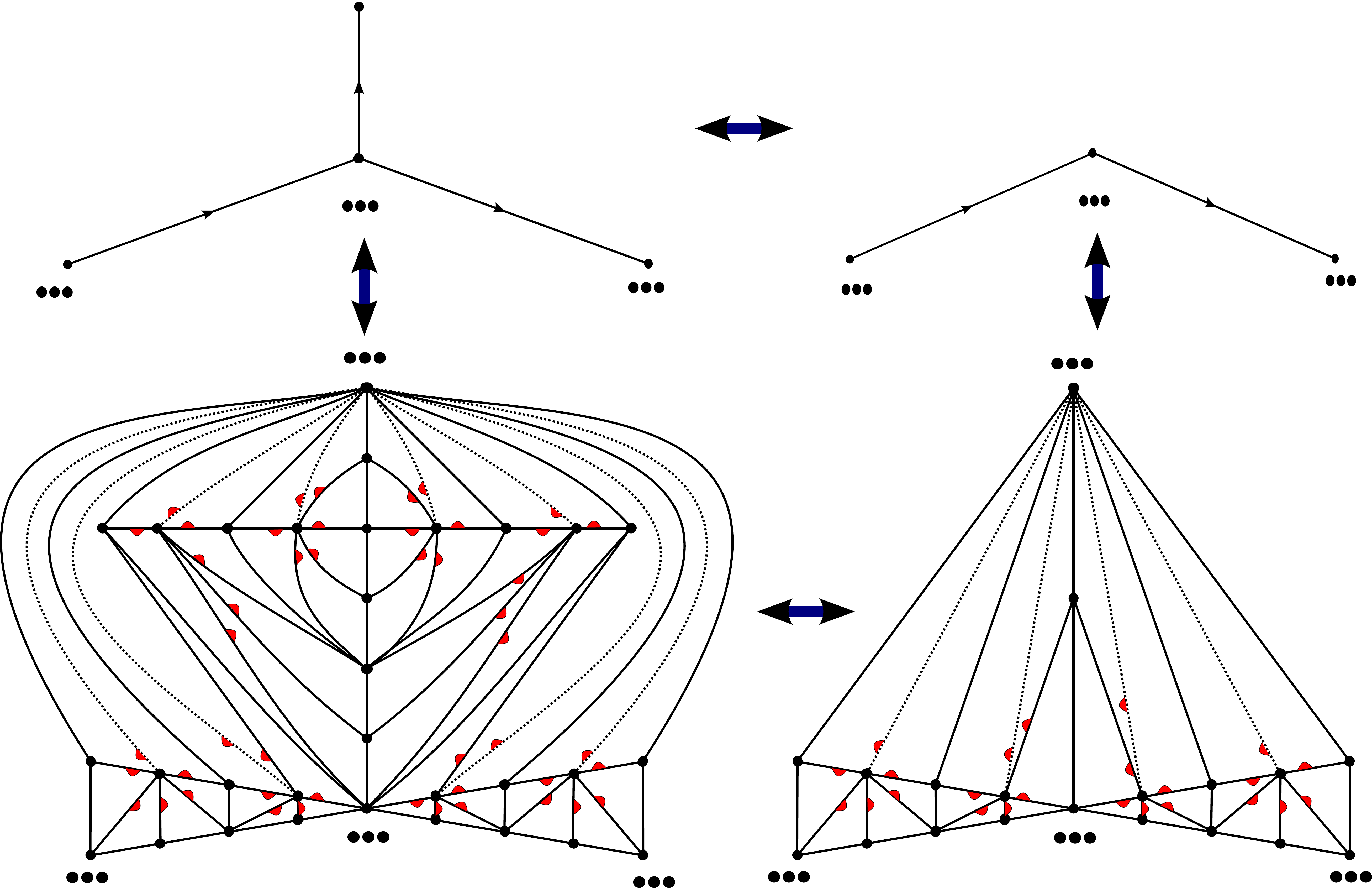
		\caption{The part of the triangulations 
			that need to be transformed into one another
			in case of removing or adding an univalent vertex $v$ with its edge. 
			The dotted edges have edge index -2, all other unlabeled edges have edge index -1.
			The orientation of the edges is left implicit, cf. Definition~\ref{def:rspintriang}.
			We need to remove the 24 numbered triangles from the middle, we proceed by removing them in pairs.
			We use the $T_n$ moves consecutively: first remove the two triangles marked by 1,
			then the two triangles marked by 2, etc until finally removing the two triangles marked by 12.
		}
		\label{fig:univalent_triang}
	\end{figure}

\begin{lemma}
	Removing a univalent vertex (whose edge was not marked) induces an isomorphism of $r$-spin structures.
	\label{lem:univalent}
\end{lemma}
\begin{proof}
	When we remove an edge from a PLCW decomposition we need to compare the 
	associated triangulation with marking from Definition~\ref{def:rspintriang} 
	and then use the above defined moves to go from one to the other. 
	The part of the triangulations that need to be transformed 
	into one another together with the transformation steps
	are shown in Figure~\ref{fig:univalent_triang}.
\end{proof}

\begin{proof}[Proof of Proposition~\ref{prop:elementarymoves}]~
\\
\textsl{Move $b)$ in Figure~\ref{fig:edgemove} for $v \neq v'$:}
		As in the proof of Lemma~\ref{lem:univalent} we need to compare the marked triangulations associated to the marked PLCW decompositions.
	The part of the triangulations that need to be transformed 
	into one another 
		together with the transformation steps
	are shown in Figure~\ref{fig:edgemove_triang}.

	\begin{figure}[tb]
		\centering
		\def\svgwidth{16cm}
		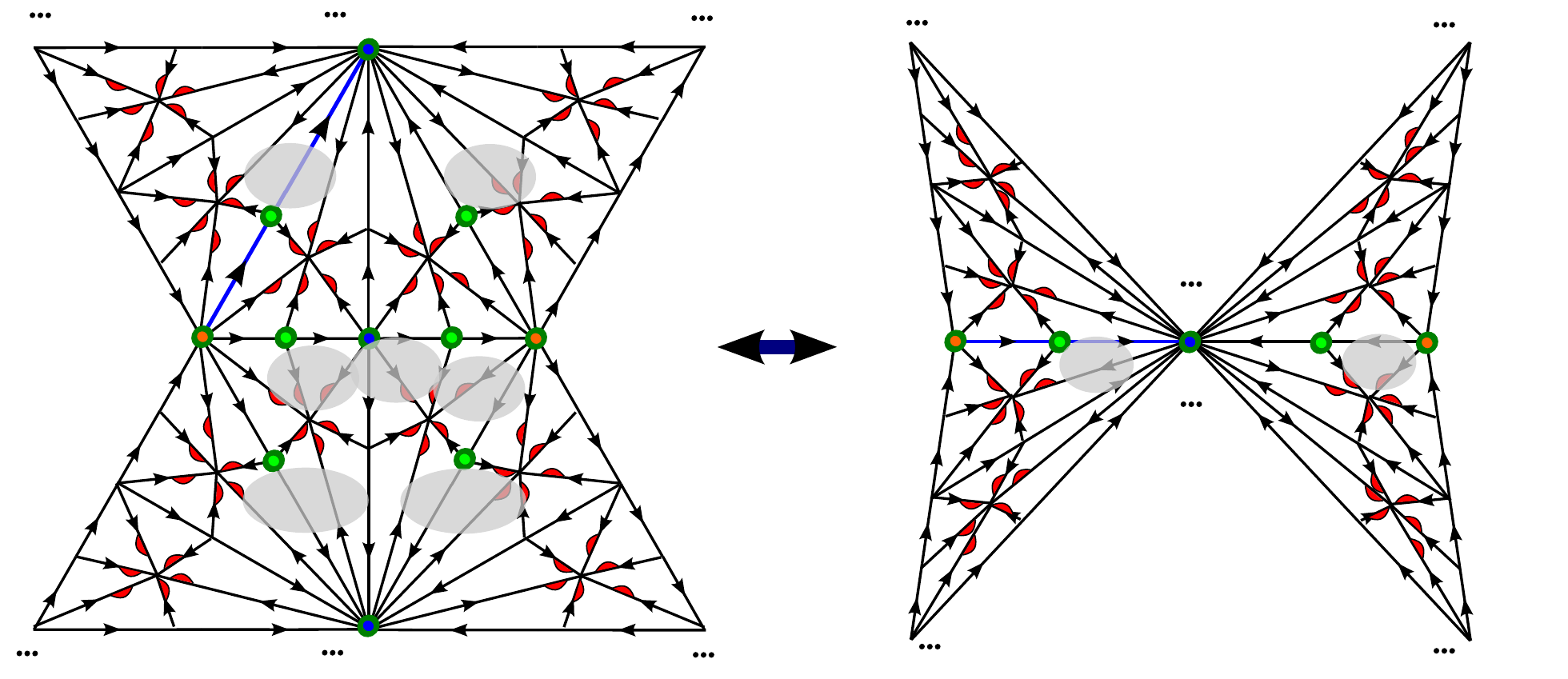
		\caption{The part of the triangulations 
			that need to be transformed into one another
			in case of removing or adding an edge 
			between the vertices $v$ and $v'$
				(cf.\ Figure~\ref{fig:edgemove}~$b)$).
				The edges between $v$, $v_l^{\mathrm{up}}$ and $v_m^{\mathrm{up}}$ have edge index -2,
				all other edges without edge index have edge index -1.
			We need to remove the 24 triangles from the middle, 
			of which 12 has been numbered in pairs.
			We use the $T_n$ moves consecutively: first remove the two triangles marked by 1,
			then the two triangles marked by 2, etc until finally removing the two triangles marked by 6.
			Then do the same thing again for the mirror pairs. 
		}
		\label{fig:edgemove_triang}
	\end{figure}

	Since we did local moves which induce
	isomorphisms of $r$-spin structures, it is enough to check how the edge indices will change
	at those vertices which have been touched by the above moves.
	These vertices are marked with a circle.
	Observe that at the vertices 
	$v_l$, $v_m$ and $v_r$
	one does not get any condition on $s$.
		The vertices $v_l^{\mathrm{up}}$, $v_l^{\mathrm{down}}$,
		$v_r^{\mathrm{up}}$ and $v_r^{\mathrm{down}}$ get identified with others.

	Assume that the vertices $v$ and $v'$ are distinct and that $s_i'=s_i$ ($i=1,\dots,4$). 
	At these two vertices one obtains 
		$s\equiv0\Mod{r}$.

\smallskip

\noindent
\textsl{Move $a)$ in Figure~\ref{fig:edgemove}:}
	When removing a bivalent vertex as in Figure~\ref{fig:edgemove}~$a)$, a similar argument applies.

\smallskip

\noindent
\textsl{Move $b)$ in Figure~\ref{fig:edgemove} for $v = v'$:}
	Indeed, look at the original PLCW decomposition and
	assume that the vertices $v$ and $v'$ are the same. Insert a bivalent vertex on the edge,
	remove one of the two new edges by the above and then the univalent vertex with its edge using 
	Lemma~\ref{lem:univalent}.  Again, one obtains
		$s\equiv0\Mod{r}$.
	
This completes the proof of the proposition.
\end{proof}

\subsection{Proof of Theorem~\ref{thm:tft}}\label{app:pf:thm:tft}

For Part~1 a
direct computation shows that 
the morphism assigned to a PLCW decomposition and
the morphism assigned to the triangulation obtained by 
the refinement of the PLCW decomposition are the same.
One needs to use that multiplication with the $\tau^{-1}$'s
in the state-sum construction amount to canceling 
the ``bubbles'' $\mu\circ\Delta$.
Independence of the choice of the function $V$ follows from
the fact that $\tau$ is a central element.

	Next we check independence from the triangulation and from the choice of marking (for a given $r$-spin structure).
Let us assume that $\Sigma$ has $b$ ingoing and no outgoing boundary components.
Let $T_A(\Sigma)$ denote the morphism in $\Sc$ assigned to $\Sigma$ using 
a triangulation by the state sum construction of \cite{Novak:2015phd}.
	Note that we get three tensor factors of $A$ for each boundary component, since each boundary component consists of three edges.
	Now we explain how to reduce $A^{\otimes3}$ to $A$ for each boundary component.
	Recall that we used the notation $(13)$ for the cyclic permutation of the first and third tensor factors.
Composing $T_A(\Sigma)$ with
$\bigotimes_{i=1}^b(13)\circ(\Delta\otimes\id_A)\circ\Delta\circ(\tau^{-2}\cdot(-))\circ\iota_{\lambda_i}$,
we obtain the morphism in \eqref{eq:step6b}. 
	To show this we use that the factors of $\tau^{-1}$ remove the ``bubbles'' $\mu\circ\tau$.
If $\Sigma$ has outgoing boundary components, it is easy to see that
composing with appropriate factors of $\Gamma_{i,j,\eps}$ maps of \cite[Sec.\,5.4]{Novak:2015phd} and $\pi_{\lambda_i}\circ\mu^{(3)}\circ (13)$ again yields the morphism in \eqref{eq:step6b}.
Independence of the details of the triangulation is shown in
\cite[Thm.\,5.10]{Novak:2015phd}.
This latter theorem also states that $T_A(\Sigma)=T_A(\Sigma')$ for 
isomorphic $r$-spin surfaces $\Sigma$ and $\Sigma'$, so that the
assignment 
	$\funZ_A:\Bord{r}\to\Sc$ 
is well defined on morphisms.

For Part~2	functoriality can now be seen easily from
the above discussion and by using
\cite[Prop.\,5.11]{Novak:2015phd},
since the embeddings and projectors $\iota_{\lambda_i}$ and $\pi_{\lambda_i}$ 
compose to $P_{\lambda_i}$,
which can be omitted due to \cite[Prop.\,5.13]{Novak:2015phd}.
Monoidality and symmetry follow directly 
from the construction.
This completes the proof of Theorem~\ref{thm:tft}.

\subsection{Proof of Proposition~\ref{prop:indexhol}}\label{app:pf:lem:indexhol}

	Part~\ref{lem:indexhol1} does not involve 
	the marked PLCW decomposition and is shown in
	 \cite[Lem.\,3.12]{Novak:2015phd}.

	Part~\ref{lem:indexhol3} follows directly from the discussion in the main text:
	homotopic curves in $\Sigma$ 
	(in the sense described in the beginning of Section~\ref{sec:arfinv})	
have homotopic lifts in $F_{GL}\Sigma$ and homotopic
	curves in $F_{GL}\Sigma$ have the same lifts in $P_{\widetilde{GL}}\Sigma$
	after fixing them at the same starting point.

	\begin{figure}[tb]
		\centering
		\def\svgwidth{10cm}
		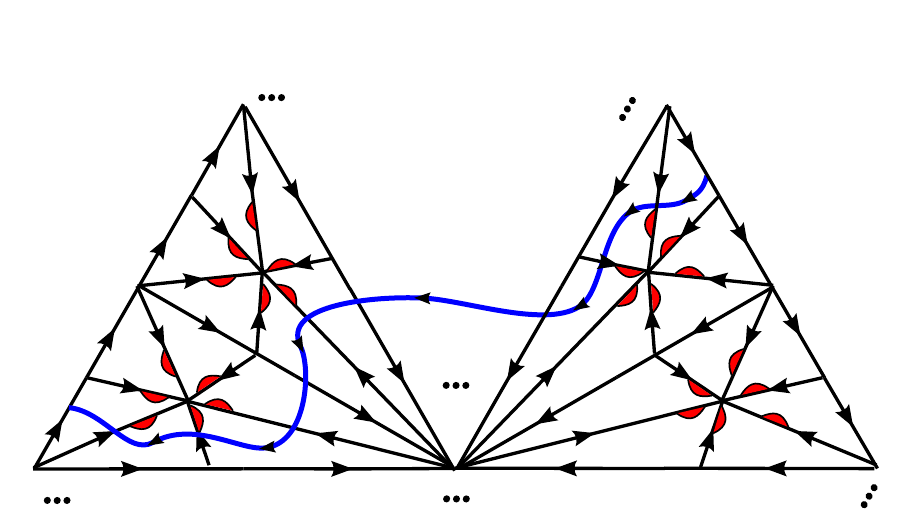
		\caption{Detail of a face with interior edges of a refined PLCW decomposition 
		with the segment $p\in A(\gamma)$ of the curve
		$\gamma$ crossing it. 
		Using Part~\ref{lem:indexhol3}, we can assume that the segment of the curve crosses as shown in the figure.
		All edge indices without edge labels are $-1$.
		Notice that when crossing the dotted area, the lift of the curve does not pick up
		any of the $\omega_e$ contributions.}
		\label{fig:holonomy}
	\end{figure}
	For Part~\ref{lem:indexhol2}, we are going to calculate the holonomy by summing up the contributions
	for all arcs $A(\gamma)$ as in \eqref{eq:holonomy}.

	\begin{figure}[tb]
		\centering
		\def\svgwidth{6cm}
		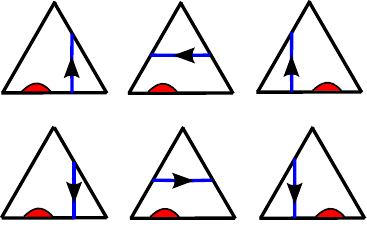
		\caption{The different values of $\kappa_e$ for different positions of the crossing curve segment.
		The edge $e$ is where the line segment leaves the triangle.}
		\label{fig:kappa}
	\end{figure}

	The contribution for $p\in A(\gamma)$ can be computed as follows. 
	Take the face $f_p$ which $p$ crosses and take its refinement to a triangulation as in Section~\ref{app:refine}.
	Let us first assume that this face has only inner edges, as in Figure~\ref{fig:holonomy}.
	Let $e_{f_p}$ be the edge where $p$ leaves the face $f_p$.
	The contribution of $p$ can now be calculated by summing up for each triangle
	the ``$\omega_{e}$'' contributions of \cite[Section\,4.7]{Novak:2015phd}. 
	For a given triangle $t$ and edge $e$, where the curve leaves $t$, 
	the contribution is $\omega_{e}=\hat{s}_{e}+\kappa_{e}$
	by \cite[(4.33)]{Novak:2015phd},
	where $\kappa_{e}$ is given in Figure~\ref{fig:kappa}.

	First the curve crosses 3 triangles, which give a contribution of
	\begin{align*}
		(0+0)+(0+0)+(0+0)=0\ .
	\end{align*}
	Notice that when afterwards  crossing the dotted area, the lift of the curve does not pick up
	any of contributions: for every group of 4 triangles the contribution is
	\begin{align*}
		(-2+1)+(0+1)+(-2+1)+(0+1)=0\ .
	\end{align*}
	If the marked edge of the face $f_p$ is on the right side of $p$
	with respect to the orientation of $f_p$ then the curve has crossed
	the corresponding edge with edge label -2 and the contribution is
	\begin{align*}
		\hat{\delta}_{f_p}^p=1.
	\end{align*}
	Finally the curve crosses 6 triangles, which give a contribution of
	\begin{align*}
		(-2+1)+(0+1)+(-2+1)+(-1+1)+(-1+1)+(\hat{s}_{e_p}+1)=\hat s_{e_p}\ .
	\end{align*}
	This proves the formula \eqref{eq:holonomy} if $\gamma$ is a closed curve.

	\begin{figure}[tb]
		\centering
		\def\svgwidth{10cm}
		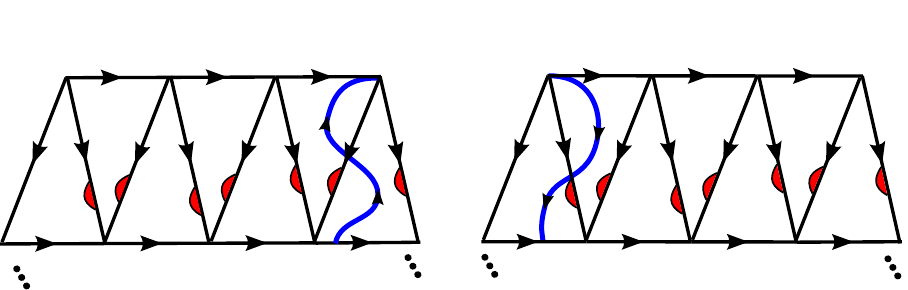
		\caption{Detail of two (not necessarily different) 
		faces with two boundary edges of a refined PLCW decomposition 
		where the curve $\gamma$ starts $(b)$ and ends $(a)$, 
		i.e.\ at the image of $1\in\Cb^{\times}$ under the boundary parametrisation.
		All edge indices are $-1$ unless otherwise noted.}
		\label{fig:holonomy_bdry}
	\end{figure}

	If the curve $\gamma$ starts and ends on the boundary of the surface then we take it into account as follows.
	The parts of the triangulation where $\gamma$ starts and ends is shown in Figure~\ref{fig:holonomy_bdry}.
	\begin{figure}[tb]
		\centering
		\def\svgwidth{4cm}
		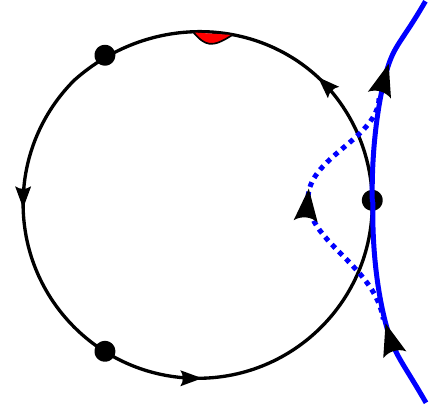
		\caption{Detail of $D^{0}$ with the image of the identification of the neighbourhoods of
		the starting and ending point of $\gamma$.
		The circle denotes the two boundary components mapped onto each other.
		We obtain a closed curve which, by using Part~\ref{lem:indexhol3}, 
		we are allowed to change by a homotopy to the dotted curve.
		This allows us to compute the holonomy in terms of the ``$\omega_e$'' contributions as before.
		}
		\label{fig:closing_curve_bdry}
	\end{figure}
	As described in the main text we have $r$-spin isomorphisms $D^{\kappa}\to D^0$ of some neighbourhoods
	of the starting and ending point of $\gamma$, both sending these two points to $1\in D^0\subset \Cb^0$.
	Under these isomorphisms the neighbourhood of $1$, together with a part of $\gamma$ and the boundary edges 
	is shown in Figure~\ref{fig:closing_curve_bdry}. This way we can handle $\gamma$ as a closed curve
	and by Part~\ref{lem:indexhol2} we can modify the curve by a homotopy as in Figure~\ref{fig:closing_curve_bdry},
	so that it crosses the edges $e_{\mathrm{start}}$ and $e_{\mathrm{end}}$.

	We can now calculate the contribution of these crossed triangles as before.
	The curve first crosses the boundary triangle in Figure~\ref{fig:closing_curve_bdry} picking up the contribution 
	\begin{align*}
		\hat{s}_{e_{\mathrm{start}}}+1\ .
	\end{align*}
	Then it crosses the two triangles in Figure~\ref{fig:holonomy_bdry}~$b)$ picking up the contribution
	\begin{align*}
		(0+0)+(-1+0)\ .
	\end{align*}
	After crossing inner edges finally it crosses the two triangles in 
	Figure~\ref{fig:holonomy_bdry}~$a)$, using Figure~\ref{fig:closing_curve_bdry},
	picking up the contribution
	\begin{align*}
		(0+1)+(0+\hat{s}_{e_{\mathrm{end}}})\ .
	\end{align*}
	Summing up the above contributions, we get formula \eqref{eq:holonomy}.

	This completes the proof of Proposition~\ref{prop:indexhol}.

\phantomsection


\addcontentsline{toc}{section}{References}

\begin{thebibliography}{{Mun}}

    \small
  \itemsep -3pt

\bibitem[BT]{Barrett:2013sp}
 J.W. Barrett and S.O.G. Tavares, {\em {Two-dimensional state sum models and
  spin structures}}.
\newblock \href{http://dx.doi.org/10.1007/s00220-014-2246-z}{Comm. Math. Phys.
  {\bfseries 336} (2015) 63--100}
  \href{http://arxiv.org/abs/1312.7561}{{\ttfamily [1312.7561 [math.QA]]}}.

\bibitem[Da]{Davydov:2010fc}
 A.~{Davydov}, {\em {Centre of an algebra}}.
\newblock \href{http://dx.doi.org/10.1016/j.aim.2010.02.018}{Adv. Math.
  {\bfseries 225} (2010) 319--348}
  \href{http://arxiv.org/abs/0908.1250}{{\ttfamily [0908.1250 [math.CT]]}}.

\bibitem[DK]{Dyckerhoff:2015csg}
 T.~Dyckerhoff and M.~Kapranov. {\em {Crossed simplicial groups and structured
  surfaces}}. \href{http://dx.doi.org/10.1090/conm/643}{In T.~Pantev,
  C.~Simpson, B.~To{\"e}n, M.~Vaqui{\'e}, and G.~Vezzosi, editors, {\em {Stacks
  and Categories in Geometry, Topology, and Algebra}}},
  \href{http://dx.doi.org/10.1090/conm/643}{volume 643},
  \href{http://dx.doi.org/10.1090/conm/643}{AMS}
  \href{http://dx.doi.org/10.1090/conm/643}{(2015)}
  \href{http://arxiv.org/abs/1403.5799}{{\ttfamily [1403.5799 [math.AT]]}}.

\bibitem[FRS]{Fuchs:2004r3}
 J.~Fuchs, I.~Runkel, and C.~Schweigert, {\em {TFT construction of RCFT
  correlators: III: simple currents}}.
\newblock \href{http://dx.doi.org/10.1016/j.nuclphysb.2004.05.014}{Nucl. Phys.
  B {\bfseries 694} (2004) 277--353}
  \href{http://arxiv.org/abs/hep-th/0403157}{{\ttfamily [hep-th/0403157]}}.

\bibitem[FS]{Fuchs:2008fa}
 J.~{Fuchs} and C.~{Stigner}, {\em {On Frobenius algebras in rigid monoidal
  categories}}.
\newblock Arab. J. Sci. Eng. {\bfseries 33-2C} (2008) 175--191
  \href{http://arxiv.org/abs/0901.4886}{{\ttfamily [0901.4886 [math.CT]]}}.

\bibitem[GG]{Geiges:2012rs}
 H.~Geiges and J.~Gonzalo, {\em {Generalised spin structures on 2-dimensional
  orbifolds}}.
 \newblock \href{https://doi.org/10.18910/4191}{Osaka J. Math. {\bfseries 49}, 449--470 (2012).}

\bibitem[GK]{Gaiotto:2016spin}
 D.~Gaiotto and A.~Kapustin, {\em {Spin TQFTs and fermionic phases of matter}}.
\newblock \href{http://dx.doi.org/10.1142/S0217751X16450445}{Int. J. Mod. Phys.
  A {\bfseries 31} (2016) 1645044}
  \href{http://arxiv.org/abs/1505.05856}{{\ttfamily [1505.05856
  [cond-mat.str-el]]}}.

\bibitem[Gu]{Gunningham:2016sph}
 S.~Gunningham, {\em {Spin Hurwitz numbers and topological quantum field
  theory}}.
\newblock \href{http://dx.doi.org/10.2140/gt.2016.20.1859}{Geom. Topol.
  {\bfseries 20} (2016) 1859--1907}
  \href{http://arxiv.org/abs/1201.1273}{{\ttfamily [1201.1273 [math.QA]]}}.

\bibitem[HJ]{Humphries:1989wn}
 S. Humphries and D. Johnson, {\em {A generalization of winding number functions on surfaces}}. 
\newblock \href{http://doi.org/10.1112/plms/s3-58.2.366}{Proc. London
	Math. Soc. {\bfseries 58} (1989) 366--386}.

\bibitem[Hu]{Husemoller:fb}
 D.~Husemoller. {\em Fibre Bundles}.
  \href{http://dx.doi.org/10.1007/978-1-4757-2261-1}{Graduate Texts in
  Mathematics}. \href{http://dx.doi.org/10.1007/978-1-4757-2261-1}{Springer},
  \href{http://dx.doi.org/10.1007/978-1-4757-2261-1}{3rd edition}
  \href{http://dx.doi.org/10.1007/978-1-4757-2261-1}{(1994)}.

\bibitem[Ja]{Jarvis:1998}
T.J.~Jarvis, {\em Geometry of the moduli of higher spin curves}.
\newblock \href{http://dx.doi.org/10.1142/S0129167X00000325}{Int.\ J.\ Math.\ {\bfseries 11} (2000) 637--663}
  \href{http://arxiv.org/abs/math/9809138}{{\ttfamily [math/9809138 [math.AG]]}}.

\bibitem[Jo]{Johnson:1980arf}
 D.~Johnson, {\em {Spin structures and quadratic forms on surfaces}}.
\newblock \href{http://dx.doi.org/10.1112/jlms/s2-22.2.365}{J. London Math.
  Soc. {\bfseries 2} (1980) 365--373}.

\bibitem[JS]{Joyal:1993bmc}
 A.~Joyal and R.~Street, {\em Braided tensor categories}.
\newblock \href{http://dx.doi.org/10.1006/aima.1993.1055}{Adv. Math. {\bfseries
  102} (1993) 20--78}.

\bibitem[Ka]{Kawazumi:2017mcg}
 N.~Kawazumi, {\em The mapping class group orbits in the framings of compact surfaces}.
 \newblock \href{http://dx.doi.org/10.1093/qmath/hay024}{Quart. J. Math.
  {\bfseries 69} (2018) 1287--1302}
  \href{http://arxiv.org/abs/1703.02258}{{\ttfamily [1703.02258 [math.GT]]}}.

\bibitem[Ki]{Kirillov:2012pl}
 A.~Kirillov, Jr., {\em On piecewise linear cell decompositions}.
\newblock \href{http://dx.doi.org/10.2140/agt.2012.12.95}{Algebr. Geom. Topol.
  {\bfseries 12} (2012) 95--108}
  \href{http://arxiv.org/abs/1009.4227}{{\ttfamily [1009.4227 [math.GT]]}}.

\bibitem[Ko]{Kock:2004fa}
 J.~Kock. {\em Frobenius Algebras and 2D Topological Quantum Field Theories}.
  \href{http://dx.doi.org/10.1017/CBO9780511615443}{Cambridge University
  Press} \href{http://dx.doi.org/10.1017/CBO9780511615443}{(2004)}.

\bibitem[LP]{Lauda:2007oc}
 A.D.~Lauda and H.~Pfeiffer, {\em State sum construction of two-dimensional
  open-closed topological quantum field theories}.
\newblock \href{http://dx.doi.org/10.1142/S0218216507005725}{J. Knot Theor.
  Ramif. {\bfseries 16} (2007) 1121--1163}
  \href{http://arxiv.org/abs/math/0602047}{{\ttfamily [math/0602047
  [math.QA]]}}.

\bibitem[{Lu}]{Lutz:2005fw}
 F.H.~Lutz, {\em {Triangulated Manifolds with Few Vertices: Combinatorial
  Manifolds}}.
\newblock \href{http://arxiv.org/abs/math/0506372}{{\ttfamily math/0506372
  [math.CO]}}.

\bibitem[MS]{Moore:2006db}
 G.W. Moore and G.~Segal, {\em {D-branes and K-theory in 2D topological field
  theory}}.
\newblock \href{http://arxiv.org/abs/hep-th/0609042}{{\ttfamily
  hep-th/0609042}}.

\bibitem[Mu]{Munkres:1966dt}
 J.R.~Munkres. {\em Elementary Differential Topology},
  \href{http://dx.doi.org/10.1515/9781400882656}{volume~54 of {\em Annals of
  Mathematics Studies}}.
  \href{http://dx.doi.org/10.1515/9781400882656}{Princeton University Press}
  \href{http://dx.doi.org/10.1515/9781400882656}{(1966)}.

\bibitem[No]{Novak:2015phd}
 S.~Novak. {\em {Lattice topological field theories in two dimensions}}. PhD
  thesis, Universit\"at Hamburg, 2015,
  \url{http://ediss.sub.uni-hamburg.de/volltexte/2015/7527}.

\bibitem[NP]{Natanzon:2004harf}
  S.~Natanzon and A.~Pratoussevitch, {\em Topological quantum field theories on spin surfaces}.
\newblock J. Lie Theory. {\bfseries 19} (2009) 107--148
  \href{http://arxiv.org/abs/math/0411375}{{\ttfamily [math/0411375 [math.AG]]}}.

\bibitem[NR]{Novak:2014sp}
 S.~Novak and I.~Runkel, {\em State sum construction of two-dimensional
  topological quantum field theories on spin surfaces}.
\newblock \href{http://dx.doi.org/10.1142/S0218216515500285}{J. Knot Theor.
  Ramif. {\bfseries 24} (2015) 1550028}
  \href{http://arxiv.org/abs/1402.2839}{{\ttfamily [1402.2839 [math.QA]]}}.

\bibitem[Ra]{Randal:2014rs}
 O.~{Randal-Williams}, {\em {Homology of the moduli spaces and mapping class
  groups of framed, r-Spin and Pin surfaces}}.
\newblock \href{http://dx.doi.org/10.1112/jtopol/jtt029}{J. Topol. {\bfseries
  7} (2014) 155--186} \href{http://arxiv.org/abs/1001.5366}{{\ttfamily
  [1001.5366 [math.GT]]}}.

\bibitem[Sa]{Salter:2017mon}
 N. Salter, {\em {Monodromy and vanishing cycles in toric surfaces}}.
\newblock 
\href{http://doi.org/10.1007/s00222-018-0845-6}{Inventiones mathematicae {\bfseries 216} (2019) 153--213}
\href{http://arxiv.org/abs/1710.08042}{{\ttfamily [1710.08042
  [math.AG]]}}.

\bibitem[Sd]{Steenrod:1951top}
 N.E. Steenrod. {\em The Topology of Fibre Bundles},
  \href{http://dx.doi.org/10.1515/9781400883875}{{\em Princeton
  Mathematical Series} \bfseries{14}}.
  \href{http://dx.doi.org/10.1515/9781400883875}{Princeton University Press}
  \href{http://dx.doi.org/10.1515/9781400883875}{(1951)}.

\bibitem[Sn]{Stern:2016stft}
 W.H. Stern, {\em {Structured Topological Field Theories via Crossed Simplicial
  Groups}}.
\newblock \href{http://arxiv.org/abs/1603.02614}{{\ttfamily 1603.02614
  [math.CT]}}.

\bibitem[SnSz]{SternSzegedy}
W.~Stern and L.~Szegedy, 
{\em{Topological field theories on open-closed $r$-spin surfaces}}. 
\newblock \href{https://arxiv.org/abs/2004.14181}{\ttfamily 2004.14181 [math.QA]}.

\bibitem[Sz]{Szegedy:2018phd}
 L.~Szegedy. {\em {}State-sum construction of two-dimensional functorial field theories}. PhD
  thesis, Universit\"at Hamburg, 2018,
  \url{https://ediss.sub.uni-hamburg.de/handle/ediss/7848}.

\bibitem[Wi]{Witten:1993mm}
 E.~Witten. {\em {Algebraic Geometry Associated with Matrix Models of Two
  Dimensional Gravity}}. In L.R. Goldberg and A.V. Phillips, editors, {\em
  Topological Methods in Modern Mathematics}, pages 235--269. Publish or
  Perish, Inc.\ (1993).

\end{thebibliography}
\end{document}